\documentclass[letter, 12pt, onepage]{amsart}
\usepackage[utf8]{inputenc}
\usepackage{enumitem}
\usepackage{mathtools}
\usepackage{pdfsync}
\usepackage{graphicx}
\usepackage{hyperref}
\usepackage{amssymb}
\usepackage{amsmath}
\usepackage{mathtools}
\usepackage{mathrsfs}
\usepackage{tkz-euclide}
\usepackage{tikz-cd}
\usepackage[all]{xy}
\usetkzobj{all}
\usepackage{latexsym}
\usepackage[colorinlistoftodos]{todonotes}
\usepackage[english]{babel}
\usepackage{csquotes}
\usepackage{titletoc}
\usepackage{lipsum}
\usepackage{cancel}
\usepackage{tcolorbox}
\usepackage{appendix}
\usepackage{fancyhdr}

\makeatletter
\newcommand\restrict[1]{\raisebox{-.0ex}{$\upharpoonright$}_{#1}}

\DeclarePairedDelimiterX{\inner}[2]{\langle}{\rangle}{#1, #2}

\DeclareMathOperator{\Ad}{Ad}

\DeclareMathOperator{\Aut}{Aut}

\DeclareMathOperator{\Hom}{Hom}

\DeclareMathOperator{\SU}{SU}

\DeclareMathOperator{\Aff}{Aff}
\DeclareFontFamily{U}{wncy}{}
\DeclareFontShape{U}{wncy}{m}{n}{<->wncyr10}{}
\DeclareSymbolFont{mcy}{U}{wncy}{m}{n}
\DeclareMathSymbol{\Sh}{\mathord}{mcy}{"58}

\newcommand{\SL}{{\rm SL}}
\newcommand{\GL}{{\rm GL}}

\newcommand{\SO}{{\rm SO}}

\renewcommand{\phi}{\varphi}

\newcommand{\Vol}{{\rm Vol}}
\newcommand{\Diff}{{\rm Diff}}

\newcommand{\Bcal}{ {\mathcal B}}

\newcommand{\Fcal}{ {\mathcal F}}
\newcommand{\Hcal}{ {\mathcal H}}
\newcommand{\Mcal}{ {\mathcal M}}

\newcommand{\Ucal}{ {\mathcal U}}

\newcommand{\Acal}{ {\mathcal A}}
\newcommand{\Ccal}{ {\mathcal C}}

\newcommand{\Lcal}{ {\mathcal L}}

\newcommand{\Abb}{ {\mathbb A}}

\newcommand{\Fbb}{ {\mathbb F}}
\newcommand{\Gbb}{ {\mathbb G}}
\newcommand{\Hbb}{ {\mathbb H}}
\newcommand{\Kbb}{ {\mathbb K}}

\newcommand{\Cbb}{\mathbb{C}}
\newcommand{\Qbb}{\mathbb{Q}}
\newcommand{\Tbb}{{\mathbb T}}
\newcommand{\Rbb}{{\mathbb R}}
\newcommand{\Zbb}{{\mathbb Z}}
\newcommand{\Nbb}{{\mathbb N}}

\newcommand{\Lbb}{{\mathbb L}}
\newcommand{\Xbb}{{\mathbb X}}
\newcommand{\Ubb}{{\mathbb U}}

\newcommand{\nlie}{ {\mathfrak n}}
\newcommand{\llie}{ {\mathfrak l}}

\newcommand{\Asc}{\mathscr {A} }
\newcommand{\Bsc}{\mathscr {B} }

\newcommand{\Csc}{\mathscr {C} }

\newcommand{\Esc}{\mathscr {E} }

\newcommand{\Ksc}{\mathscr {K} }

\theoremstyle{definition}
\newtheorem{theorem}{Theorem}[section]

\theoremstyle{definition}
\newtheorem{defn}[theorem]{Definition}

\newtheorem{nota}{\textbf{Assumption}}

\newtheorem{notaa}{\textbf{Assumption on the action}}

\newtheorem{coro}[theorem]{\textbf{Corollary}}

\newtheorem{lem}[theorem]{\textbf{Lemma}}
\newtheorem{prop}[theorem]{\textbf{Propostion}}
\newtheorem*{theorem*}{Theorem}

\newtheorem{THM}{Theorem}

\newcommand{\Sp}{\textrm{Sp}}
\newtheorem{rmk}[theorem]{\textbf{Remark}}

\theoremstyle{remark}

\newtheorem{eg}[theorem]{Example}

\numberwithin{theorem}{section}
\makeindex

\newtheorem{Not*}[theorem]{Standard Notation.}
\begin{document}

\title[Rigidity theorems]{Rigidity theorems\\for higher rank lattice actions}

\author{Homin Lee}
\address{Indiana University Bloomington \newline 831 East 3rd St., Bloomington, IN 47405}
\email{\href{mailto:hl63@indiana.edu}{hl63@indiana.edu}}
\date{\today}
\maketitle
\begin{abstract}
Let $\Gamma$ be a weakly irreducible higher rank lattice. In this paper, we will prove various rigidity results for the $\Gamma$-action following a philosophy of the Zimmer program. We provide new rigidity results including local and global rigidity of the $\Gamma$-action when $\Gamma$ does not have property (T).
 
The new ingredient is a \emph{dynamical cocycle super-rigidity theorem}. It can be thought as the generalization of Zimmer's cocycle super-rigidity theorem since it provides almost the same consequences dynamically and algebraically. This allows us to derive various our rigidity results using dynamical super-rigidity instead of Zimmer's cocycle super-rigidity theorem.

\end{abstract}

\tableofcontents

\section{Introduction}
Let $G$ be the connected real semisimple higher rank algebraic Lie group without compact factors. That is $G=\Gbb(\Rbb)^{0}$ the connected component of the identity of the $\Rbb$-points of a real semisimple algebraic group $\Gbb$ such that $\textrm{rank}_{\Rbb}(\Gbb)\ge 2$ such that the $\Gbb$ does not have $\Rbb$-anistropic almost simple factors, i.e. $G$ does not have compact factors.  By an algebraic group $\Gbb$, we always mean an affine algebraic subgroup of $\GL(d,\Cbb)$ for some $d>0$. Recall that, up to finite index, any lattice $\Gamma$ can be represented as $\prod\Gamma_{i}$, that is a product of irreducible lattices $\Gamma_{i}$ in connected normal subgroups $H_{i}$ in $G$ \cite[5.22 Theorem]{Rag}. We will call that $\Gamma$ be a \emph{weakly irreducible lattice} in $G$ if $H_{i}$ are all higher rank. For example, $\Gamma\simeq\SL_{2}(\Zbb[\sqrt{17}])$ can be embedded as a weakly irreducible lattice in $G=\SL(2,\Rbb)\times \SL(2,\Rbb)$ using Galois conjugation. Note that we do \emph{not} assume that every simple factor $G$ have real rank $2$ or higher. We only require that $\textrm{rank}_{\Rbb}(G)\ge 2$. When all simple factors of $G$ have rank at least $2$, any lattice is weakly irreducible.

In this paper, we study actions of $G$ or its weakly irreducible lattices on compact manifolds via diffeomorphisms. The compact manifolds will always be a smooth Riemannian compact manifold without boundary. We will prove various rigidity results under assumptions such as a dimension condition or the existence of an invariant measure. The main point of the paper is the case when $G$ has rank $1$ factors.

\subsection{Main rigidity results}

Let $G$ be the connected real semisimple algebraic Lie group without compact factors, and  $\textrm{rank}_{\Rbb}(G)\ge 2$. Let $\Gamma$ be a weakly irreducible lattice in $G$. The reader should keep in mind the example $G=\SL(2,\Rbb)\times \SL(2,\Rbb)$ and $\Gamma\simeq \SL(2,\sqrt{17})$.

For any rank $1$ factor $F$ of $G$, if it exists, we may find its complement $F^{c}$ that is the almost direct product of the other simple factors, i.e. $G=F\cdot F^{c}$. Recall that the $G$ action on a standard probability space $(S,\mu)$ is called \emph{weakly irreducible} if the $F^{c}$-action on $(S,\mu)$ is ergodic for any rank $1$ factor $F$ of $G$. For a lattice $\Gamma$, $\Gamma$ is a weakly irreducible lattice in $G$ if and only if $G$ action on $G/\Gamma$ is weakly irreducible by definition. 

We will say a $\Gamma$ action on a standard probability $\Gamma$-space $(S,\mu)$ is \emph{induced weakly irreducible}, if the induced $G$ action on $(G\times S)/\Gamma$ is weakly irreducible. Induced weakly irreducible action are ergodic.

Note that when all simple factor of $G$ have higher rank, weakly irreducibility is just ergodicity.

\subsubsection*{Local rigidity} Let $H$ be the connected component of identity in a real algebraic Lie group and let $\Lambda$ be a cocompact lattice. Let $\Acal_{0} : \Gamma\rightarrow \textrm{Aff}(H/\Lambda)$ be an affine action of $\Gamma$ on compact manifold $H/\Lambda$. Here an affine transformation is a composition of a left translation and an automorphism of $H$ preserving $\Lambda$.
 
Let $\Acal_{0}$ be an affine action of $\Gamma$ on the compact homogeneous space $H/\Lambda$. We will see there is a finite index subgroup $\Gamma'=\Gamma'(\Acal_{0})$ in $\Gamma$ that depends only on $\Acal_{0}$ so that the $\Acal_{0}\restrict{\Gamma'}$ has a nice description.

Recall that the group action is \emph{weakly hyperbolic} if there is a finite set of elements $\{\gamma_{1},\dots,\gamma_{k}\}$ in group $\Gamma$ such that $\gamma_{i}$ acts as a partially hyperbolic diffeomorphism for $i=1,\dots,k$ and the contraction subspaces of the $\gamma_{i}$ span the tangent space at every point.

Now, we state our local rigidity result.

\begin{THM}[Local rigidity]\label{thm:localrigid}   Assume that the affine action $\Acal_{0}$ on $H/\Lambda$ is weakly hyperbolic. 
Let $\Acal$ be a $C^{1}$ action on $H/\Lambda$ (resp. $C^{\infty}$ action) such that 
\begin{enumerate}
\item there is an $\Acal\restrict{\Gamma'}$ invariant fully supported Borel probability measure $\mu$ on $H/\Lambda$; and
\item $\Acal\restrict{\Gamma'}$ is weakly induced irreducible with respect to $\mu$.
\end{enumerate}
If the action $\Acal$ is sufficiently $C^{1}$ close (resp. $C^{\infty}$ close)   to $\Acal_{0}$ then $\Acal$ is $C^{0}$ conjugate (resp. $C^{\infty}$ conjugate) to $\Acal_{0}$.
\end{THM}

For any affine action $\Acal_{0}$, we will say that the action $\Acal_{0}$ is \emph{$C^{0}_{r,ind}$-local rigid} (resp. \emph{$C^{\infty}_{r,ind}$-local rigid}) if the conclusion of the Theorem \ref{thm:localrigid} holds.
\begin{rmk}If the affine action $\Acal_{0}$ only consists of left translations or automorphsim as in \cite{MQ} then the subgroup $\Gamma'(\Acal_{0})$ is the full group $\Gamma$. 

 Also, we will prove the above theorem for affine weakly hyperbolic $G$ actions. 
\end{rmk}

\subsubsection*{Global rigidity} Next, we will consider $\Gamma$-actions on the nilmanifold that has an Anosov element. When the $\Gamma$-action lifts to the universal cover, we can define an associated linear data $\rho$ that is a $\Gamma$ action on a nilmanifold by automorphisms. 

\begin{THM}[Global rigidity]\label{thm:globalrigid} Let $G$ and $\Gamma$ be as above. Let $N$ be a simply connected, connected nilpotent Lie group and $\Lambda$ be a cocompact lattice in $N$. Let $\alpha$ be a $C^{1}$ $\Gamma$-action (resp. $C^{\infty}$ $\Gamma$-action) on nilmanifold the $M=N/\Lambda$ satisfying the following conditions; 
\begin{enumerate}
\item there is a $\alpha(\Gamma)$-invariant fully supported Borel probability measure $\mu$ on $M$ such that $\alpha$ is induced weakly irreducible with respect to $\mu$,
\item there is a $\gamma_{0}\in\Gamma$ such that $\alpha(\gamma_{0})$ is an Anosov diffeomorphism of $M$,
\item $\alpha$ lifts to an action on universal cover $N$ and let $\rho$ be the associated linear data of $\alpha$.
\end{enumerate}
Then $\alpha$ is $C^{0}$ conjugate to (resp. $C^{\infty}$ conjugate to) $\rho : \Gamma\rightarrow \Aut(N/\Lambda)$. 
\end{THM}


\begin{rmk}
We compare with results in \cite{MQ} and \cite{BRHW}.
\begin{enumerate}
\item When all simple factors of $G$ are higher rank, weakly irreducibility of an action is just ergodicity and a weak irreducible lattice is just a lattice. In this case, ergodic decomposition provides appropriate tool. However, we can not expect such a decomposition for weakly irreducibility. Furthermore, weakly irreducibility does \emph{not} pass to subgroups.
\item The assumption about existence of invariant measure for $\Acal$ in Theorem \ref{thm:localrigid} is due to absence of property (T). We use the assumption about measure for Theorem \ref{thm:globalrigid} in order to use dynamical super-rigidity.
\end{enumerate} 
Nevertheless, the author expects that one may be able to get rid of the additional assumptions in the statements for actions on nilmanifolds. 
\end{rmk}

When all simple factors of $G$ have real rank at least $2$, then \cite{MQ}, \cite{FM2} and \cite{BRHW} give stronger rigidity statements. As we follow their strategy, the new rigidity results of this paper is in the case when $G$ has rank $1$ factor. Indeed, when all simple factors of $G$ have real rank at least $2$, lots of rigidity results are established. We recall a few local and global rigidity results that are related to this paper. Local rigidity for affine Anosov $\Gamma$-actions is proved by Katok and Spatzier in \cite{KS97}. In \cite{MQ}, Margulis and Qian proved local rigidity of certain affine $\Gamma$-actions assuming weakly hyperbolicity. In \cite{FM1} and \cite{FM2}, Fisher and Margulis proved local rigidity results for general affine $\Gamma$-actions in full generality without hyperbolicity assumptions. In \cite{FM3}, they proved local rigidity of isometric actions even for arbitrary discrete group with property (T). 

On the other hand, in \cite{MQ} Margulis and Qian proved topological global rigidity of Anosov $\Gamma$-actions on nilmanifolds assuming that the action can be lifted and there is a fully supported invariant probability measure. In \cite{Ben}, Schmidt proved that $C^{2}$-volume preserving weakly hyperbolic action on a torus is semi-conjugate to affine actions. Recently, in \cite{BRHW}, Brown, Rodriguez Hertz and Wang proved topological and smooth global rigidity of Anosov $\Gamma$-actions on nilmanifolds only assuming action lifts. They even prove stronger results, that is global rigidity results when the action has hyperbolic linear data. Especially, they proved smooth global rigidity of Anosov actions without assuming existence of invariant probability measures.

It is worth mentioning the very recent breakthrough in the Zimmer program. The long standing Zimmer's conjecture is settled by Brown, Fisher and Hurtado in \cite{BFH1} and \cite{BFH2}. One version of their theorems is that every smooth $\Gamma$-action on compact manifold is trivial when the dimension of manifold is less than $\textrm{rank}_{\Rbb}G$. We refer the reader to surveys \cite{Fisher_loc}, \cite{Fisher_oldZim}, \cite{Fisher_recZim} by Fisher and \cite{Spatzier} by Spatzier for detailed accounts on Zimmer program and for detailed introduction to rigidity theory.

In all of the prior results mentioned above, property (T) is used in essential ways. For this reason, the authors need to assume all simple factors of $G$ have higher rank. In particular, to prove the most useful form of Zimmer's cocycle super-rigidity theorem, one needs to use property (T).

The main ingredient in the proofs of Theorem \ref{thm:localrigid} and \ref{thm:globalrigid} is our \emph{dynamical super-rigidity} theorem. We call it \emph{dynamical cocycle super-rigidity} since it is useful from a dynamical viewpoint. The author expects the dynamical cocycle super-rigidity theorem can be applied widely in place of a Zimmer's cocycle super-rigidity theorem.

\subsection{Dynamical cocycle super-rigidity} 

When $G$ is an algebraically simply connected and all simple factors of $G$ have rank $2$ or higher, the Zimmer's cocycle super-rigidity theorem says that any integrable cocycles over an ergodic $G$ action is cohomlogous to a homomorphism up to \emph{compact error}. (See \cite[Theorem 1.4, 1.5]{FM1} )  However, if $G$ does not have property (T), for example $G=\SL(2,\Rbb)\times \SL(2,\Rbb)$, then we can not expect a compact error term. This is the main obstruction to proving cocycle super-rigidity theorems without property (T). 

Note that Margulis' super-rigidity theorem still holds for $G$ under the assumption \ref{std} below.

Despite of the above discussions, we will show that the error can be controlled so that we have \emph{dynamical cocycle super-rigidity} under mild conditions. This will provide same consequences with Zimmer's cocycle super-rigidity theorem dynamically. 

There are various generalizations of Zimmer's cocycle super-rigidity theorems, such as \cite{Furmanmonod}, \cite{FuBa} and \cite{FM1}. 
In all of these papaer, one either requires an assumption about algebraic hull of the cocycle or that groups have property (T).


\subsubsection*{Setting} First, we need to fix some notations and assumptions what we will mainly deal with.
\begin{nota}\label{std} Throughout $G=\Gbb(\Rbb)$ will be the $\Rbb$-points of algebrically simply connected, connected, semisimple algebraic group $\Gbb$ defined over $\Rbb$. We will further assumet that $G$ does not have compact factors.  In this case, we have
$$G=G_{1}\times \dots \times G_{n}, \quad n\ge 1$$ for some $G_{i}=\Gbb_{i}(\Rbb)$ where $\Gbb_{i}$ is $\Rbb$-almost simple algebraic group defined over $\Rbb$ with $\textrm{rank}_{\Rbb}(\Gbb_{i})\ge 1$.  Also throughout  $\Gamma$ will be a weakly irreducible lattice in $G$.

\begin{rmk} When we say that a group $G$ satisfies \ref{std}, $G$ need not be higher rank. For example, Theorems \ref{thm:Rvalue}, \ref{RvalueG}, \ref{Hilbertcase}, \ref{prophil} are all true when $G$ is a rank $1$ group.

 On the other hand, when we say $\Gamma$ satisfies \ref{std}, the envelop $G$ should have higher rank. Especially, if $n=1$ then $G=G_{1}$ should be higher rank simple algebraic group so that has property (T).
\end{rmk}

\end{nota}

Furthermore, we will mainly focus on group action satisfies the following assumption.
\begin{notaa}\label{std2} Let $D=G$ or $\Gamma$ satisfies \ref{std}.
Throughout $D$-action on standard probability space $(S,\mu)$ are assumed to satisfy 
\begin{enumerate}
\item when $D=G$, the $G$ action on $(S,\mu)$ is weakly irreducible, 
\item when $D=\Gamma$, the $\Gamma$ action is induced weakly irreducible, i.e. the induced $G$ action on $(G\times S)/\Gamma$ is weakly irreducible. \end{enumerate}
Note that the $D$ action on $(S,\mu)$ is ergodic in both cases.
\end{notaa}

\subsubsection*{Lyapunov exponents and Lyapunov Spectrum}

Let a topological locally compact second countable group $T$ act measurably on a standard probability space $(S,\mu)$. We will call the standard probability space $(S,\mu)$ a \emph{$T$-ergodic space}, when the $T$ action on $(S,\mu)$ is measurable, $\mu$-preserving and ergodic.

Now let $T$ be a locally compact second countable group and $(S,\mu)$ be a $T$-ergodic space. Motivated by \cite{FM1}, we will say a measurable cocycle $\beta:T\times S\rightarrow \GL(d,\Rbb)$ is \emph{$L^{2}$-integrable} if following holds : For any compact subset $Q\subset T$,  $$\sup_{t\in Q}\ln||\beta(t,-)||\in L^{2}(S,\mu).$$

From now on, $G$ and $\Gamma$ satisfies \ref{std}. Let $D=G$ or $\Gamma$. Let $(S,\mu)$ be a $D$-ergodic space. Further assume \ref{std2} for the $D$ action. For any $g\in D$, we can define a $\Zbb$-cocycle $\beta_{g} :\Zbb\times S\rightarrow \GL(d,\Rbb)$ as $$\beta_{g}(m,x)=\beta(g^{m},x).$$

For $A\in\GL(d,\Rbb)$, we will denote $||A||$ as $||A||=\textrm{max} \left( ||A||_{op}, ||A^{-1}||_{op} \right)$ where $||\cdot||_{op}$ is operator norm with respect to some norm on $\Rbb^{d}$.

If $\beta$ is $L^{2}$-integrable, then $\beta_{g}$ is also $L^{2}$-integrable for any $g\in D$ so that we can apply Oseledet's Multiplicative Ergodic Theorem. Therefore, for generic $x_{0}\in S$, there is number $k(x_{0})\le d$, a measurable decomposition $$\Rbb^{d}=V_{1}^{\beta_{g}}(x_{0})\oplus\dots\oplus V_{k(x_{0})}^{\beta_{g}}(x_{0}) $$  and  $\lambda_{1}^{\beta_{g}}(x_{0})>\dots >\lambda_{k(x_{0})}^{\beta_{g}}(x_{0})$ such that  for $v\in V_{i}(x_{0})^{\beta_{g}}\setminus \{0\}$, 
$$ \lim_{m\rightarrow \pm\infty}\frac{1}{m}\ln|\beta_{g}(m,x_{0})v|=\lambda_{i}^{\beta_{g}}(x_{0})$$
 for all $i=1,\dots,k(x_{0})$, where $|\cdot|$ is a norm on $\Rbb^{n}$. Denote $k_{i}^{\beta_{g}}(x_{0})=\dim V_{i}^{\beta_{g}}(x_{0})$ and
 define the \emph{Lyapunov spectrum at $x_{0}$} as  $\Sp_{x_{0}}(\beta_{g})$ as the collection of the above datum, i.e. $$\Sp_{x_{0}}(\beta_{g})=\{(\lambda_{i}^{\beta_{g}}(x_{0}),k_{i}^{\beta_{g}}(x_{0}),k(x_{0})):i=1,\dots, k(x_{0})\}.$$
Furthermore, we collect Lyapunov spectrums for every generic point and denote $\Sp(\beta_{g})=\left\{\Sp_{x}(\beta_{g}):x \textrm{ is generic in $S$}\right\}.$ For two $\GL(d,\Rbb)$-valued cocycles $\Asc$ and $\Bsc$ over same $\Zbb$-action, we will write $\Sp(\Asc)=\Sp(\Bsc)$ if $\Sp_{x}(\Asc)=\Sp_{x}(\Bsc)$ for almost every $x\in S$.

Finally, for fixed $M\in \GL(d,\Rbb)$,  Denote $\Sp(M)$ as the Lyapunov spectrum of cocycle $\Asc_{M}$, $\Asc_{M} :\Zbb\times (S,\mu)\rightarrow \GL(d,\Rbb)$, $(n,x)\mapsto M^{n}$ for any $\mu$-preserving $\Zbb$ action on $(S,\mu)$.

\subsubsection*{Simple version of Dynamical super-rigidity} Now we can state a simple version of our dynamical cocycle super-rigidity theorem. 
\begin{THM}[Theorems \ref{thm:superrigidG} and \ref{thm:superrigidGamma}]\label{Intsuperrigid}
Let $G$ and $\Gamma$ satisfy \ref{std} with $\textrm{rank}_{\Rbb}(G)\ge 2$. Let $D=G$ or $\Gamma$. Let $(S,\mu)$ be a $D$-ergodic space. Assume \ref{std2} for the $D$-action. Let $\beta:D\times S\rightarrow \GL(d,\Rbb)$ be an $L^{2}$-integrable measurable cocycle. Then there is a continuous representation $\pi:G\rightarrow \GL(d,\Rbb)$ such that for any $g\in D$, we have $\Sp(\beta_{g})=\Sp(\pi(g)).$

\end{THM}

The above simple version of dynamical super-rigidity already says that the Lyapunov spectrum of the action has algebraic origin.

Note that if the group $G$ has property (T), for example all $\Gbb_{i}$ has $\Rbb$-rank at least $2$, then the Theorem \ref{Intsuperrigid} is a direct consequence of the cocycle super-rigidity theorems in \cite{FM1}.

\begin{rmk} In Theorems \ref{thm:superrigidG} and \ref{thm:superrigidGamma}, we will address the full version of dynamical super-rigidity. That is the cocycle is cohomologous to the product of the homomorphism $\pi$ and a \emph{commuting cocycle}. \end{rmk}


\subsection{Outline of the paper}
In this subsection, we will briefly discuss the outline of the paper including the strategy of proofs for Theorems \ref{thm:localrigid},\ref{thm:globalrigid} and \ref{Intsuperrigid}. We start by collecting some preliminaries in Section \ref{sec:pre}.

Through Sections \ref{sec:Rvalue} and \ref{sec:dynsuperrigid}, we will prove our dynamical super-rigidity theorems. The theorems in Section \ref{sec:Rvalue} can be thought of as dynamical super-rigidity for $\Rbb$-valued cocycles.

Now we state main result in Section \ref{sec:Rvalue}, which has independent interest. As before, $D=G$ or $\Gamma$ will satisfy \ref{std}. And we assume the $D$ action on $(S,\mu)$ is measure preserving. 
	\begin{theorem}\label{thm:Rvalue} 
With the assumptions as above, let $\delta$ be a measurable cocycle $\delta:D\times S\rightarrow \Rbb$. Assume that $\delta$ satisfies $L^{2}$-integrability, i.e. 
$$(L^{2}) :\quad \forall g \in D,\quad |\delta(g,-)|\in L^{2}(S).$$
Then for every $g\in D$, we have $$\lim_{m\rightarrow\pm\infty} \frac{|\delta(g^{m},x)|}{m}=0$$ almost every $x\in S$.
	\end{theorem}

\begin{rmk} We will prove slightly general results in Theorem \ref{RvalueG} and \ref{RvalueGamma}. In these theorems, we do not require ergodicity only measure preserving action.  Furthermore, when $D=G$, Theorem \ref{thm:Rvalue} does \emph{not} require that $G$ is higher rank. Theorem \ref{RvalueG} is true when $D=G$ is locally isomorphic to $\SO(m,1)$ or $\SU(m,1)$. 
\end{rmk}

In Section \ref{sec:Rvalue}, we convert Theorem \ref{RvalueG} into a problem about unitary representation. We solve this problem using theorems from \cite{Delrome}, \cite{Shalom} and growth of harmonic maps on symmetric spaces.  
 
	In the section \ref{sec:dynsuperrigid},  we will state the dynamical super-rigidity Theorem \ref{thm:superrigidG}, in full generality and prove it using Theorem \ref{thm:Rvalue}. Indeed, we will calculate Lyapunov exponents carefully using functorial properties. This calculation will give Theorem \ref{Intsuperrigid} for $D=G$. See \textbf{Part 1} of the proof of Theorem \ref{thm:superrigidG}.  This part is highly inspired by \cite{FM1} and \cite{Zimalg}. We will follow same strategy with in \cite{FM1}. When $D=\Gamma$, we will use the standard technique of inducing cocycle.

 	In the section \ref{sec:locrigid} and \ref{sec:globrigid}, we will prove theorems \ref{thm:localrigid} and \ref{thm:globalrigid} respectively. We use tools developed in \cite{BRHW}, \cite{MQ}, \cite{FM1} and \cite{FM2}. The dynamical super-rigidity theorem will replace those authors use of Zimmer's cocycle super-rigidity theorem. 
	
	The section \ref{sec:locrigid} relies on \cite{FM1}, \cite{FM2} and \cite{MQ}. Especially, in the section \ref{sec:locrigid}, we will prove local rigidity for constant cocycles as in \cite{MQ} and \cite{FM1}. That will provide theorem \ref{thm:localrigid}. On the other hand, the section \ref{sec:globrigid} relies on \cite{MQ} and \cite{BRHW}.

Finally, in the section \ref{sec:other},  we collect a few almost direct consequences of dynamical super-rigidity. The results are analogue of statements in \cite{QZ} and \cite{Zimbook}.

\subsubsection*{\rm{\bf{Acknowledgement}}} The author thanks Alex Furman for useful discussion. The author deeply appreciates to Aaron Brown for his continuous attention to this project and fruitful discussions. Thanks to his kindness, the author could have delightful discussions with him about various mathematical subjects.  Lastly, but most importantly the author deeply appreciates his Ph.D advisor David Fisher. The author learned numerous interesting math from enjoyable and supportive discussions with him. Thanks to his encouragement, the author was able to finish this paper.

\section{Preliminaries}\label{sec:pre}

\subsection{Algebraic group}\label{sec:prealg}

 We will say an\emph{ algebraic group} $\Lbb <\GL(d,\Cbb)$, for some $d$,  is group which is a zero locus of polynomials for our purpose, i.e. an affine algebraic group over $\Cbb$. When the algebraic group is defined over $\Rbb$, we will denote $\Lbb(\Rbb)=\Lbb \cap \GL(n,\Rbb).$ We will assume that the algebraic group and rational homomorphism are all defined over $\Rbb$ if there is no mention about it. Indeed, except the section \ref{sec:Cinfglobrigid}, we only care about algebraic group and rational homomorphism defined over $\Rbb$.
 
  We will denote $\Lbb^{0}$ be the connected component containing identity of $\Lbb$. On the other hand, $\Lbb(\Rbb)^{0}$ is the connected component of identity of $\Lbb(\Rbb)$ with respect to Hausdorff topology.

We will say connected real algebraic Lie group $L$ if there is a connected algebraic group $\Lbb$ defined over $\Rbb$ such that $L=\Lbb(\Rbb)^{0}$.

Note that we can find algebraic universal cover $\Gbb$ for connected semisimple real algebraic Lie group without compact factor $G_{0}=\Gbb_{0}(\Rbb)^{0}$.  $\Gbb$ satiesfies \ref{std} and there is a central isogeny $\Gbb\rightarrow \Gbb_{0}$. Let $\Gamma_{0}$ be an weakly irreducible lattice in $G_{0}$. Then $\Gamma=p^{-1}(\Gamma_{0})$ is still a weakly irreducible lattice in $\Gbb(\Rbb)$ where $p:\Gbb(\Rbb)\rightarrow G_{0}$ be a projection.

The facts in algebraic groups used in this paper can be found various literatures such as \cite{M}, \cite{Zimbook}, \cite{Borel}, \cite{OVbook} and \cite{OVlie3}.

\subsection{Unitary representation and $1$-cocycle}
\begin{defn} For LCSC group $G$, the \emph{unitary representation} $\pi:G\rightarrow \Ucal(\Hcal)$ on the Hilbert space $\Hcal$ is group homomorphism and continuous with respect to strong operator topology on $\Ucal(\Hcal)$. We will say the continuous map $b:G\rightarrow \Hcal$ is \emph{$1$-cocycle over the unitary representation $\pi$} if we have $b(gh)=b(g)+\pi(g)(b(h))$ for all $g,h\in G$. We will denote $Z(G,\pi)$ be the space of all $1$-cocycle. If $1$-cocycle $b\in Z(G,\pi)$ is called \emph{$1$-coboundary} if there is a $v\in \Hcal$ such that $b(g)=\pi(g)v-v$ for all $g\in G$. We will denote $\Bcal(G,\pi)$ be the space of all $1$-coboundary. We also define $1$-cohomology $H^{1}(G,\pi)$ and $1$-almost cohomology $\overline{H^{1}}(G,\pi)$ as $H^{1}(G,\pi)=Z(G,\pi)/\Bcal(G,\pi)$ and $\overline{H^{1}}(G,\pi)=Z(G,\pi)/\overline{\Bcal(G,\pi)}$ respectively. We will call the element in $\overline{\Bcal(G,\pi)}$ as \emph{almost coboundary}.
\end{defn}
Recall that the measurable group homomorphism between two LCSC group is continuous. We can get continuity from the cocycle relation too.
\begin{lem}\label{conti} 
 Let $\pi : G\rightarrow \mathcal{U(H)}$ be a unitary representation of $G$. 
Let $b : G\rightarrow \mathcal{H}$ be a weakly measurable and satisfy following condition. 
$$b(gh)=b(g)+\pi(g)[b(h)]$$ for all $g,h\in G$. 
Then $b\in Z^{1}(G,\Hcal)$.
\end{lem}
\begin{proof}  \cite{BHV}, Exercise 2.14.3.  \end{proof}
We define a sublinear growth of $1$-cocycle with respect to unitary representation as follows. 
\begin{defn} Let cocycle $b:G\rightarrow \Hcal$ which is associated to unitary representation $\pi:G\rightarrow \mathcal{U(H)}$ has \emph{sublinear growth} if for any $\epsilon>0$, there is $N\in \Nbb$ such that for any $g\in G$ with $|g|_{G}>N$, $|| b(g)||<\epsilon |g|_{G}$ where $||\cdot||$ denotes norm on $\Hcal$
\end{defn}
The followings are some properties of sublinear growth.
\begin{lem}[\cite{CTV} Corollary 3.3]\label{almsub} Almost coboundary has sublinear growth. \end{lem}
\begin{proof} See \cite{CTV} Corollary 3.3. Indeed the set of sublinear growth cocycle forms closed set in $1$-cocycle $Z^{1}$.  \end{proof}

 \begin{lem}\label{summ} For compactly generated LCSC metrizable group $H$ with some unitary representation $(\pi,\mathcal{H})$, consider two 1-cocycles $c_{1}, c_{2}:H \rightarrow \mathcal{H}$ which both have sublinear growth. Then $c_{1}+c_{2}$ has also sublinear growth.
\end{lem} 

\begin{proof}
Let denotes $|\cdot|$ as word metric on $H$ and $||\cdot||$ as norm on $\mathcal{H}$. 
Fix $\epsilon >0$ and choose $N_{1},N_{2}\in \Nbb$ such that for any $h\in H$ with $|h|>N_{i}$ then $||c_{i}(h)|| <\frac{\epsilon}{2} |h|$ for $i=1,2$.
Since $c_{1}$ and $c_{2}$ have sublinear growth, we can find such $N_{1}$ and $N_{2}$. Now define $N=\max(N_{1}, N_{2})$. Then for any $h\in H$ with $|h|>N$, $||c_{1}(h)+c_{2}(h)||\le ||c_{1}(h)||+||c_{2}(h)|| \le \epsilon |h|.$ This means that $c_{1}+c_{2}$ has sublinear growth as we required. \end{proof}

Combining the above facts, we can get following lemma. It helps us to ignore almost coboundary when we calculate a growth of the cocycle.

\begin{lem}\label{samegrowth}
If $b\in Z^{1}(G,\Hcal)$ has sublinear grwoth
 then for any cocycle $b'\in Z^{1}(G,\Hcal)$ which is almost cohomologous to $b$ has sublinear growth.
\end{lem}
\begin{proof} This comes from Lemma \ref{almsub} and \ref{summ} directly. \end{proof}

  \subsection{Cocycles over the group action}\label{precoc}
  \subsubsection*{Measurable cocycles} Let $D$ and $T$ be LCSC group. Let $(S,\mu)$ be a standard probability space with $D$ action. Assume that $D$ action preserves $\mu$.
 \begin{defn} We will call the measurable (resp. continuous) map $\alpha:D\times S\rightarrow T$ is a mearsurable (resp. continuous) cocycle, if for any $g_{1},g_{2}\in D$, we have $\alpha(g_{1}g_{2},x)=\alpha(g_{1},g_{2}.x)\alpha(g_{2},x)$ for almost every $x\in S$. We will call two cocycles $\alpha_{1},\alpha_{2}:D\times S\rightarrow T$ is \emph{cohomologous} if there is a measurable map $\phi:S\rightarrow T$ such that $$\alpha_{1}(g,x)=\phi(g.x)^{-1}\alpha_{2}(g,x)\phi(x)$$ for any $g\in D$ and almost every $x\in S$. We sometimes denote the cocycle $\alpha_{1}^{\phi}$ as $\alpha_{1}^{\phi}(g,x)=\phi(g.x)^{-1}\alpha_{1}(g,x)\phi(x)$
 \end{defn}

 When we have group homomorphism $\pi : D\rightarrow T$, we can make constant cocycle $\pi:D\times S\rightarrow T$ for any $D$ action on $S$. Therefore, we abuse notations so that the group homomorphism as cocycle when the context is clear enough.
 \begin{lem}\label{conticoc} Let $G$ be LCSC metrizable group. Let $\alpha:G\times S\rightarrow \Rbb$ be a cocycle. Assume that $$(L^{2}) : \quad \alpha(g,-)\in L^{2}(S) \quad \forall g\in G$$ then the map $b:G\rightarrow L^{2}(S), b(g)(s)=\alpha(g^{-1}, s)$ is $1$-cocycle with respect to unitary representation $\pi:G\rightarrow \Ucal(L^{2}(S))$ where $\pi(g)(f)= f^{g}$ and $f^{g}(s)=f(g^{-1}.s).$ In other words, $b\in Z^{1}(G,\pi)$.
 \end{lem}
\begin{proof} The map $b$ satisfies cocycle identity $b(gh)=b(g)+\pi(g)(b(h))$ comes from direct calculation. The remaining duty is proving continuity. It can be proved by Lemma \ref{conti}. Or we can think $\alpha\in F(G,L^{2}(S))$ where $F(G,L^{2}(S))$ is the space measurable map from $G$ to $L^{2}(S)$, and use Theorem 3 in \cite{Moore}.  \end{proof}

\begin{rmk} Following \cite{Zimbook}, the measurable cocycle can be realized by \emph{strict} measurable cocycle almost everywhere if $D$ and $T$ is a locally compact second countable metrizable group. (See \cite{Zimbook} Chapter 4 and Appendix.) Therefore, we will not distinguish measurable cocycle with strict cocycle. Indeed, the every group that will be appeared in this paper will be LCSC group except the unitary operator group $\Ucal(\Hcal)$ of infinite dimension Hilbert space.
\end{rmk}

Note that the following lemma is easily deduced by direct calculation.
\begin{lem}\label{comcoc} Let $\beta,\delta :D\times S\rightarrow T$ be measurable cocycles and $\Delta:D\times S\rightarrow T$ be a measurable map. Assume that that $\delta(g,x)=\beta(g,x)\Delta(g,x)$ for all $g\in D$ and almost every $x\in S$. If $\beta(D\times S)$ commutes with $\Delta(D\times S)$ then $\Delta$ is indeed a cocycle.
\end{lem}

\subsubsection*{Algebraic hull}

For a measurable cocycle $\beta$ over an ergodic action, there may not exist smallest subgroup among target groups of all cohomologous cocycles. However, if we restrict to our attention to algebraic groups, then there is a such group. More precisely, we can find a  $\Rbb$-algebraic group $\Lbb(\Rbb)$ so that there is no algebraic subgroup $\Lbb'$ of $\Lbb$ such that 
$\beta$ is cohomologous to a cocycle takes values in $\Lbb'(\Rbb)$. Such a minimal $\Rbb$-algebraic group $\Lbb$ is unique up to conjugation and called $\Lbb$ or $\Lbb(\Rbb)$ as an \emph{algebraic hull} of $\alpha$.  (See, \cite{Zimbook} Chapter 9.)   If the cocycle $\beta:G\times S\rightarrow \Lbb(\Rbb)$ takes values in an algebraic hull then we will say $\beta$ is \emph{minimal}.

\subsubsection*{Finite extension}\label{finext}
Let $G$ satisfy \ref{std} with $n\ge 1$. Let $\beta : G\times X \rightarrow L$ be measurable cocycle under ergodic $G$-action on standard probability space $X$. Further assume that $L=\Lbb(\Rbb)$ be a $\Rbb$ point of real algebraic group $\Lbb$ and algebraic hull of the $\beta$. 

Let $\Lbb_{0}$ be a algebraic normal subgroup of $\Lbb$ such that $C=L/L_{0}$ is compact group where $L_{0}=\Lbb_{0}(\Rbb)$. We have a cocycle $i:L\times X\rightarrow C$ as $i=pr\circ \alpha$, where $C=D/D_{0}$ is compact group and $pr:L\rightarrow C$ is projection. We define action of $G$ on $\widehat{X}=X\times_{i} C$ as $g.(x,c)=(g.x,i(g,x)c).$ Note that we can give measure on $\widehat{X}$ as product of $\mu$ and Haar measure on $C$. Now define cocycle $$\overline{\beta} : G\times (X\times C) \rightarrow L,\quad\overline{\beta}(g,(x,f))= \beta(g,x).$$ over the $G$ action on $\widehat{X}$.
It is easy to see that there is bounded measurable function $\psi:\widehat{X}\rightarrow L$ such that $\delta :G\times \widehat{X} \rightarrow L_{0}$, $$\delta(g,(x,f))=\psi(g.(x,f))^{-1} \overline{\beta}(g,(x,f))\psi(x,f)$$ is $L_{0}$ valued cocycle. Indeed, one can define $\psi(x,f)=\overline{\psi}(f)$ there $\overline{\psi}$ is measurable section from $C$ to $L$.
  
 Assume that $\beta$ is $L^{2}$-integrable cocycle. Then $\delta$ is also $L^{2}$-integrable cocycle and for any $g\in G$, $\Sp(\delta_{g})=\Sp(\beta_{g})$ when we apply MET to trivial bundle $S\times \Rbb^{d}$ for some $d$.

 The following theorem allows us to assume that the algebraic hull is connected algebraic group without loss of generality.
 	\begin{theorem}[Cf. \cite{Stuck} Theorem 2.1.]\label{thm:finextirred} Under notations and settings  as above, we have following.
 	\begin{enumerate}
	\item[(a)] $\Lbb_{0}$ is algebraic hull of $\overline{\beta}$, so that $\delta$ is a minimal cocycle.
	\item[(b)] Moreover, if $G$ action on $X$ is weakly irreducible then $G$ action on $\widehat{X}$ is also weakly irreducible.
 	\end{enumerate}
	In particular, when $\Lbb_{0}=\Lbb^{0}$ is connected component of identity of $\Lbb$, $C$ is finite and $G$ action on $\widehat{X}$ is weakly irreducible if $G$ action on $X$ is weakly irreducible. In this case, we will call $\widehat{X}$ as \emph{finite extenstion}.
 
 \end{theorem}
\begin{proof}  For (a), see \cite{Stuck}  Proposition1.4. For (b), the proof of Theorem 2.1 in \cite{Stuck} can be applied same way. By the definition of weakly irreducible action, $G$ action on $F\times X$ is ergodic for any rank one factor $F$ of $G$. This is enough to prove that the $F^{c}$ action on $\widehat{X}$ is ergodic where $F^{c}$ is complement of $G$, i.e. $G=F\times F^{c}$. \end{proof}
\subsection{Lyapunov subspaces and exponents} \label{Lyapsec}\label{funct}
We need funtorial property of Lyapunov exponents and subspaces for integrable cocycle and quasi-integrable cocycle. For more informations of it, see \cite{M} V.2.3., \cite{Feres} 9.3.7. and \cite{Zimalg}.

 Let $(S,\mu)$ be a standard probability space. Assume that $\theta:S\rightarrow S$ gives $\mu$-preserving $\Zbb$-action on $(S,\mu)$. Let $\Csc:\Zbb\times S\rightarrow \GL(W)$ be a measurable cocycle defined over $\theta$. 
 
  When $\Csc$ is integrable, we will denote $E_{0}^{\Csc}(x)$,$E_{+}^{\Csc}(x)$,$E_{-}^{\Csc}(x)$,$E_{\le 0}^{\Csc}(x)$ and $E_{\ge 0}^{\Csc}(x)$ for the subspaces in $W$ corresponding signs for Lyapunov exponents of the cocycle $\Csc$. 
 
 Let $Q$ be a linear subspace in $W$ such that $\Csc(m,x)Q=Q$ for all $m\in\Zbb$. Let $F=W/Q$ and denote by $p:W\rightarrow F$ the natural projection. Then we can define two induce cocycles $$\Csc\restrict{Q} : \Zbb\times S\rightarrow \GL(Q),\quad(m,x)\mapsto \Csc(m,x)\restrict{Q}$$ and $$\Csc\restrict{F}:\Zbb\times S\rightarrow \GL(W),\quad\Csc\restrict{F}(m,x)\overline{z}=p(\Csc(m,x))z$$ for any $m\in \Zbb$, $\overline{z}\in F$, $p(z)=\overline{z}$ and $\mu$-almost every $x\in S$. Note that $\Csc\restrict{F}$ is well-defined.

\subsubsection*{Functoriality of the Integrable cocycle} Assume that $\Csc$ is $L^{2}$ integrable so that $\Csc\restrict{Q}$ and $\Csc\restrict{F}$ are $L^{2}$ integrable. Using Oseledet Multiplicative ergodic theorem, find Lyapunov subspaces and exponents for $\Csc$, $\Csc\restrict{Q}$ and $\Csc\restrict{F}$. For $\Csc$,  for $\mu$-almost every $x\in S$, $$\chi_{1}(x)>\dots>\chi_{t(x)}(x),\quad W=\bigoplus_{i=1}^{t(x)} W_{\chi_{i}(x)}(x)$$ be Lyapunov exponents and subspaces respectively.
For $\Csc_{Q}$ and $\Csc_{F}$, for $\mu$-almost every $x\in S$, $$\chi^{Q}_{1}(x)>\dots>\chi^{Q}_{q(x)}(x),\quad Q=\bigoplus_{j=1}^{q(x)}Q_{\chi^{Q}_{j}(x)}(x)$$ and $$\chi^{F}_{1}(x)>\dots>\chi^{F}_{f(x)}(x),\quad F=\bigoplus_{j=1}^{f(x)}F_{\chi^{F}_{j}(x)}(x)$$ be Lyapunov exponents and subspaces respectively. 
Fix generic $x\in S$ for all the above ergodic theorems. 
\begin{enumerate}
\item The decomposition $$Q=\bigoplus_{i=1}^{t(x)} \left(W_{\chi_{i}(x)}(x)\bigcap Q\right)$$ is Lyapunov decomposition for cocycle $\Csc\restrict{Q}$ at $x$. The Lyapunov exponents for $\Csc\restrict{Q}$ are exponents $\chi_{i}(x)$ of $\Csc$ such that $Q\cap W_{\chi_{i}(x)}(x)\neq \emptyset$. 

\item For each $j=1,\dots,f(x)$, $\chi_{j}^{F}(x)=\chi_{i}(x)$ and $F_{\chi^{F}_{j}}(x)=p(W_{\chi_{i}(x)}(x))$ for some $i\in\{1,\dots,t(x)\}$.
\item If, for some $i\in \{1,\dots,t(x)\}$, $\chi_{i}(x)\neq \chi^{F}_{j}(x)$ for all $j$, then $W_{\chi_{i}(x)}(x)$ is a subspace of $Q$. 

\end{enumerate}

Especially, if two cocycles $\Csc,\Csc':\Zbb\times S\rightarrow \GL(W)$ over same $\Zbb$-action preserve $Q<W$ and $\Sp_{x}(\Csc\restrict{Q})=\Sp_{x}(\Csc'\restrict{Q})\textrm{ and }  \Sp_{x}(\Csc\restrict{F})=\Sp_{x}(\Csc'\restrict{F})$ almost every $x\in S$ then we have $\Sp(\Csc)=\Sp(\Csc').$
The converse is also true.
\subsubsection*{Functoriality of the exponential cocycle}  

Define the cocycle $\Csc'$ is called \emph{quasi $L^{2}$-integrable} if it is cohomologous to a $L^{2}$-integrable cocycle $\Csc$, i.e. there is a measurable map $\phi:X\rightarrow \GL(W)$ such that $\Csc'=\Csc^{\phi}$. In this case, we can still say about Lyapunov exponent using $\{\phi<N\}$ for any $N\in \Nbb$. Especially  $E_{*}^{\Csc'}(x)=\phi(x)^{-1}E_{*}^{\Csc}(x)$ is desired Lyapunov subspaces with respect to $\Csc'$.

For these subspaces, we also have same analogue of funtoriality. See \cite{Zimalg} Lemma 2.5.

 \begin{rmk} Note that quasi-$L^{2}$ integrable cocycle is exponential in the sense of \cite{Zimalg}. Also, integrable and quasi-integrable conditions are enough for functoriality. 
\end{rmk}

 \subsection{Norm on groups}\label{prenorm}
 
 \begin{defn} 
We will call the \emph{pseudo-norm} on group $H$, the map $||\cdot||:H\rightarrow [0,\infty)$ such that $||gh||\le ||g||+||h||$ for any $g,h\in H$. We will call pseudo-norm as norm if $||h||=0$ if and only if $h$ is identity.
We will call the map $j:H_{1}\rightarrow H_{2}$ between two groups with pseudo-norm $(H_{1},||\cdot||_{1})$ and $(H_{2},||\cdot||_{2})$ is called \emph{Quasi-isometry} if there are constant $C_{1}\ge1$ and $C_{2}\ge0$ such that $$C_{1}^{-1}||h||_{1}-C_{2}\le ||j(h)||_{2}\le C_{1}||h||_{1}+C_{2}$$ for any $h\in H_{1}$.
\end{defn} 
\begin{eg} 
Let $F$ be compactly generated LCSC metrizable group. Fix compact symmetric generating set $\Omega$ containing identity of $F$. Then we can define \emph{word norm with respect to $\Omega$} on $F$ as $|f|_{\Omega}=\inf\{n:f\in \Omega^{n}\}.$ Any two compact generating set $\Omega_{1}$ and $\Omega_{2}$ give quasi-isometric word norm on $F$. We will call that the norm $|\cdot |_{F}$ is \emph{canonical norm} on $F$ if $|\cdot|_{F}$ is quasi-isometric to word norm with respect to some, hence any compact generating set.
\end{eg}
\begin{eg} Let $F< \GL(m,\Rbb)$ be closed subgroup of $\GL(m,\Rbb)$. Then we can define norm  $\ln||\cdot||_{op}$ on $F$ using operator norm. Since any two norm on $\Rbb^{n}$ is compatible, $||\cdot||_{op}$ are all compatible. Now, for some fixed norm on $\Rbb^{n}$, we can define pseudo-norm $$||b||=\max(\ln||b||_{op}, \ln||b^{-1}||_{op}) $$ for any $b\in F$. \end{eg}
For reductive algebraic group, they are quasi-isometric.
\begin{theorem}[See \cite{Abel04}]\label{rednorm}
If $F<\GL(m,\Rbb)$ is $\Rbb$-points of a reductive algebraic group defined over $\Rbb$, then $||\cdot||$ is quasi isometric to canonical norm on $F$. 
\end{theorem} 

\begin{eg} Let $G$ be connected semisimple Lie group with finite center.  Then, we can define pseudo norm $||\cdot||_{G/K}$ on $G$ as following. Fix maximal compact subgroup $K<G$. Define pseudo norm $||\cdot||_{G/K}$ on $G$ as $$||g||_{G/K}=d_{G/K}(eK,bK)$$ where $d_{G/K}$ be left $G$-invariant Riemannian structure on a symmetric space $G/K$. 
\end{eg}
It can be verified that $||\cdot||_{G/K}$ is quasi isometric to $|\cdot|_{G}$, where $|\cdot|_{G}$ is canonical norm on $G$ using Cartan decomposition of $G$. 
We can state as follows.  

\begin{lem}\label{comp} If $G$ is connected semisimple Lie group with finite center, there is constant $D_{1}, D_{2}>0$ and compact generating set $Q$ such that $ D_{1}|g|_{G} \le d(gK, K) \le D_{2} |g|_{G}$
for any $g\in G$ with $|g|_{G}>3$ .
\end{lem}

\subsection{The space of group actions and cocycles}\label{preact}

   Let $D$ be LCSC groups. Let $(S,\mu)$ be a standard probability space with $\mu$-preserving $D$ action. Let $T<\GL(d,\Rbb)$ be real algebraic Lie group with pseudo-norm $\ln||\cdot||$ inherited from $\GL(d,\Rbb)$. 
Fix compact generating set $\Omega$ of $D$.  

Let $C^{0}(D,S,T)$ (resp. $\Mcal(D,S,T))$ be the space of continuous (resp. measurable) cocycle from $D\times S$ to $T$. We will say $\beta,\delta\in C^{0}(D,S,T)$ is \emph{$C^{0}$-closed} if for any $g\in \Omega$ we have $\sup_{x\in S} \ln||\delta(g,x)\beta(g,x)^{-1}||<\epsilon$ for some $\epsilon>0$. On the other hand, two cocycles $\beta,\delta \in \Mcal(D,S,T)$ is \emph{$L^{\infty}$-closed} if for any $g\in \Omega$ we have $\textrm{esssup}_{x\in S}\ln||\delta(g,x)\beta(g,x)^{-1}||<\epsilon$ for some $\epsilon>0$.

For the space of the actions, let $M$ be a compact manifold. We will call $C^{k}$-action of $D$ on $M$ for homomorphism $D\rightarrow \Diff^{k}(M)$.  We can endow $C^{r}$-topology on $\Diff^{k}(M)$ for any $r\le k$.  Then we can define $C^{k}(D,M)$ be the space of $C^{k}$-action of $D$ on $M$.

For any $\Acal_{0}\in C^{r}(D,M)$ we define $U(\Acal_{0},W)\subset C^{r}(D,M)$ for each $W$ that is open in $C^{k}$-topology in $\Diff^{r}(M)$ containing $id_{M}$ as $\Acal\in U(\rho_{0},W)$ if and only if $\Acal_{0}^{-1}(g)\rho(g)\in W$ for any $g\in \Omega$ and all $x\in M$. We can give a topology on $C^{r}(D,M)$ as $\{U(\Acal_{0},W)\}_{W}$ be a open neighborhood basis of $\Acal_{0}$. This allows us to say $C^{k}$-closed notion between two actions $\Acal,\Acal_{0}\in C^{r}(D,M)$.

\subsection{Group actions on manifolds}\label{sec:prehyp}
\subsubsection*{Affine actions} 
   Let $H$ be a connected real algebraic Lie group. Let $\Lambda$ be the cocompact lattice in $H$. We will denote $\Aff(H/\Lambda)$ as the group of affine transformations on $H/\Lambda$. Also, we will denote $\Aut(H/\Lambda)=\{L\in \Aut(H):L(\Lambda)\subset \Lambda\}$. Recall that any $f\in\Aff(H/\Lambda)$ can be written as $f=h_{f}\circ L_{f}$ where $h_{f}\in H$ and $L_{f}\in \Aut(H/\Lambda)$. Therefore,  we have the map $\Phi:\Aut(H/\Lambda)\ltimes H\rightarrow \Aff(H/\Lambda)$ given by $f\mapsto (L_{f},h_{f})$.
   
   In \cite{FM1}, Fisher and Margulis provided full description and properties of $\Aff(H/\Lambda)$ as follows.
   \begin{prop}[\cite{FM1}, Proposition 6.1] Define a map $\Delta^{-1}:\Lambda\rightarrow \Aut(H/\Lambda)\ltimes H$ given by $\lambda\mapsto  (c_{\lambda},h_{\lambda})$ where $c_{\lambda}\in\Aut(H/\Lambda)$ be a conjugation map by $\lambda$. Then the kernel of the map $\Phi$ is $\Delta^{-1}(\Lambda)$.
   \end{prop}

   \begin{lem}[\cite{FM1}, Lemma 6.2]  Let $H$ be a real algebraic Lie group and $\Lambda<H$ be a cocompact lattice. Let $p:H'\rightarrow H$ be a covering and $\Lambda'=p^{-1}(\Lambda)$. Then
   	\begin{enumerate}
		\item $H/\Lambda$ is diffeomorphic to $H'/\Lambda'$
		\item $\Aff(H/\Lambda)<\Aff(H'/\Lambda')$.
	\end{enumerate}
   \end{lem}
   \begin{theorem}[\cite{FM1}, Proposition 6.3] There is a cover $p:H'\rightarrow H$ and a realization of $H'$ as $\Hbb'(\Rbb)$ for a connected real algebraic group $\Hbb'$ such that 
   	\begin{enumerate}
		\item there is a finite index subgroup $\Aut^{A}(H')<\Aut(H')$ such that all elements of $\Aut^{A}(H')$ are rational automorpihisms of $H'$ and 
		\item $\Aut^{A}(H)\ltimes H'$ is the real points of a real algebraic group.
	\end{enumerate}

   \end{theorem}
   
\subsubsection*{Weak hyperbolic action} In this subsection, we will collect a few facts about weakly hyperbolic actions. Let $\Acal$ be a $C^{1}$-action of $D$ on $M$.  
   \begin{defn}[Weak hyperbolic]
   We will say $\Acal$ is \emph{weakly hyperbolic} if there is finite number of elements $g_{1},\dots, g_{l}\in D$ such that 
   \begin{enumerate}
   \item for all $i=1,\dots,l$, $\Acal(g_{i})$ is a partially hyperbolic diffeomorphism, and 
   \item $TM=\sum_{i=1}^{l} W_{-}^{i}$ where  $W_{-}^{i}$ be strong stable subbundle of $TM$ for $g_{i}$. (not necessarily direct sum.)
\end{enumerate}
   \end{defn}
   For certain types of affine actions, weakly hyperbolic is characterized in \cite{MQ}.

In \cite{Benth,Ben}, Schmidt proved that $C^{2}$-volume preserving weakly hyperbolic action is ergodic, indeed weakly mixing. 
\begin{theorem}[\cite{Benth,Ben}]\label{thm:weakhypmix} Any $C^{2}$-volume preserving weakly hyperbolic $\Gamma$-action is weakly mixing. 
\end{theorem}

The following is easily deduce from the definition of weakly hyperbolic action. (See e.g. \cite{MQ} and \cite{FM2}.)
\begin{lem} 
The weakly hyperbolic affine action $\Acal_{0}$ is expansive. \end{lem}
Recall that if $\Acal_{0}$ is expansive then any sufficiently $C^{1}$-close action $\Acal$ is still expansive.

As a special case of weakly hyperbolic action, we can consider Anosov action.  Let $\alpha : D\rightarrow \Diff^{1}(M)$ be a $C^{1}$ $D$-action on compact manifold $M$. The element $g\in D$ is called \emph{Anosov element} if $\alpha(g)$ is an Anosov difeomorphism. The action $\alpha$ is called \emph{Anosov} if there is an Anosov element. It is clear that the Anosov action is weakly hyperbolic. 

We will consider Anosov action on nilmanifolds later. Let $N$ be a connected, simply connected nilpotent Lie group and $\Lambda$ be a lattice in $N$. Recall that $\Lambda$ should be a cocompact. Let $D$ be a finitely generated group and $\alpha:D\rightarrow \Diff^{1}(M)$ is Anosov $D$ action on nilmanifold $M=N/\Lambda$. 

Assume that $\alpha$ lifts on the universal cover $N$, then we can obtain well defined homomorphism $\rho:D\rightarrow \Aut(\Lambda)$ by identifying Deck transformation group and $\Lambda$. As well known, every element in $\Aut(\Lambda)$ can be uniquely extended to the element in $\Aut(N)$. (e.g. \cite{Rag}) Therefore $\rho$ induces the homomorphism $\rho:D\rightarrow \Aut(N)$ and $\rho: D\rightarrow \Aut(N/\Lambda)$ where $\Aut(N/\Lambda)$ is a group of automorphisms of $N$ that preserving $\Lambda$.

Using linear data, we can define continuous cocycle $\beta$ that contains information of the action $\alpha$. We will identify $N/\Lambda$ with the space $\left(\rho(\Gamma)\ltimes N\right)/\left(\rho(\Gamma)\ltimes \Lambda \right)$ canonically. Note that $\rho(\Gamma)$ is discrete subgroup of $\Aut(N
)$ and $\rho(\Gamma)\ltimes N$ is a subgroup of $\Aut(N)\ltimes N$.
\begin{lem}[\cite{MQ} Example 2.4 case 2]\label{liftcoc}
There exist a continuous cocycle $\beta:\Gamma\times N/\Lambda\rightarrow \rho(\Gamma)\ltimes N$ such that 
\begin{enumerate}
\item For every $\gamma\in \Gamma$ and $[x]\in N/\Lambda$, $\alpha(\gamma)([x])=\beta(\gamma,[x])[(1,x)]$ where $1$ is the identity element in $\rho(\Gamma)$.
\item There is a continuous map $u:\Gamma\times N/\Lambda\rightarrow N$ such that $\beta(\gamma,[n])=\rho(\gamma)u(\gamma,[n])$ for any $\gamma\in \Gamma$ and $[n]\in N/\Lambda$. In other words, $\Aut(N)$-component of $\beta$ is same as $\rho$. 
\end{enumerate}
\end{lem}
\begin{rmk} Note that the map $u$ may \emph{not} satisfy cocycle equation. It is a \emph{twisted cocycle} in the sense of \cite{F}. In addition, in the \cite{MQ}, they do not use the term linear data. However, the construction shows that the desired automorphism is a linear data of $\alpha$. 
   \end{rmk}

      

  \section{Sublinear growth of $\Rbb$-valued cocycles}\label{sec:Rvalue}

In this section we will prove Theorem \ref{thm:Rvalue}. Indeed, we prove more general theorems. Throughout this section, $G$ or $\Gamma$ will always satisfy \ref{std}. We denote $|\cdot|_{w}$, $|\cdot|_{i}$ and $|\cdot|_{G}$ as canonical symmetric norm on $\Gamma$, $G_{i}$ and $G$ respectively.  

Firstly, when $D=G$ case, we have a following theorem. 

\begin{theorem}\label{RvalueG} 
Let $G$ satisfy \ref{std} and $(S,\mu)$ be the standard probability space with measure preserving $G$-action. Let $\delta:G\times S\rightarrow \Rbb$ be  $L^{2}$-integrable cocycle in the following sense. $$(L^{2}) :\quad \forall g \in G,\quad |\delta(g,-)|\in L^{2}(S)$$

Then for any $\epsilon>0$, there is $N>0$ such that
$$||\delta(g,-)||_{L^{2}(S)}=\left( \int_{S} |\delta(g,s)|^{2} d\mu(s)\right)^{\frac{1}{2}} \le \epsilon|g|_{G}$$  for $|g|_{G}>N$. Especially, for every $g\in G$, we have $$\lim_{m\rightarrow\pm\infty} \frac{|\delta(g^{m},x)|}{m}=0$$ almost every $x\in S$.
\end{theorem}

The first observation is that it is enough to prove following theorem.
\begin{theorem}\label{Hilbertcase} Under assumption \ref{std} (possibly $n=1$), for any unitary representation $\pi :G\rightarrow \mathcal{U}(\Hcal)$, every $1$-cocycle $b\in Z^{1}(G,\pi)$ has sublinear growth.
\end{theorem}
\begin{proof}[Proof of Theorem \ref{RvalueG} from Theorem \ref{Hilbertcase}]
We define unitary representation $\pi:G\rightarrow L^{2}(S,\Cbb)$ and $b:G\rightarrow L^{2}(S,\Cbb)$ associated with $\pi$ as $b(g)(s)=\delta(g^{-1},s).$ By Lemma \ref{conti}, we may assume that $b$ is in $Z^{1}(G,\pi)$. Now we can use Theorem \ref{Hilbertcase}. So that $b$ has sublinear growth so for any $\epsilon>0$, there is a $N$ such that for any $g\in G$ with $|g|_{G}>N$, we have $||b(g)||<\epsilon|g|_{G}$ where $||\cdot||$ is norm on $L^{2}(S,\Rbb)$. Then, 
$$||b(g)||=\left( \int_{S} |\delta(g^{-1},s)|^{2} d\mu(s)\right)^{\frac{1}{2}} \le \epsilon|g|_{G}=\epsilon|g^{-1}|_{G}$$  for $|g|_{G}>N$. This proves first assertions.

For second assertions, fix  $g\in G$. Then we have
$$\lim_{m\rightarrow\infty} \frac{||\delta(g^{m},-)||_{L^{2}(S)}}{m}=0.$$
 Here we use the fact that $|g^{m}|_{G}\le m|g|_{G}$ for all  $m$. So there is sequence of $m_{k}\rightarrow \infty$ as $k\rightarrow \infty$ so that for almost every $x\in S$, 
$$\lim_{k\rightarrow\infty} \frac{|\delta(g^{m_{k}},x)|}{m_{k}}=0.$$ 
However, subadditivity ergodic theorem tells that for almost every $x\in S$, the limit $$\lim_{m\rightarrow \infty}\frac{|\alpha(g^{m},x)|}{m}$$ exists. Therefore for almost every $x\in S$, we have $$\lim_{m\rightarrow \infty}\frac{|\delta(g^{m},x)|}{m}=0. $$  Using $g^{-1}$ instead of $g$, we can get same limit when $m\rightarrow-\infty$. This proves second assertion. \end{proof}

In order to prove \ref{Hilbertcase}, we first prove following proposition which is $n=1$ case for Theorem \ref{Hilbertcase}.
\begin{prop}\label{prophil}
Let $G=\Gbb(\Rbb)$ be $\Rbb$ point of algebraically simply connected, $\Rbb$ simple, $\Rbb$ algebraic group $\Gbb$. Let $b:G\rightarrow \Hcal$ is cocycle associated with unitary representation $\pi:G\rightarrow \mathcal{U(H)}$. Then $b$ has sublinear growth. 
\end{prop}

\begin{proof}[Proof of Proposition \ref{prophil}]First of all, we may assume $G$ is not compact. Furthermore, if $G$ has property (T) then, $b$ should be a coboundary, so it is bounded and it implies sublinear growth. Therefore, we can assume that $G$ does not have property (T). Since $G=\Gbb(\Rbb)$ is a real point of an algebraically simply connected real algebraic group, we may assume that that $G$ is connected Lie group which is locally isomorphic to $\SO(n,1)^{\circ}$ or $\SU(n,1)$ with finite center.

 In addition, we may assume $b|_{K} \equiv 0$ where $K$ is maximal compact subgroup of $G$. Indeed, for an affine action $A$ as $A(g)(w)=\pi(g)w +b(g)$ on $\Hcal$  we may find $A(K)$ invariant vector $v\in\Hcal$ due to compactness. This implies that $b'(g)=b(g)+\pi(g)v-v$ is cohomologous to $b$ and $b'|_{K}\equiv 0$. Now Lemma \ref{samegrowth} tells that we can assume $b|_{K}\equiv 0$.

We can decompose $\pi$ and $b$ into irreducible representations. (e.g. \cite{PS}.) More precisely, there is a Borel space $Z$, $\sigma$-finite measure $\nu$ on $Z$ with a measurable field of Hilbert space $(\Hcal_{x})_{x\in Z}$ such that 

\begin{eqnarray}  \pi=\int_{Z}^{\oplus} \pi_{x} d\nu(x) \end{eqnarray} 
\begin{equation}\label{decompc} 
b(g)=\int_{Z} b_{x}(g) d\nu(x)\quad \textrm{ a.e. } g\in G\end{equation} such that $\nu$-almost every $x$, $\pi_{x}$ is an irreducible unitary representation of $G$ on $\Hcal_{x}$ and $b_{x}\in Z^{1}(G,\pi_{x}).$ Note that for generic $g$, we have $$||b(g)||^{2}=\int_{Z}||b_{x}(g)||^{2}d\nu(x).$$ 

We claim that it is enough to show that for any $\epsilon>0$ there is $N$ such that for any $h\in G_{gen}$ with $|h|_{G}>N$, we have $||b(h)|| <\epsilon|h|_{G}$ where $$G_{gen}=\left\{ h\in G : b(h)=\int_{Z}b_{x}(h) d\nu(x) \right\}.$$ 
Indeed there is a constant $r>0$ such that for any $g\in G$, we can find a $h\in G_{gen}$ such that $||b(g)-b(h)||<r$ and $|g|_{G}=|h|_{G}$, since Haar measure is fully supported and $||b(\cdot)||$ gives continuous norm on $G$ so that bounded on compact set.

 For irreducible representations, we have following facts due to Delorme.
\begin{theorem}[P.Delorme \cite{Delrome}]\label{del} ${}$
\begin{itemize} 
\item Let $H$ is connected simple Lie group with Lie algebra $\mathfrak{su}(n,1)$, $n\ge 2$ or $\mathfrak{so}(2,1)$ then there exist exactly 2 irreducible unitary representations, up to unitarily equivalent, with non trivial 1-cohomology. 
\item Let $H$ is connected simple Lie group with Lie algebra $\mathfrak{so}(n,1)$, $n\ge 3$ then there exist exactly one irreducible unitary representation, up to unitarily equivalent, with non trivial 1-cohomology. 
\end{itemize}
Moreover, the dimension of $1$-cohomology $H^{1}$ of these representations are all equal to $1$.
\end{theorem}

On the other hand, a non-trivial cocycle over an irreducible unitary representation vanished on a maximal compact subgroup can be interpreted by a harmonic map on the related symmetric space and we can control growth of it.
\begin{theorem}[See \cite{BHV}]\label{irredcase} Let $G$ be connected Lie group with finite center and locally isometric to $\SO(n,1)^{\circ}$ ($n\ge 2$) or $\SU(m,1)$ ($m\ge 1$). For any irreducible unitary representation $(\rho,\mathcal{L})$, assume that $1$-cocycle $c:G\rightarrow \mathcal{L}$ under $(\rho,\mathcal{L})$ satisfies following. \begin{enumerate}
\item $c|_{K}\equiv 0$
\item $c$ is not coboundary.
\end{enumerate}
Fix left $G$-invariant Riemannian metric $d=d_{G/K}$ on $G/K$. Then $$||c(g)||^{2} =C_{c,\rho,d} d(gK,K)+o(d(gK,K))\quad \textrm{as} \quad d(gK,K)\rightarrow \infty$$  
for some constant $C_{c,\rho,d}$ which depends on $c$, $\rho$ and $d$.

\end{theorem}
Recall that Riemannian metric on the symmetric space and canonical word length on $G$ is quasi-isomteric as in Lemma \ref{comp}.

Due to Theorem \ref{del}, we can find all concrete unitary representation $\sigma_{i} : G \rightarrow \mathcal{U}(\mathcal{L}_{i})$ which has non trivial $1$-cohomology. Here $i=1$ when $G$ is locally isomorphic to $\SO(n,1)$ with $n\ge 3$ and $i=1,2$ when $G$ is locally isomorphic to $\SU(n,1)$ or $SO(2,1)$.  Also fix the $1$-cocycle $c_{i}: G\rightarrow \mathcal{L}_{i}$ which is not coboundary such that $c_{i}$ satisfies $c_{i}|_{K}=0.$

In (\ref{decompc}), fix $Z_{0}\subset Z$ such that $\nu(Z\setminus Z_{0})=0$ and for any $x\in Z_{0}$, we have $b_{x}\in Z^{1}(G,\pi_{x})$. Let $Z_{1} =\{x\in Z_{0} : b_{x}\not \in B^{1}(G,\pi_{x})\}.$
We know that $Z_{1}$ is measurable set in $Z$ by \cite{Delrome} Lemma 1.1. Especially for $y\in Z_{1}$, $\pi_{y}$ admits non-trivial $1$-cohomology, so that $b_{y}$ is cohomologous to $c_{i}$ up to unitarily equivalent for some $i$. For each $y\in Z_{1}$, we can take unitary operator $L_{y}: \Lcal_{i}\rightarrow \mathcal{H}_{y}$ for some $i$ such that $L_{y}(\sigma_{i}(g)v)=\pi_{y}(g)(L_{y}(v))$ for every $v\in \Lcal_{i}$.
Therefore there exists $v_{x} \in \Hcal_{x}$ for each $x\in Z_{1}$ such that we can write $b$ as $$b=\int_{Z_{1}}b_{y}d\nu(y)+\int_{Z\setminus Z_{1}} \delta_{x} d\nu(x)$$ where $\delta_{x}\in B^{1}(G,\pi_{x})$. Now $\int_{Z\setminus Z_{1}} \delta_{x} d\nu(x) $  is direct integral of coboundaries, so it is almost coboundary (for example, see \cite{BHV}).

On the other hand,
 Theorem \ref{irredcase} says that there exist measurable function $\lambda : Z_{1} \rightarrow \Cbb$, $y\mapsto \lambda_{y}$ and  vector $v_{y}\in \Lcal_{i(y)}$ for $i(y)\in \{1, 2\}$ for each $y\in Z_{1}$ such that $$b_{y}(g)=L_{y}(\lambda_{y} c_{i(y)}(g) + \sigma_{i(y)}(g)v_{y} -v_{y})\quad g\in G_{gen}.$$ For $g\in G_{gen}$, we have
\begin{eqnarray*}
\int_{Z_{1}} ||b_{y}(g)||^{2} d\nu(y) &=& \int_{Z_{1}} ||L_{y}(\lambda_{y} c_{i(y)}(g) + \sigma_{i(y)}(g)v_{y} -v_{y})||^{2} d\nu(y)\\
&=& \int_{Z_{1}}||\lambda_{y} c_{i(y)}(g) + \sigma_{i(y)}(g)v_{y} -v_{y}||^{2} d\nu(y) 
\end{eqnarray*}
We already assume that $b|_{K}\equiv 0$, so $$\int_{Z_{1}}||b_{y}(k)||^{2} d\nu(y)=0$$ for any $k\in K\cap G_{gen}$. Furthermore, since we also choose $c_{i}$ as $c_{i}|_{K}\equiv 0$, we have $$\int_{Z_{1}} ||\sigma_{i(y)} (k)v_{y}-v_{y}||^{2}d\nu(y)=0 \quad \forall k\in K\cap G_{gen}.$$ 
This implies that $v_{y}$ is $\sigma_{i(y)}(K)$ invariant vector in $\Lcal_{i(y)}$ for a.e. $y\in Z_{1}$ since Haar measure is fully supported again. However, irreducible unitary representations $(\Lcal_{i},\sigma_{i})$ has non-trivial $1$-cohomology, we know that the only $\sigma_{i}(K)$-invariant vector in $\Lcal_{i}$ is $0$. (\cite{BHV}.) This implies that $v_{y}\equiv 0$ for a.e. $y\in Z_{1}$. Therefore, $$ \int_{Z_{1}} ||b_{y}(g_{0})||^{2} d\nu(y)= \int_{Z_{1}}|\lambda_{y}|^{2}||c_{i(y)}(g_{0})||^{2} d\nu (y).$$

We claim that $|\lambda|\in L^{2}(Z_{1},\nu|_{Z_{1}})$.  Indeed, $|\lambda_{z}|=\frac{||b_{z}(g)||}{||c_{i(z)}(g)||}$ implies that $|\lambda|$ is measurable function.  First, find $g_{0}\in G_{gen}$ such that $c_{i}(g_{0})\neq 0$ for every $i=1,2$. Let $0<m_{0}=\min\{||c_{i}(g_{0})||:i=1,2\}$. 
Then

\begin{eqnarray*}
\int_{Z_{1}}m_{0}^{2}|\lambda_{y}|^{2} d\nu(y) &\le& \int_{Z_{1}}|\lambda_{y}|^{2}||c_{i(y)}(g_{0})||^{2} d\nu (y)\\
&= & \int_{Z_{1}} ||b_{y}(g_{0})||^{2} d\nu(y) \\
&\le & \int_{Z} ||b_{x}(g_{0})||^{2} d\nu(x) \\
&<&\infty
\end{eqnarray*}
This implies $\lambda \in L^{2}(Z_{1}, \nu|_{Z_{1}})$

Now fix left $G$ -invariant Riemannian metric $d$ on $G/K$ then Theorem \ref{irredcase} says that  for any $\epsilon>0$, there is $T>0$ and constants $C_{c_{i},\sigma_{i},d}$ such that for any $d(gK,K)>T$, $$||c_{i}(g)||^{2}<C_{c_{i},\sigma_{i},d}d(gK,K)+\epsilon d(gK,K).$$ Since $i=1$ or $i=1,2$ we can find  constant $C_{1}$ such that $C_{c_{i},\sigma_{i},d}<C_{1}$ for all $i$.  Due to Lemma \ref{comp}, there is constant $C_{2}$ and $T'$ such that for any $g\in G, |g|_{G}>T'$, we have  $$||c_{i}(g)||^{2}<C_{1}d(gK,K)+\epsilon d(gK,K)<C_{1}C_{2}|g|_{G}+\epsilon C_{2}|g|_{G}.$$  Using this inequality, for $g\in G_{gen}$ with $|g|_{G}>T'$, we have
\begin{eqnarray*}
\int_{Z_{1}} ||b_{y}(g)||^{2} d\nu(y) &=& \int_{Z_{1}}|\lambda_{y}|^{2}||c_{i(y)}(g)||^{2} d\nu(y)\\
&\le & \int_{Z_{1}} |\lambda_{y}|^{2}( C_{1}C_{2}|g|_{G} +\epsilon C_{2}|g|_{G}) d\nu(y)\\ 
&\le & M( C_{1}|g|_{g}+\epsilon C_{1}C_{2}|g|_{G}) 
\end{eqnarray*} where $M=\int_{Z_{1}}|\lambda_{y}|^{2} d\nu(y).$
This implies that for any $\epsilon>0$ there is $N$ such that for any $g\in G_{gen}$ with $|g|_{G}>N$, we have $$\left|\left|\int_{Z_{1}} b_{y}(g) d\nu(y)\right|\right| < \epsilon|g|^{\frac{1}{2}+\delta}$$ for any $\delta>0$.  This implies sublinearity of $\int_{Z_{1}}b_{y}d\nu(y)$.

As we saw before, $b$ is almost cohomologous to $\int_{Z_{1}}b_{y}d\nu(y)$. This proved Proposition \ref{prophil} due to Lemma \ref{samegrowth}.
 \end{proof}
\begin{rmk} The above proof shows that if unitary representation $\pi$ has spectral gap, i.e. almost coboundary is actually coboundary, then we have more strong conclusion, \emph{$\left(\frac{1}{2}+\delta\right)$-growth} as follows. For any $\epsilon >0$ , there is $N$ such that every $g\in G$ with $|g|_{G}>N$, we have $$||b(g)||<\epsilon |g|_{G}^{\frac{1}{2}+\delta}$$ for any $\delta >0$. As a special case, we can prove that the cocycle defined by the composition of surjective homomorphism from $\Gamma$ to $\Zbb$ with return cocycle, $G\times G/\Gamma\rightarrow \Gamma\twoheadrightarrow \Zbb$ has $\left(\frac{1}{2}+\delta\right)$ growth provided that the $\Gamma$ is cocompact lattice. Indeed, the integrability is automatically satisfied and $G$ action on $G/\Gamma$ has a spectral gap.

\end{rmk}
Now, we are ready to prove Theorem \ref{Hilbertcase}.
\begin{proof}[Proof of Theorem \ref{Hilbertcase}] 
We want to prove that for any $\epsilon>0$, there is $N\in \Nbb$ such that for any $g \in G$ with $|g|_{G}>N$, $||b(g)||_{2}<\epsilon|g|_{G}.$

\underline{\textbf{Step 1.}} Shalom\textquoteright s super-rigidity theorem for reduced cohomology 
allow us to reduce problems to each simple factors.
\begin{theorem}[\cite{Shalom}Theorem 3.1]\label{shalG}
Retaining notations, $b$ is  almost cohomologous to a cocycle of the form $b_{0}+b_{1}+\dots+b_{n}$, where $b_{0}$ takes values in the subspace $G$-invariant vectors $\mathcal{H}_{0}$, and each $b_{i}$ for $1\le i \le n$ is a cocycle depending only on $G_{i}$, and takes values in a $G$-invariant subspace $\mathcal{H}_{i}$, on which is $G$-unitary representation which factors through a unitary representation of $G_{i}$.  Note that this theorem also holds for a product of at least two LCSC groups.
\end{theorem}
We can think $b=b_{0}+b_{1}+\dots+b_{n} +b_{r}$ where $b_{0},b_{1},\dots,b_{n}$ as above and for some almost coboundary $b_{r}:G\rightarrow \Hcal$.

\underline{\textbf{Step 2.}} Due to Lemma \ref{samegrowth}, we only need to care about 1-cocycle which is almost cohomologous to $b$. Especially, it is enough to show that $b_{0}+b_{1}+\dots+b_{n}:G \rightarrow \Hcal$ has sublinear growth as $1$-cocycle defined on $G$. 
\begin{enumerate}
\item For $b_{0}$ : We know that $b_{0}$ takes values at $G$-invariant vectors. This implies that $b_{0}:G\rightarrow \Hcal$ is a homomorphism. Since $\Hcal$ is abelian, $b_{0}$ should factor through $G/[G,G]\rightarrow \Hcal.$ Since $G$ is unimodular, it should be trivial. \\

\item For $b_{i}$, $i\ge 1$ : For the cocycle $b_{i}:G\twoheadrightarrow G_{i}\rightarrow \mathcal{H}_{i}$ as above, $b_{i}$ has sublinear growth as $1$-cocycle defined on $G_{i}$,  due to Proposition \ref{prophil}.

\end{enumerate}
Indeed, $b_{i}$ has sublinear growth  as $1$-cocycle defined on $G$ as follows.
 \begin{lem}\label{glob} $b_{i}:G\twoheadrightarrow G_{i} \rightarrow \mathcal{H}_{i}$ has sublinear growth as $1$-cocycle defined on $G$. 
 \end{lem}
\begin{proof}[Proof of Lemma \ref{glob}.]
For any $g\in G$, let denotes $g^{i}=pr_{i}(g)$ where $pr : G\rightarrow G_{i}$ is projection onto $G_{i}$. Now choose any $g\in G$ and write $g=g_{1}\dots g_{l}$ with minimal word length with respect to generating set $Q=Q_{1}\times \dots \times Q_{n}$. In this case, $|g|_{G}=l$. Then we can rewrite $$g=(g_{1}^{1}\dots g_{l}^{1})(g_{1}^{2}\dots g_{l}^{2})\dots(g_{1}^{n} \dots g_{l}^{n})$$  where $g_{j}^{k} \in Q_{k}$ for each $j=1,\dots, l$ and $k=1,\dots ,n$. This comes from the fact that elements in $G_{k}$ and $G_{l}$ commute each other if $k\neq l$. Now $g^{i}=g_{1}^{i}\dots g_{l}^{i}$ so that $|g^{i}|_{i} \le l$. This gives the proof of the fact that for any $g\in G$, $|g^{i}|_{i} \le |g|_{G}$. For any $\epsilon>0$ choose $l(\epsilon) \in \Nbb$ such that 
$$||b_{i}(g^{i})|| <\epsilon |g^{i}|_{i} \quad \textrm{ if } \quad g^{i}\notin (Q_{i})^{l(\epsilon)}$$
We can find such $l(\epsilon)$ because we assume that the $b_{i}$ has sublinear growth as $1$-cocycle defined on $G_{i}$. 
On the other hand, because $Q_{i}$ is compact, we can find $$M(\epsilon)=\max(||b_{i}(q)||:q\in (Q_{i})^{l(\epsilon)})<\infty$$
Find $L(\epsilon) \in \Nbb$ such that $L(\epsilon) >\frac{M(\epsilon)}{\epsilon}$.
For any $g\in G$ with $|g|_{G} >L(\epsilon)$, either 
\begin{enumerate}
\item $g^{i}\in (Q_{i})^{l(\epsilon)}$ or
\item $g^{i}\notin (Q_{i})^{l(\epsilon)}$
\end{enumerate}
For the first case, since $g^{i}\in (Q_{i})^{l(\epsilon)}$, $$||b_{i}(g)||_{2}=||b_{i}(g^{i})||_{2} \le M(\epsilon) <\epsilon \times L(\epsilon) < \epsilon |g|_{G}$$ by construction and due to the part of the Theorem \ref{shalG}. On the other hand, for the second case, since $g^{i}\notin (Q_{i})^{l(\epsilon)}$, $$||b_{i}(g)||_{2}=||b_{i}(g^{i})||_{2} <\epsilon|g^{i}|_{i} \le \epsilon |g|_{G}$$
Therefore, for any cases, $b_{i}$ has sublinear growth as $1$-cocycle defined on $G$.  \end{proof}

\underline{\textbf{Step 3.}}  Finally, we know that the sum of two sublinear growth cocycle has sublinear growth, due to Lemma \ref{summ}.
More precisely, Lemma \ref{summ} and \ref{glob} give us that $b_{1}+b_{2}+\dots+b_{n}$ has sublinear growth as cocycle defined on $G$. 
So we proved Theorem \ref{Hilbertcase}. \end{proof}

For a weakly irreducible lattice $\Gamma$ case, we also prove slightly general theorem as follows. Recall that $|\cdot|_{w}$ is symmetric canonical norm on $\Gamma$.
\begin{theorem}\label{RvalueGamma}

Let $\Gamma$ satisfy \ref{std} and $(S,\mu)$ be the standard probability space with measure preserving $\Gamma$-action. Let $\delta:\Gamma\times S\rightarrow \Rbb$ be  $L^{2}$-integrable cocycle in the following sense. $$(L^{2}) :\quad \forall \gamma \in \Gamma,\quad |\delta(\gamma,-)|\in L^{2}(S)$$

Then for any $\epsilon>0$, there is $N>0$ such that
$$||\delta(\gamma,-)||_{L^{2}(S)}=\left( \int_{S} |\delta(\gamma,s)|^{2} d\mu(s)\right)^{\frac{1}{2}} \le \epsilon|\gamma|_{w}$$  for $|\gamma|_{w}>N$. Especially, for every $\gamma\in \Gamma$, we have $$\lim_{m\rightarrow\pm\infty} \frac{|\delta(\gamma^{m},x)|}{m}=0$$ almost every $x\in S$.

\end{theorem}
We need following theorem corresponding to Theorem \ref{Hilbertcase}.
\begin{theorem}\label{Hilgamma} Let $\Gamma$ satisfy \ref{std}. Let $(\pi,\Hcal)$ be unitary representation of $\Gamma$ and $b:\Gamma\rightarrow \Hcal$ be $1$ cocycle. Then $b$ has sublinear growth.
\end{theorem}
Exactly same arguments with $D=G$ case, it is enough to prove Theorem \ref{Hilgamma}.

\begin{proof}[Proof of Theorem \ref{Hilgamma}] ${}$

\textbf{$\Gamma$ is an irreducible lattice case:} We will prove when $\Gamma$ is an irreducible lattice first. When $G$ is a simple higher rank algebraic group, in other words $n=1$, $\Gamma$ has property (T). In this case, every $1$-cocycle is bounded so we are done. Therefore, we may assume that $G$ has at least $2$ simple factors, in other words $n\ge 2$. We want to prove that for any $\epsilon>0$, there is $N\in \Nbb$ such that for any $\gamma \in \Gamma$ with $|\gamma|_{w}>N$, $||b(\gamma)||_{2}<\epsilon|\gamma|_{w}$.

\underline{\textbf{Step 1.}} For the irreducible lattice $\Gamma$, Shalom\textquoteright s super-rigidity for reduced cohomology (See \cite{Shalom} Theorem 4.1)   states follow. 
\begin{theorem}[\cite{Shalom}]\label{shalGam}
Let $\Gamma, G_{i}$ and $G$ be as above, $b$ is  almost cohomologous to a cocycle of the form $b_{0}+b_{1}+\dots+b_{n}$, where $b_{0}$ takes values in the subspace of $\Gamma$-invariant vectors $\mathcal{H}_{0}\subset L^{2}(S)$, and each $b_{i}$ for $1\le i \le n$, takes values in a $\Gamma$-invariant subspace $\mathcal{H}_{i}\subset L^{2}(S)$, on which the $\Gamma$ representation extends continuously to $G$-representation which factors through a representation of $G_{i}$. Furthermore, each $b_{i}$ extends continuously to a cocycle depending only on $G_{i}$ for $1\le i \le n$ and $b_{0}$ extends continuously to a cocycle of $G$.
\end{theorem}
Now due to above theorem, we have $b=b_{0}+b_{1}+\dots+b_{n} +b_{r}$ where $b_{0},b_{1},\dots,b_{n}$ as above and for some almost coboundary $b_{r}:\Gamma\rightarrow \Hcal$. We will think $b_{0},b_{1},\dots,b_{n}$ and $b$ as 1-cocycle defined on $G$, i.e. 
$b_{0}:G\rightarrow \Hcal$ and 
$b_{i}:G\twoheadrightarrow G_{i} \rightarrow \Hcal$ for $i=1,\dots,n. $

\underline{\textbf{Step 2.}} This step is same as the proof when $D=G$ case.\\

\underline{\textbf{Step 3.}} Finally, we claim that the $1$-cocycle $b$ has sublinear growth as the cocycle defined on $\Gamma$ using the following argument as we desired. 
Indeed, for cocycle $b_{1}+b_{2}+\dots+b_{n} : G\rightarrow \Hcal$, then we have that $b_{1}+b_{2}+\dots+b_{n}\restrict{\Gamma} : \Gamma \rightarrow \Hcal$ has a sublinear growth as a $1$-cocycle defined on $\Gamma$ by the following Lemma. 

\begin{lem}\label{finlem}
For a cocycle $c:G\rightarrow \Hcal$ which has a sublinear growth, then $c\restrict{\Gamma} : \Gamma \rightarrow \Hcal$ has sublinear growth as $1$-cocycle defined on $\Gamma$.
\end{lem}
\begin{proof}[Proof of Lemma \ref{finlem}.]
Now assume that $c:G \rightarrow \mathcal{H}$ has sublinear growth as $1$-cocycle defined on $G$. For any $\gamma \in \Gamma$, write $\gamma=a_{1}a_{2}\dots a_{l}$ with minimal word length, where $a_{1},\dots,a_{l}\in W$. In this case $|\gamma|_{w}=l$. Since we may assume that $W\subset Q$, $|\gamma|_{G} \le l=|\gamma|_{w}$. Therefore, for any $\gamma\in \Gamma$, $|\gamma|_{G}\le |\gamma|_{w}$. Now fix arbitrary $\epsilon>0$ then there exist $N(\epsilon) \in \Nbb$ such that for any
$$g\in G-(Q)^{N(\epsilon)}=G-(Q_{1}\times Q_{2}\times \dots \times Q_{n})^{N(\epsilon)} \Rightarrow ||c(g)||<\epsilon|g|_{G}$$
Equivalently, $|g|_{G} >N(\epsilon) \Rightarrow ||c(g)||<\epsilon|g|_{G}$. Since we assume that $c$ has a sublinear growth as a $1$-cocycle defined on $G$, we can find such a $N(\epsilon)$. Now take $A=Q^{N(\epsilon)}\cap \Gamma$. Since $\Gamma$ is discrete, $Q^{N(\epsilon)}$ is compact, $A$ is finite. So, we have $P=\max(|\gamma|_{w} : \gamma \in A) <\infty$.
Moreover, for any $\gamma\in \Gamma$ with $|\gamma|_{w}>P$, 
\begin{eqnarray*}
|\gamma|_{w}> P &\Rightarrow & \gamma \notin A=\Gamma \cap Q^{N(\epsilon)}\\
 &\Rightarrow & \gamma \notin Q^{N(\epsilon)} , \gamma \in G\\
&\Rightarrow   & ||c(\gamma)||<\epsilon |\gamma|_{G} \le \epsilon |\gamma|_{w}\\ 
\end{eqnarray*} 
This proves the lemma. \end{proof}

Using Lemma \ref{samegrowth} and \ref{summ} again, $b=b_{1}+\dots+b_{n}+b_{r} : \Gamma \rightarrow \Hcal$ has a sublinear growth as a $1$-cocycle defined on the $\Gamma$.

\textbf{$\Gamma$ is a weakly irreducible lattice case:} Now we can prove the theorem for a weakly irreducible lattice $\Gamma$. Let $\Gamma$ be a weakly irreducible lattice in $G$.  Recall that we can find a finite index subgroup $\prod_{j=1}^{k} \Gamma_{j}$ where $\Gamma_{j}$ is an irreducible lattice in some semisimple normal subgroup of $G$. Note that if $b\restrict{\prod_{j=1}^{k}\Gamma_{j}}$  has a sublinear growth, then $b$ has a sublinear growth since $\prod_{j=1}^{k}\Gamma_{j}$ is a finite index subgroup. Also, if $k=1$ then $\Gamma$ is irreducible. Therefore, we may assume that $k\ge 2$ and $\Gamma=\prod_{j=1}^{k}\Gamma_{j}$ for some irreducible lattices $\Gamma_{j}$. In this case, we can use Theorem \ref{shalG} for $\prod_{j=1}^{k}\Gamma_{j}$ since the theorem holds for product of LCSC groups. Therefore, it is enough to prove that any $1$-cocycle has sublinear growth for each irreducible lattices $\Gamma_{j}$ using same arguments in the proof of Theorem \ref{Hilbertcase}. This case is already proved as above.  Therefore, we prove Theorem \ref{Hilgamma} for a general weakly irreducible lattice $\Gamma$. 
\end{proof}

\begin{rmk} Let $D=G$ or $\Gamma$ satisfy \ref{std}. 
The above theorems allow us to prove that an amenable algebraic group valued cocycle over $D$ action is subexponential under $L^{2}$-integrability. 
We will not use this results in the remaining paper. 


\end{rmk}

\section{Dynamical Cocycle super-rigidity}\label{sec:dynsuperrigid}
Now we are in the position to prove our main ingredients, dynamical cocycle super-rigidity theorem.

Let $D=G$ or $\Gamma$ satisfy \ref{std} with $\textrm{rank}_{\Rbb}(G)\ge 2$. Let $(S,\mu)$ be the standard $D$-space. Assume that $D$ action satisfies \ref{std2}.

Let $\Fbb$ be a connected reductive algebraic group, then there are connected semisimple algebraic groups $\Hbb$, $\Kbb$ and central torus $\Tbb$ so that \begin{enumerate}
\item all of them defined over $\Rbb$ and $\Fbb$ is almost direct product of $\Hbb$, $\Kbb$ and $\Tbb$, i.e. 
$\Fbb=\Hbb\cdot \Kbb\cdot \Tbb.$ 
\item $H=\Hbb(\Rbb)$ is semisimple with non-compact factor. In other words, every almost $\Rbb$-simple factors of $\Hbb$ is $\Rbb$-isotropic,
\item $K=\Kbb(\Rbb)$ is compact semisimple. In other words, every almost $\Rbb$-simple factors of $\Kbb$ is $\Rbb$-anisotropic.
\end{enumerate}

Recall that the following is a direct consequence of Zimmer's cocycle super-rigidity theorem in the semisimple algebraic hull. It is proved in the \cite{FM1} Theorem 3.16 implicitly. 
	\begin{theorem}\label{thm:redzim} 
	Under same notations as above, let $\alpha$ be the measurable cocycle $\beta:D\times S\rightarrow F$ with an algebraic hull $F=\Fbb(\Rbb)$.  Then there is a rational homomorphism $\pi:\Gbb\rightarrow \Hbb$ defined over $\Rbb$, a cocycle $$\Esc:D\times S\rightarrow A=\left(Z(\Hbb)\cdot\Kbb\cdot\Tbb\right)(\Rbb)$$ taking values in an amenable group $A$, a measurable map $\phi:D\rightarrow F$ such that 
\begin{enumerate}
\item[(a)] for any $g\in D$, $\beta(g,x)=\phi(g.x)^{-1}\pi(g)\Esc(g,x)\phi(x)$ almost every $x\in S$,
\item[(b)] $\pi(G)$ and $A$ commutes. Especially, $\pi$ and $\Esc$ commutes.
\end{enumerate}
	\end{theorem}
	\begin{rmk} There are remarks on the above theorem.
		\begin{enumerate}
		\item Indeed, one can show that the amenable group valued error term $\Esc$ is indeed $L^{2}$-integrable provided that $\beta$ is $L^{2}$-integrable. Furthermore, the cocycle $\Esc$ is itself subexponential. 
		\item When $G$ has property (T), one can deduce Zimmer's cocycle super-rigidity under reductive algebraic hull condition from Theorem \ref{thm:redzim} due to the fact that the amenable group valued cocycle is cohomologous to the compact group valued cocycle when the action is ergodic. When $G$ does not have property (T), we can not appeal to use that fact. This is the main reason that we can \emph{not} expect compact error. 
		\end{enumerate}
	\end{rmk}

\subsection{Dynamical cocycle super-rigidity}

The following is the dynamical cocycle super-rigidity when $G$ is the source group.

\begin{theorem}[Dynamical cocycle super-rigidity for $G$.]\label{thm:superrigidG}
Let $G$ satisfy \ref{std} with $\textrm{rank}_{\Rbb}(G)\ge 2$. Let $(S,\mu)$ be a weakly irreducible $G$-ergodic standard probability space. Let $\beta:G\times S\rightarrow \GL(d,\Rbb)$ be $L^{2}$-integrable cocycle, i.e. $$\forall g\in G,\quad \ln||\beta(g,-)||\in L^{2}(S).$$  After replacing $S$ to finite extension $\widehat{S}=S\times_{i}I$,  there is a measurable map $\phi: \widehat{S} \rightarrow \GL(d,\Rbb),$ a rational homomorphism $\pi: \Gbb\rightarrow \GL(d,\Cbb),$ an amenable real algebraic group $A_{0}$ and $A_{0}$-valued cocycle $\Esc_{\circ}:G\times \widehat{S} \rightarrow A_{0}$ such that 
	\begin{enumerate}
		\item[(a)] We have for any $g\in G$, $$\beta(g,x)=\phi(g.(x,j))\pi(g)\Esc_{\circ}(g,x,j)\phi(x,j)^{-1}$$ almost every $(x,i)\in \widehat{S}$.

		\item[(b)] For any $g\in G$, $\Sp(\beta_{g})=\Sp(\pi(g))$. Especially, the Lyapunov spectrum does not depend on $x$. 
		\item[(c)] For almost every $x\in S$ and all $j\in I$,  the Lyapunov subspaces $E_{\lambda}(x)$ of $\beta_{g}$ corresponding the Lyapunov exponent $\lambda$ is $\phi(x,j)^{-1}E_{\lambda}$ where $E_{\lambda}$ is the Lyapuonv subspace of $\pi(g)$ corresponding the Lyapunov exponent $\lambda$. 
	\item[(d)] We can write $\Esc_{\circ}=\Esc\cdot u$ where $\Esc$ is the cocycle that takes values on an amenable subgroup $A$ in the Levi subgroup of the algebraic hull and $u$ is the twisted cocycle that takes values on the unipotent radical of the algebraic hull. Furthermore, $\pi(G)$ commutes with $A$.
	\item[(e)]  $\pi(G)$ commutes with the unipotent radical of the algebraic hull so that $\pi$ and $u$ commutes. In particular, $\pi$ and $\Esc_{0}$ commutes. 
	\end{enumerate}
\end{theorem}
\begin{rmk}\label{rmk:errorint}
In the dynamical cocycle super-rigidity theorems, we do \emph{not} claim that $\Esc_{\circ}$ is integrable or quasi-integrable. Even though $\Esc_{\circ}$ may not be integrable, the statements in the theorems are make sense, since $\beta$ and $\pi$ are integrable.
\end{rmk}
\begin{rmk} In general, $u$ may not satisfy cocycle equation since $u$ may not commute with $\Esc$. The measurable map $u$ will be a twisted cocycle as in \cite{F}.
\end{rmk}
We will start to prove Theorem \ref{thm:superrigidG}. First of all, we need to define finite extension $\widehat{S}$ precisely. Let $G$ satisfy \ref{std} with $n\ge 2$. As in the theorem, we assume that $G$ acts irreducibly on the standard probability space $(S,\mu)$.  We will use $G$ action on various spaces just $(g,x)\mapsto g.x$ since it will be clear what action is used in the context.

Let $\beta:G\times S\rightarrow \GL(d,\Rbb)$ is $L^{2}$-integrable cocycle. Denote an algebraic hull of $\beta$ as $\widehat{L}=\widehat{\Lbb}(\Rbb)$ where $\widehat{\Lbb}$ is a real algebraic group. Then, we can find a measurable map $\phi_{1}':S\rightarrow \GL(d,\Rbb)$ such that $\beta^{\phi_{1}'}:G\times S\rightarrow \widehat{L}$ is a measurable minimal cocycle. Let $\Lbb=\widehat{\Lbb}^{0}$ be the algebraic connected component of identity for $\widehat{\Lbb}$ and $L=\Lbb(\Rbb)$.  We have a cocycle $pr\circ\alpha=i:G\times S\rightarrow I$ where $I=\widehat{L}/L$ is  a finite set and $pr:\widehat{L}\rightarrow I$ is the natural projection. Now we have a finite extension $\widehat{S}=S\times_{i} I$ and $G$ action $g.(x,j)=(g.x, \beta^{\phi_{1}'}(g,x)\cdot j)$ on $\widehat{S}$.

Define $\tilde{\beta}:G\times \widehat{S} \rightarrow \GL(d,\Rbb)$, $\tilde{\beta}(g,x,j)=\beta(g,x)$ for any $g\in G$ and almost every $(x,j)\in \widehat{S}$. 

Using facts about finite extension in the section \ref{precoc}, we know that $G$ action on $\widehat{S}$ is still irreducible and $\widehat{\beta}$ has algebraic hull $L=\Lbb(\Rbb)$. More precisely, fix (finite valued) section $\phi_{2}':I\rightarrow \widehat{L}$ of $pr$. Then the cocycle 
$\widehat{\beta}:G\times \widehat{S} \rightarrow L$, $$\widehat{\beta}(g,x,j)=\phi_{2}'(\beta^{\phi_{1}'}(g,x)\cdot j)^{-1} \beta^{\phi_{1}'}(g,x)\phi_{2}'(j)$$ is  minimal. Define measurable maps $\phi_{1}:\widehat{S} \rightarrow \GL(d,\Rbb)$ and $\phi_{1}(x,j)=[\phi_{1}'(x)\phi_{2}'(j)]^{-1}$ then it can be easily checked that $$\tilde{\beta}(g,x,j)=\beta(g,x)=\phi_{1}(g.(x,j))^{-1}\widehat{\beta}(g,x,j)\phi_{1}(x,j)$$ for all $g\in G$ and almost every $(x,j)$ in $\widehat{S}$.

 Now fix a Levi component that is a connected reductive $\Rbb$-subgroup $\Fbb<\Lbb$ such that $\Lbb=\Fbb\ltimes \Ubb $ where $\Ubb$ is the unipotent radical of $\Lbb$. Further, we know that $\Lbb(\Rbb)=\Fbb(\Rbb)\ltimes \Ubb(\Rbb).$ Denote $F=\Fbb(\Rbb)$, $U=\Ubb(\Rbb)$, $\llie=\textrm{Lie}(L)$ and natural projection $p_{F}:L\rightarrow F$.

We have a cocycle $f:G\times \widehat{S}\rightarrow F$ and a measurable map $u':G\times \widehat{S}\rightarrow U$ such that  $$\widehat{\beta}(g,x,j)=f(g,x,j)\cdot u'(g,x,j).$$ Using Theorem \ref{thm:redzim}, there is a rational homomorphism $\pi:\Gbb\rightarrow \Fbb<\Lbb$ defined over $\Rbb$, a cocycle $$\Esc:G\times \widehat{S}\rightarrow Z(\Hbb)\cdot\Kbb\cdot\Tbb(\Rbb),$$ a measurable map $\phi:\widehat{S}\rightarrow \overline{\pi(G)}^{Z}(\Rbb)<F$  and a bounded measurable map $\psi:\widehat{S}\rightarrow F$ such that the cocycle $\widehat{\beta}$ can be expressed as for any $g\in G$, $$\widehat{\beta}(g,x,j)=\psi(g.(x,j))^{-1}\phi(g.(x,j))^{-1}\pi(g)\Esc(g,x,j)\phi(x,j)\psi(x,j)u'(g,x,j)$$ for almost every $(x,j)\in \widehat{S}$.
 Define a measurable map $u:G\times \widehat{S}\rightarrow U$ as $$u(g,x,j)=\phi(x,j)\psi(x,j)u'(g,x,j)\psi(x,j)^{-1}\phi(x,j)^{-1}.$$ Note that the $u$ still takes values at $U$, since $U$ is normal subgroup of $L$. 
 Direct calculation shows that
\begin{align*} 
\widehat{\beta}(g,x,j)&= \psi(g.(x,j))^{-1}\phi(g.(x,j))^{-1} \pi(g) \Esc(g,x,j) \phi(x,j)\psi(x,j)u'(g,x,j)\\
&=\psi(g.(x,j))^{-1}\phi(g.(x,j))^{-1}\pi(g)\Esc(g,x,j)u(g,x,j)\phi(x,j)\psi(x,j)
\end{align*}
The above calculation shows that $\pi\cdot \Esc\cdot u$ is cocycle. Define $\Esc_{\circ}=\Esc\cdot u$. So far we don't know that the $\Esc_{\circ}$ is cocycle. If $\pi(G)$ commutes with $U$ then, $\Esc_{\circ}=\Esc \cdot u$ is indeed a cocycle.

Now proof of the theorem \ref{thm:superrigidG} can be divided into Part $1$ and $2$.

\begin{enumerate}
\item[\textbf{Part 1.}] For any $g\in G$, $\Sp(\beta_{g})=\Sp(\pi(g))$. 
\item[\textbf{Part 2.}]  $\pi(G)$ commutes with $U$ so that $\pi$ commutes with $\Esc_{\circ}$. This implies that $\Esc_{\circ}$ is indeed cocycle.
\end{enumerate}
\begin{rmk}\label{rmk:errorterm}
Note that Part 1 and 2 are enough to prove the theorem. We already get formula appeared in item (a).
We also already get informations about $\Esc_{\circ}$. The cocycle $\Esc_{\circ}$ can be written as a product of $F$-valued cocycle $\Esc$ and $U$-valued measurable map $u$. Furthermore, $\Esc$ and $\pi$ commutes by Theroem \ref{thm:redzim}. This is the item (d). Part 1 will prove the items (b) and (c).  Part 2 shows that $\Esc_{\circ}$ is indeed cocycle and $\pi$ and $\Esc_{\circ}$ commutes. That is the item (e).
\end{rmk}

After establishing \textbf{Part 1.}, one can prove \textbf{Part 2.} using same strategy with \cite{FM1} Theorem 3.11. In other words, using Lyapunov spectrum, one can prove followings.
\begin{theorem}[Cf. \cite{FM1} Theorem 3.11.]\label{comm} 
In the setting in Theorem \ref{thm:superrigidG} and above notations, assume $\pi(G)$ does not commute with $U$.  Then there is an integer $0<k<\dim(\llie)$ and measurable map $$\Phi:P\backslash G \times S\rightarrow \textrm{Gr}_{k}(\llie)$$  such that 
\begin{enumerate}
\item $\phi$ is  $\beta^{\phi}$-equivariant, i.e.  $$\Phi([hg^{-1}],g.x)=\Ad_{\llie}(\beta^{\phi}(g,x))(\Phi([h],x))$$ for all $g\in G$, $[h]\in P\backslash G$ and almost every $x\in S$ 
\item a pointwise stabilizer of the image does not contain all of $U$.
\end{enumerate}
\end{theorem}
This gives proof of \textbf{Part 2.} as in the \cite{FM1} Theorem 3.10. Since the proof has same structure, we left details for the reader. Indeed, the proof is much simpler using \textbf{Part 1.} because \textbf{Part 1.} already provides Lyapunov subspaces.

Now we show the \textbf{Part 1.}.
\begin{proof}[Proof of \textbf{Part 1.}] Note that the finite extension does not effect to Lyapunov spectrum. Therefore, for simplicity, assume that $\widehat{\Lbb}$ is connected algebraic group so that we will use the simplified notations such as $\Lbb=\widehat{\Lbb}$, $S=\widehat{S}$, etc.

Let $\phi_{1}:S\rightarrow \GL(d,\Rbb)$ be a measurable map such that $\beta^{\phi_{1}}:G\times S\rightarrow L$ is a minimal cocycle as before.  We will use same notations as in the above and the theorem \ref{thm:redzim}. There is a measurable map $\phi_{2}:S\rightarrow F,$ a rational homomorphism $\pi:\Gbb\rightarrow \Hbb<\GL(\Cbb^{d}),$ a $Z(\Hbb)\cdot\Kbb\cdot\Tbb(\Rbb)$ valued cocycle $$\Esc:G\times S\rightarrow Z(\Hbb)\cdot\Kbb\cdot \Tbb(\Rbb)$$ and a measurable map $u':G\times S\rightarrow U$ such that for any $g\in G$, $$\beta^{\phi_{1}}(g,x)=\phi_{2}(g.x)^{-1}\pi(g)\Esc(g,x)\phi_{2}(x)u'(g,x)$$ for almost every $x\in M$. 
Define a measurable map $u:G\times S\rightarrow U$ and $\phi:S\rightarrow \GL(d,\Rbb)$ as $u(g,x)=\phi_{2}(x)u'(g,x)\phi_{2}(x)^{-1}$ and $\phi(x)=\phi_{2}(x)\phi_{1}(x).$  Then, we have $\beta^{\phi}=\pi\cdot\Esc\cdot u$.

Find vector space filteration $$0=V^{0}\subset V^{1}\subset\dots\subset V^{k}=\Rbb^{d}$$ such that 
\begin{enumerate}
\item  each $V^{i}$ is $L$ invariant and
\item $U$ acts trivially on $V^{i+1}/V^{i}$, for $i=0,1,\dots,k-1$. 
\end{enumerate}
In other words, for each $i=0,\dots,k-1$, we have homomorphism $$\pi^{i}:L\rightarrow \GL(V^{i+1}/V^{i})$$ so that $\pi^{i}(L)$ is reductive. This can be shown inductively. (e.g. \cite{Feres} Proposition 6.5.7.) Since $F$ is reductive, for each $i$, we can decompose $V^{i+1}/V^{i}$ with $F$-irreducible modules. More precisely, there are irreducible $F$-submodules, $$\overline{W^{i}_{1}},\dots, \overline{W^{i}_{m(i)}} <V^{i+1}/V^{i}$$ such that 
$$ V^{i+1}/V^{i}=\bigoplus_{j=1}^{m(i)} \overline{W^{i}_{j}}.$$
Find an ordered basis $\left\{\overline{\epsilon_{j,1}^{i}},\dots \overline{\epsilon_{j,l(i,j)}^{i}}\right\}$ of $\overline{W^{i}_{j}}$ and lift it to $\Rbb^{d}$ so that $\Bsc=\left\{\epsilon_{j,s}^{i}\right\}_{i,j,s}$
is an ordered basis of $\Rbb^{d}$ such that the projection of $\epsilon_{j,s}^{i}$ in to the $V^{i+1}/V^{i}$ is $\overline{\epsilon_{j,s}^{i}}$ for any $i=0,\dots,k-1$, $j=1,\dots,m(i)$ and $s=1,\dots,l(i,j)$.

Using basis $\Bsc$ of $\Rbb^{d}$, we can use Iwasawa decomposition (or Gram-Schmidt orthogonalization) on $\GL(d,\Rbb)$ on $\phi$. There are measurable map $$o:S\rightarrow O(d,\Rbb), a:S\rightarrow \textrm{Diag.} \textrm{ and }n:S\rightarrow N$$ such that $\phi(x)=o(x)a(x)n(x)$ for $\mu$-almost every $x\in S$ where $\textrm{Diag.}$ is subgroup of diagonal matrix with positive diagonal entries and $N$ is subgroup of upper triangular matrix with $1$ in diagonal entries. Then for any 
$g\in G$, $\mu$-almost every $x\in S$ we have $$\beta^{o}(g,x)=a(g.x)n(g.x)\pi(g)\Esc(g,x)u(g,x)n(x)^{-1}a(x)^{-1}.$$
Note that this construction makes that $\beta^{0}$ preserves $\bigoplus_{j=1}^{r}\overline{W_{j}^{i}}$ for any $r=1,\dots, m(i)$.

For each $i,j$ and $s$, define a measurable map $a^{i}_{j,s}:S\rightarrow \Rbb$ as $$a^{i}_{j,s}(x)=\ln\left(\frac{|a(x)\epsilon^{i}_{j,s}|}{|\epsilon^{i}_{j,s}|}\right). $$ Then we have $$a(x)\left(\epsilon^{i}_{j,s}\right)=e^{a^{i}_{j,s}(x)}\epsilon^{i}_{j,s}$$ for each $i,j$ and $s$ by construction of $a(x)$.
For notational convention, denote the cocycle $\Omega:G\times S\rightarrow L$ as $$\Omega=\pi \cdot \Esc \cdot u.$$

The cocycle $\beta^{o}$ preserves $V^{i+1}/V^{i}$ for every $i$ by our construction of basis. Since $o$ is bounded measurable map, the cocycle $\beta^{o}$ is still $L^{2}$-integrable. Furthermore, Lyapunov spectrums of $\beta_{g_{0}}^{o}$ and $\beta_{g_{0}}$ coincide for any $g_{0}\in G$.

Recall that we have the decomposition $V^{i+1}/V^{i}=\bigoplus_{j=1}^{m(i)}\overline{W^{i}_{j}}$.
Since $W^{i}_{j}$ is an irreducible $F$-module, the center $Z(F)$ of $F$ acts on $\overline{W^{i}_{j}}$ as a scalar multiplication by Schur's lemma. Furthermore, $$\Omega\restrict{\overline{W^{i}_{j}}}= \pi\restrict{\overline{W^{i}_{j}}}\cdot \Esc\restrict{\overline{W^{i}_{j}}}$$ since $U$ acts on $V^{i+1}/V^{i}$ trivially. Denote the $\Rbb$-valued cocycle $\lambda_{j}^{i}:G\times S\rightarrow \Rbb$ as 
	$$\lambda_{j}^{i}(g,x)=\frac{1}{l(i,j)}\ln\left|\det\left(\Omega(g,x)\restrict{\overline{W^{i}_{j}}}\right)\right|$$  
for each $j=1,\cdots, m(i)$. Recall that the $l(i,j)$ is the dimension of $\overline{W^{i}_{j}}$. These cocycles $\lambda_{\bullet}^{i}$ encode informations about the $Z(F)$ action. Moreover, $\Esc$ is the cocycle valued at a $\Rbb$-points of almost direct product of a central torus, a compact semisimple group and a finite group. This implies that the cocycle $$e^{-\lambda_{j}^{i}(-,-)}\Omega(-,-)\restrict{\overline{W_{j}^{i}}} :G\times S\rightarrow \GL(\overline{W_{j}^{i}})$$ $$(g,x)\mapsto e^{-\lambda_{j}^{i}(g,x)}\Omega(g,x)\restrict{\overline{W_{j}^{i}}}$$ is a bounded cocycle. Indeed, we can find the compact group valued cocycle $\Ksc_{j}^{i}$ such that $$\Omega(g,x)\restrict{\overline{W_{j}^{i}}}=e^{\lambda_{j}^{i}(g,x)}\pi\restrict{\overline{W_{j}^{i}}}(g)\Ksc_{j}^{i}(g,x).$$ Note that $\pi\restrict{\overline{W_{j}^{i}}}$ and $\Ksc_{j}^{i}$ still commute.

We claim that the cocycle $$\widetilde{\lambda_{j}^{i}}:=(\lambda_{j}^{i})^{\frac{1}{l(i,j)}\sum_{s=1}^{l(i,j)}a^{i}_{j,s}}:G\times M\rightarrow \Rbb,$$ $$(g,x)\mapsto \widetilde{\lambda_{j}^{i}}(g,x)= \lambda_{j}^{i}(g,x)+\frac{1}{l(i,j)}\sum_{s=1}^{l(i,j)}a^{i}_{j,s}(g.x)-\frac{1}{l(i,j)}\sum_{s=1}^{l(i,j)}a^{i}_{j,s}(x)$$ is the $L^{2}$-integrable cocycle for any $i$ and $j$.

We will use induction on $j$. \begin{enumerate}
\item $j=1$ case ;  Due to our construction, the $\beta^{o}$ preserves $\overline{W^{i}_{1}}$. Since the $\beta^{o}$ is $L^{2}$-integrable, we know that $\beta^{o}\restrict{\overline{W^{i}_{1}}}$ is a still $L^{2}$-integrable cocycle. Especially, the cocycle $\Rbb$-valued cocycle $$\ln\left|\det\left(\beta^{o}\right)\right|:G\times S\rightarrow \Rbb$$ is a $L^{2}$-integrable cocycle. This implies that the cocycle $$l(i,1)\cdot\widetilde{\lambda_{1}^{i}}:G\times S\rightarrow \Rbb$$ $$(g,x)\mapsto (l(i,1)\cdot\lambda_{1}^{i}(g,x))-\sum_{s=1}^{l(i,1)}a^{i}_{1,s}(x)+\sum_{s=1}^{l(i,1)}a^{i}_{1,s}(g.x)$$ is a $L^{2}$-integrable cocycle. Therefore we can deduce that $\widetilde{\lambda_{1}^{i}}$ is a $L^{2}$-integrable cocycle.

\item Assume that $\widetilde{\lambda^{i}_{1}},\dots,\widetilde{\lambda^{i}_{r}}$ are all $L^{2}$-integrable cocycles for some $1\le r\le m(i)-1$. Note that $\beta^{o}$ preserves the subspace $\bigoplus_{s=1}^{r+1} \overline{W^{i}_{s}}.$ Since $\beta^{o}$ is a $L^{2}$-integrable cocycle, we also know that $\beta^{o}\restrict{\oplus_{s=1}^{r+1} \overline{W^{i}_{s}}}$ is a $L^{2}$-integrable cocycle. Using same arguments with $j=1$ case, we can deduce that the cocycle $$l(i,1)\lambda^{i}_{1}(g,x)+\dots+l(i,r)\lambda^{i}_{r+1}(g,x)-\sum_{t=1}^{r+1}\sum_{s=1}^{l(i,t)}a^{i}_{j,s}(x)+\sum_{t=1}^{r+1}\sum_{s=1}^{l(i,t)}a^{i}_{j,s}(g.x)$$ is a $L^{2}$-integrable cocycle. This implies that $$l(i,1)\widetilde{\lambda^{i}_{1}}+\dots+l(i,r+1)\widetilde{\lambda^{i}_{r+1}}$$ is a $L^{2}$-integrable cocycle. This implies $\widetilde{\lambda^{i}_{r+1}}$ is a $L^{2}$-integrable cocycle due to the induction hypothesis.

\end{enumerate}
This proves claim.

 Now we will focus on the Lypaunov spectrum of the cocycle $\beta^{o}_{g_{0}}$ for $g_{0}\in G$.  Fix an element $g_{0}\in G$ and $i\in\{0,\dots,k-1\}$. We will prove that the Lyapunov spectrum of $\beta^{o}_{g_{0}}\restrict{V^{i+1}/V^{i}}$ on $V^{i+1}/V^{i}$ is same as the Lyapunov spectrum of $\pi\restrict{V^{i+1}/V^{i}}(g_{0})$ on $V^{i+1}/V^{i}$ for each $i$. Due to the functoriality of Lyapunov spectrum, this is enough for proof.
 In other words, it is enough to show that  $$\Sp\left(\beta_{g_{0}}\restrict{V_{i+1}/V_{i}}\right)=\Sp\left(\pi(g_{0})\restrict{V_{i+1}/V_{i}}\right)$$ for each $i$, as mentioned above. Since $\beta^{o}$ does not preserve each $\overline{W_{j}^{i}}$, we will use induction on $\bigoplus_{j=1}^{r}\overline{W_{j}^{i}}$.   From now on, we will omit $i$ if there is no confusion.
\begin{enumerate}
\item $r=1$ case : We will investigate Lyapunov spectrum of $\beta_{g_{0}}\restrict{\overline{W_{1}^{i}}}$ and $\pi(g_{0})\restrict{\overline{W_{1}^{i}}}$ on $\overline{W_{1}^{i}}$. Note again that $\beta_{g_{0}}$ preserves $\overline{W_{1}^{i}}$, so $\beta_{g_{0}}\restrict{\overline{W_{1}^{i}}}$ makes sense. 
From now on we will focus on $\overline{W_{1}^{i}}$, so drop the subscript about restriction to $\overline{W_{1}^{i}}$ for the notational convenience.  
Using MET for $\beta_{g_{0}}^{o}$, find Lyapunov exponents $$\chi(1)(x)> \dots> \chi(t(x))(x)$$ and the Lyapunov decomposition of $\beta^{o}_{g_{0}}$ in $\overline{W^{i}_{1}}$ such that for $\mu$-almost every $x\in S$, $$\overline{W^{i}_{1}}=\bigoplus_{q=1}^{t} \Lcal_{\chi(q)}(x),$$  $$w\in \Lcal_{\chi(q)}(x)\setminus\{0\} \iff \lim_{m\rightarrow\pm\infty}\frac{1}{m}\ln||\beta^{o}_{g_{0}}(m,x)w||_{op}=\chi(q)(x)$$ for $q=1,\dots, t(x)$.
We can also find Lyapunov exponent $$\xi(1)> \dots> \xi(s)$$ and Lyapunov decomposition of $\pi(g_{0})$ in $\overline{W^{i}_{1}}$ such that $$\overline{W^{i}_{1}}=\bigoplus_{q=1}^{s} \mathcal{P}_{\xi(q)}$$  $$w\in \mathcal{P}_{\xi(q)}\setminus\{0\} \iff \lim_{m\rightarrow\pm\infty}\frac{1}{m}\ln||\pi(g^{m}_{0})w||_{op}=\xi(q)$$ for $q=1,\dots, s$.

Fix generic point $x_{0}\in S$ with respect to MET for $\beta_{g_{0}}^{o}$. 
For any constant $c>0$ define $$A_{c}=\{x\in S : ||a(x)||, ||n(x)||<c\}.$$ We can find $\Ccal(x_{0})>0$ such that $x_{0}\in A_{\Ccal(x_{0})}$ and 
$\mu\left(A_{\Ccal(x_{0})}\right)>0.$ Using Poincar\'e recurrence theorem, we can find increasing subsequence $\{m_{k}\}\subset \Zbb$ such that  $g_{0}^{m_{k}}.x_{0}\in A_{\Ccal(x_{0})}$ for all $k\in\Zbb,$ and $\lim_{k\rightarrow\pm\infty}|m_{k}|=\infty$. Due to $L^{2}$-integrability of $\widetilde{\lambda_{1}^{i}}$, we can apply Theorem \ref{RvalueG} to $\widetilde{\lambda_{1}^{i}}$. Then we have $$\frac{1}{m_{k}}\left|\lambda_{1}^{i}(g_{0}^{m_{k}},x)\right|\longrightarrow 0\quad \textrm{ as } k\rightarrow \pm\infty$$ since $g_{0}^{m_{k}}.x\in A_{\Ccal(x_{0})}$ implies boundedness of $a$ and $n$.
Recall that on $\overline{W_{1}^{i}}$, we have $$\beta^{o}(g_{0},x)=e^{\lambda_{1}^{i}(g_{0},x)}a(g_{0}.x)n(g_{0}.x)\pi(g_{0})\Ksc_{1}^{i}(g_{0},x)n(x)^{-1}a(x)^{-1}.$$
Therefore we have following calculation for any $q\in\{1,\dots,t(x_{0})\}$.
\begin{align*}
&w\in \Lcal_{\chi(q)}(x_{0})\setminus\{0\} \\
&\iff  \frac{1}{m_{k}}\left[\left(\ln||\pi(g_{0}^{m_{k}})n(x_{0})^{-1}a(x_{0})^{-1}w||\right)+|\lambda_{1}^{i}(g_{0}^{m_{k}},x_{0})|\right] \xrightarrow{k\rightarrow\infty}   \chi(q)(x_{0})\\
& \iff \frac{1}{m_{k}}\ln||\pi(g_{0}^{m_{k}})n(x_{0})^{-1}a(x_{0})^{-1}w|| \xrightarrow{k\rightarrow\infty}  \chi(q)(x_{0})\\
& \iff \frac{1}{m}\ln||\pi(g_{0}^{m})n(x_{0})^{-1}a(x_{0})^{-1}w|| \xrightarrow{k\rightarrow\infty}  \chi(q)(x_{0})
\end{align*}
The first and last equivalences come from MET for $\alpha^{o}_{g}$ and $\pi(g)$. The above calculation shows that $t(x_{0})=s$, $\chi(q)(x_{0})=\xi(q)$ and $$\Lcal_{\chi(q)}(x_{0})=a(x_{0})n(x_{0})\mathcal{P}_{\xi(q)}$$ for any $q=1,\dots,t(x_{0})=s$.
This implies that $$\Sp\left(\beta_{g_{0}}\restrict{\overline{W_{1}^{i}}}\right)=\Sp\left(\pi(g_{0})\restrict{\overline{W_{1}^{i}}}\right).$$

\item For some $1\le r\le m(i)-1$, assume that $$\Sp\left(\beta_{g_{0}}\restrict{\bigoplus_{j=1}^{r}\overline{W_{j}^{i}}}\right)=\Sp\left(\pi(g_{0})\restrict{\bigoplus_{j=1}^{r}\overline{W_{j}^{i}}}\right).$$ 
We will prove that \begin{align}\label{rts}\Sp\left(\beta_{g_{0}}\restrict{\bigoplus_{j=1}^{r+1}\overline{W_{j}^{i}}}\right)=\Sp\left(\pi(g_{0})\restrict{\bigoplus_{j=1}^{r+1}\overline{W_{j}^{i}}}\right).\end{align}

Define $$Q=\bigoplus_{j=1}^{r+1}\overline{W_{j}^{i}}\Big/\bigoplus_{j=1}^{r}\overline{W_{j}^{i}}$$ so that the cocycle $\beta_{g_{0}}$ and $\pi(g_{0})$ preserves $Q$. Using functoriality of Lyapunov spectrum again, it is enough to prove 
\begin{align}\label{Tsh} \Sp\left(\beta_{g_{0}}\restrict{Q}\right)=\Sp\left(\pi(g_{0})\restrict{Q}\right)\end{align} 

We will use similar argument in order to show (\ref{Tsh}). 
From now on we will focus on $Q$, so drop the subscript about restriction to $Q$ for the notational convenience.   For $\mu$-almost every $x\in S$, find Lyapunov exponents $$\chi(1)(x)> \dots> \chi(t(x))(x)$$ and Lyapunov decomposition of $\alpha^{o}_{g_{0}}$ in $Q$ such that  $Q=\bigoplus_{q=1}^{t} \Lcal_{\chi(q)}(x),$  $$w\in \Lcal_{\chi(q)}(x)\setminus\{0\} \iff \lim_{m\rightarrow\pm}\frac{1}{m}\ln||\beta^{o}_{g_{0}}(m,x)w||_{op}=\chi(q)(x)$$ for $q=1,\dots, t(x)$.
We can also find Lyapunov exponents $$\xi(1)> \dots> \xi(s)$$ and the Lyapunov decomposition of $\pi(g_{0})$ in $Q$ such that for $\mu$-almost every $x\in S$, $\overline{W^{i}_{1}}=\bigoplus_{q=1}^{s} \mathcal{P}_{\xi(q)},$  $$w\in \mathcal{P}_{\xi(q)}\setminus\{0\} \iff \lim_{m\rightarrow\pm}\frac{1}{m}\ln||\pi(g^{m}_{0})w||_{op}=\xi(q)$$ for $q=1,\dots, s$.

Fix generic point $x_{0}\in M$ with respect to MET for $\beta{g_{0}}^{o}$. 
For any constant $c>0$ define $$A_{c}=\{x\in S : ||a(x)||, ||n(x)||<c\}.$$ We can find $\Ccal(x_{0})>0$ such that $x_{0}\in A_{\Ccal(x_{0})}$ and $\mu\left(A_{\Ccal(x_{0})}\right)>0.$ Using Poincar\'e recurrence theorem, we can find increasing subsequence $\{m_{k}\}\subset \Zbb$ such that $g_{0}^{m_{k}}.x_{0}\in A_{\Ccal(x_{0})}$ for all $k\in\Zbb,$
 and $\lim_{k\rightarrow\pm\infty}|m_{k}|=\infty$. Due to $L^{2}$-integrability of $\widetilde{\lambda_{r+1}^{i}}$, we can apply Theorem \ref{RvalueG} to $\widetilde{\lambda_{r+1}^{i}}$.  This implies $$\frac{1}{m_{k}}\left|\lambda_{r+1}^{i}(g_{0}^{m_{k}},x)\right|\longrightarrow 0\quad \textrm{ as } k\rightarrow \pm\infty.$$  

Note that the induced cocycle $\beta^{0}_{g_{0}}\restrict{Q}$ is of the form 
 $$\beta^{0}\restrict{Q}(g_{0},x)=e^{\lambda_{r+1}^{i}(g_{0},x)}a(g_{0}.x)\restrict{Q}n(g_{0}.x)\restrict{Q}\pi(g_{0})\restrict{Q}\Ksc_{r+1}^{i}(g_{0},x)\restrict{Q}n(x)\restrict{Q}^{-1}a(x)\restrict{Q}^{-1}$$ by constructions.

We can calculate followings for any $q\in\{1,\dots,t(x_{0})\}$ as before.
\begin{align*}
&w\in \Lcal_{\chi(q)}(x_{0}) \\
&\iff  \frac{1}{m_{k}}\left[\left(\ln||\pi(g^{m_{k}})n(x_{0})^{-1}a(x_{0})^{-1}w||\right)+|\lambda_{r+1}^{i}(g^{m_{k}},x_{0})|\right]\xrightarrow{k\rightarrow\infty}  \chi(q)(x_{0})\\
& \iff \frac{1}{m_{k}}\ln||\pi(g^{m_{k}})n(x_{0})^{-1}a(x_{0})^{-1}w|| \xrightarrow{k\rightarrow\infty}  \chi(q)(x_{0})\\
& \iff \frac{1}{m}\ln||\pi(g^{m})n(x_{0})^{-1}a(x_{0})^{-1}w|| \xrightarrow{k\rightarrow\infty}  \chi(q)(x_{0})
\end{align*}
Again the last equivalence comes from MET for the $\pi(g)$. The above calculation shows that $t(x_{0})=s$, $\chi(q)(x_{0})=\xi(q)$ and $$\Lcal_{\chi(q)}(x_{0})=a(x_{0})n(x_{0})\mathcal{P}_{\xi(q)}$$ for any $q=1,\dots,t(x_{0})=s$.
This implies that (\ref{Tsh}), so we prove (\ref{rts}).

\end{enumerate}

The above arguments show that 
$$\Sp\left(\beta_{g_{0}}\restrict{V^{i+1}/V^{i}}\right)=\Sp\left(\pi(g_{0})\restrict{V^{i+1}/V^{i}}\right)$$ for each $i$. This proves the \textbf{Part 1.}  due to functoriality of Lyapunov spectrum as mentioned before. \end{proof}

\begin{rmk}\label{rmk:expQuasi}
The above arguments also give the proof of the item (c). 
For the quasi $L^{2}$-integrable minimal cocycle $\beta^{\phi}=\pi\cdot\Esc\cdot u$, we characterized Lyapunov subspaces of $\beta_{g}$ as
 $E_{*}^{\beta_{g}}(x)=o(x)a(x)n(x) E_{*}^{\pi(g)}$ in the above arguments where $E_{*}$ is the Lyapunov subspaces for quasi-integrable cocycles as in the section \ref{sec:pre}. This implies that
 $E_{*}^{\beta^{\phi}_{g}}(x)= E_{*}^{\pi(g)}$ almost every $x\in S$ as quasi-$L^{2}$ cocycle. 
\end{rmk}

The irreducible lattice version of the dynamical super-rigidity theorem is following.
\begin{theorem}[Dynamical cocycle super-rigidity for $\Gamma$]\label{thm:superrigidGamma}
Let $\Gamma$ satisfy \ref{std} that is $\Gamma$ be a weakly irreducible lattice of higher rank group $G$. Let $(S,\mu)$ be $\Gamma$-ergodic standard probability space. Assume \ref{std2} for the $\Gamma$ action that is the induced $G$ action is weakly irreducible. Let $\beta:\Gamma\times S\rightarrow \GL(d,\Rbb)$ be a $L^{2}$ integrable cocycle, i.e. $$\forall\gamma\in \Gamma,\quad \ln||\beta(\gamma,-)||\in L^{2}(S).$$  After replacing $S$ to finite extension $\widehat{S}=S\times_{i}I$, we have a measurable map $\phi:\widehat{S}\rightarrow \GL(d,\Rbb),$ a rational homomorphism $\pi:\Gbb\rightarrow \GL(d,\Cbb),$ an amenable real algebraic group $A_{0}$ and $A_{0}$-valued cocycle $\Esc_{\circ}:\Gamma\times \widehat{S}\rightarrow \GL(d,\Rbb)$ such that
\begin{enumerate}
	\item[(a)] For any $\gamma\in \Gamma$, $$\beta(\gamma,x)=\phi(\gamma.(x,j))\pi(\gamma)\Esc_{\circ}(\gamma,x,j)\phi(x,j)^{-1}$$ almost every $(x,i)\in \widehat{S}$.
	\item[(b)] For any $\gamma\in \Gamma$, $\Sp(\beta_{\gamma})=\Sp(\pi(\gamma))$.

	\item[(c)] For almost every $x\in S$ and all $j\in I$,  the Lyapunov subspaces $E_{\lambda}(x)$ of $\beta_{\gamma}$ corresponding the Lyapunov exponent $\lambda$ is $\phi(x,j)^{-1}E_{\lambda}$ where $E_{\lambda}$ is the Lyapuonv subspace of $\pi(\gamma)$ corresponding the Lyapunov exponent $\lambda$. 
	\item[(d)] We can write $\Esc_{\circ}=\Esc\cdot u$ where $\Esc$ is the cocycle that takes values on an amenable subgroup $A$ in the Levi subgroup of the algebraic hull and $u$ is the twisted cocycle that takes values on the unipotent radical of the algebraic hull. Furthermore, $\pi(G)$ commutes with $A$.
	\item[(e)]  $\pi(G)$ commutes with the unipotent radical of the algebraic hull so that $\pi$ and $u$ commutes. In particular, $\pi$ and $\Esc_{0}$ commutes. 
	\end{enumerate}

\end{theorem}

\begin{rmk}
As remarked on Remark \ref{rmk:errorint}, we can not get integrability of $\Esc_{\circ}$ here also. In addition, here $u$ is also a twisted cocycle since $u$ may not commute with $\Esc$. \end{rmk}

Now Theorem \ref{thm:superrigidGamma} comes from the standard induction argument and Theorem \ref{thm:redzim}.
\begin{proof}[Proof of Theorem \ref{thm:superrigidGamma}]
Recall that we have a $L^{2}$-integrable cocycle $\beta:\Gamma\times S\rightarrow \GL(d,\Rbb)$. As before, we can find a measurable map $\phi:S\rightarrow \GL(d,\Rbb)$ such that $\beta^{\phi}:\Gamma\times S\rightarrow L$ is a minimal cocycle. In other words, the $L$ is a $\Rbb$-points of algebraic group $\Lbb$ defined over $\Rbb$ and an algebraic hull of $\beta$ is $L$. Using Theorem \ref{thm:redzim}, we may write for any $\gamma\in\Gamma$ and $\mu$-almost every $x\in S$, $$\beta^{\phi}(\gamma,x)=\pi(\gamma) \Esc(\gamma,x) u(\gamma,x)$$ for some rational homomorphism $\pi:\Gbb\rightarrow \Lbb$, a cocycle $\Esc:\Gamma\times S\rightarrow A$ and measurable map $u:\Gamma\times S\rightarrow U$ after modification of conjugacy map. Here we retain notations as before and denote $A=Z(\Hbb)\cdot \Tbb\cdot \Kbb(\Rbb)$.  We follow same arguments with Part $1$ of Theorem \ref{thm:superrigidG}, using Theorem \ref{RvalueGamma} case instead of \ref{RvalueG}, we can deduce that for any $\gamma\in \Gamma$, $\Sp(\beta_{\gamma})=\Sp(\pi(\gamma))$. This proves Part $1$ in Theorem \ref{thm:superrigidGamma}.

For Part $2$, we will use induced cocycle $\widetilde{\beta^{\phi}}$of $\beta^{\phi}$. Note that $\widetilde{\beta^{\phi}}$ is cohomologous to the induced cocycle $\widetilde{\beta}$ of $\beta$. Note that the $\pi$ comes from Theorem \ref{thm:redzim} for $\beta^{\phi}$ is same as for $\widetilde{\beta^{\phi}}$.
We know that
\begin{enumerate}
\item the induced action is irreducible,
\item the induced cocycle $\widetilde{\beta}$ is $L^{2}$-integrable for some choice of fundamental domain of $\Gamma$ of $G$ (see \cite{FM1} section 3.7) and 
\item the algebraic hull of $\widetilde{\beta}$ (so the algebraic hull of $\widetilde{\beta^{\phi}}$) is still $L$.
\end{enumerate}
Now we use Part $2$ of Theorem \ref{thm:superrigidG} for $\widetilde{\beta^{\phi}}$, so that we conclude $\pi(G)$ commutes with $U$. Since $\pi(\Gamma)$ commutes with $A$ already, we can conclude that the $\Esc\cdot u :\Gamma\times S\rightarrow \GL(d,\Rbb)$ is indeed an amenable valued cocycle. This proves the Theorem \ref{thm:superrigidGamma}. \end{proof}

\begin{rmk}\label{diff}
 The differences between cocycle super-rigidity and dynamical super-rigidity are following.
\begin{enumerate}
\item in \cite{FM1}, they can get $\pi\cdot C$, product of homomorphism with compact group cocycle using cocycle super-rigidity. Especially, $\pi\cdot C$ so that the error term $C$ is always integrable even bounded. However, using dynamical cocycle super-rigidity do not tell us that the error term is even quasi-integrable. The author expect the error term is itself quasi $L^{2}$-integrable.
\item in \cite{FM1}, the algebraic hull is reductive as result of cocycle super-rigidity. However, we can not get rid of the unipotent part as a result of the dynamical cocycle super-rigidity. The commutativity of $\pi$ and $u$ and the fact that $\Esc$ takes values in a reductive subgroup of the algebraic hull will be useful for local and global rigidity.
\end{enumerate}
\end{rmk}
\subsection{Dynamical super-rigidity homomorphism}\label{sec:varcoc}
We will collect some theorems about the homomorphism from dynamical cocycle super-rigidity. We will assume $L^{2}$-integrability, \ref{std} and \ref{std2} as before. Then we can derive same variants as in \cite{FM1}.

Note that the following two facts will be used frequently. 
\begin{lem}[\cite{FM1} Lemma 3.21]\label{Fzd}
For $G$ and $\Gamma$ satisfy \ref{std}, there is finite elements $g_{1},\dots,g_{l}\in D$ such that the group $\Fcal$ generated by their polar part is Zariski dense in $G$. 
\end{lem}
\begin{theorem}[\cite{FM1} Corollary 4.3]\label{orbit}
Let $\Lbb=\Abb\ltimes \Hbb$ be the algebraic group that is a semidirect product of algebraic groups $\Abb$ and $\Hbb$. Let $\Fcal$ be the finitely generated group. Then $H=\Hbb(\Rbb)$-orbits of completely reducible homomorphisms in $\Hom(\Fcal,L)$ are Hausdorff closed.
\end{theorem}

Let's fix some notations. From now on $D=G$ or $\Gamma$ and $(S,\mu)$ be group and measure space that satisfies \ref{std} and \ref{std2}. $L$ be a real algebraic Lie group.  Let $\beta:D\times S\rightarrow L<\GL(d,\Rbb)$ be a $L^{2}$-integrable cocycle. 
\begin{enumerate}
\item When $D=\Gamma$, if we have homomorphism $\pi_{0}:\Gamma\rightarrow \GL(d,\Rbb)$ then we can find a rational homomorphism $\pi_{0}^{E}:\Gbb\rightarrow \GL(d,\Cbb)$ and homomorphism with bounded image $\pi_{0}^{K}:\Gamma\rightarrow \GL(d,\Rbb)$ such that $\pi=\pi_{0}^{E}\pi_{0}^{K}$ and they commute using Margulis' super-rigidity. Note that Zariski closure of $\pi_{0}(\Gamma)$ is semisimple.

\item When $D=G$ and we have a rational homomorphism $\pi_{0}:G\rightarrow \GL(d,\Rbb)$, we will denote $\pi_{0}^{E}=\pi_{0}$ and $\pi_{0}^{K}$ is trivial homomorphism for simplicity.

\item For both cases: Using dynamical super-rigidity, we can find rational homomorphism $\pi:\Gbb\rightarrow \GL(d,\Cbb)$ such that for any $g\in D$, $\Sp(\beta_{g})=\Sp(\pi(g)).$
\end{enumerate}

As before, when we say about $L^{\infty}$ close between two cocycles, the norm on target algebraic group is $\ln||\cdot||$ for a fixed embedding in to $\GL(d,\Rbb)$.

Since the proofs of the following two theorems are almost same with the \cite{FM1}, we just contain the sketch of the proof for reader's convenience. Indeed, we will see detailed proof of them in more general setting in the middle of the proof of theorem \ref{thm:locrigidcoc}.

\begin{theorem}[Uniqueness of super-rigidity homomorphism]\label{uniqhom}
Let $D=G$ or $\Gamma$ satisfy \ref{std}. Assume $(S,\mu)$ is standard $D$-ergodic space. Further assume \ref{std2}. Let $\beta:D\times S\rightarrow 
L$ be $L^{2}$-integrable cocycle. Assume that there is a homomorphism $\pi_{0}:D\rightarrow L$ such that $$\Sp(\beta_{g})=\Sp(\pi_{0}(g))$$ for any $g\in D$. Then $\pi$ is conjugate to $\pi_{0}^{E}$ by an element $l\in L$.
Moreover, if there is a measurable map $\phi:\widehat{S}\rightarrow L$ such that $\phi(\widehat{x})W(\beta_{g}(x))=W(\pi_{0}^{E}(g))$ for Lyapunov subspaces $W$ at $x$ corresponding cocycles, then there is a measurable map $\phi_{z}:\widehat{S}\rightarrow Z_{L}(\pi_{0}^{E})$ such that $\phi(\widehat{x})=l^{-1}\phi_{z}(\widehat{x})$.
\end{theorem}

The proof is same as the proof  of Theorem 3.20 in \cite{FM1}. Indeed, we can find finite number of elements that the group $\Fcal$ generated by their polar part is Zariski dense as in the Lemma \ref{Fzd}. We consider the $L$ action on $\Hom(\Fcal,L)$ and will prove that $\pi_{0}^{E}$ and $\pi'$ are in the same $L$-orbit using the fact that the polar element is determined by their Lyapuonv subspaces and exponents. This will give the first assetion. Furthermore, the last assertion comes from the fact that the stabilizer of $\pi_{0}^{E}(\Fcal)$ in $L$ is $Z_{L}(\pi_{0}^{E}(\Fcal))=Z_{L}(\pi_{0}^{E}(G)$.

\begin{theorem}[Local rigid of super-rigidity homomorphism]\label{locrigidhom}
Assume that $\beta:D\times S\rightarrow L$ is sufficiently $L^{\infty}$-close to the constant cocycle $\pi_{0}:D\rightarrow L$. Then $\pi$ is conjugate to $\pi_{0}^{E}$ by an element $l\in L$. 
\end{theorem}
The proof is same as the first part of the proof of Theorem 5.1 in \cite{FM1}. (In their notation, the arguments for proving $\widetilde{\pi}_{i}=(\pi_{A},\pi_{H}^{E})$ can be used here.) First, note that $\beta$ is $L^{\infty}$-closed to $\pi_{0}$ so that $\beta$ is a bounded cocycle. Especially, $\beta$ is $L^{2}$-integrable. We can use stability of hyperbolic vector bundle maps (\cite{Pesin} or Lemma 5.2 in \cite{FM1}). As before, we can find finite number of elements that the group $\Fcal$ generated by their polar part is Zariski dense by Lemma \ref{Fzd}. After that, stability of characteristic spaces for that elements, the fact that the $L$ orbits on the $\Hom(\Fcal,L)$ is closed (Theorem \ref{orbit}) and the finiteness of conjugcay classes of $\Hom(G,L)$ give us to prove that the $\pi$ is conjugate to the $\pi_{0}^{E}$. 

\section{Local rigidity}\label{sec:locrigid}

\subsection{Settings}

We will follow notations in \cite{FM1} in this section.
 Let $G$ be the connected real semisimple higher rank algebraic Lie group without compact factor. Let $\Gamma$ be a weakly irreducible lattice in $G$. Let $D=G$ or $\Gamma$. $H$ be a real algebraic Lie group and $\Lambda$ be a cocompact lattice in $H$. We will consider affine action on $D$ on $H/\Lambda$, $\Acal_{0}:D\rightarrow \Aff(H/\Lambda).$ We will prove following theorem which implies Theorem \ref{thm:localrigid} and \ref{thm:globalrigid}.
When $D=\Gamma$, we will define  the finite index normal subgroup $\Gamma'(\Acal_{0})=\Gamma'$ of $\Gamma$ that only depends on $\Acal_{0}$. 
 \begin{theorem}\label{thm:locrigidgen}

 Let $\Acal_{0}$ is $C^{1}$-affine action of $D$ on $H/\Lambda$. (resp. $C^{\infty}$-affine action.)
 Assume that the $\Acal_{0}$ is weakly hyperbolic.  
 \begin{enumerate}
 \item When $D=G$: Let the $C^{1}$-action (resp. $C^{\infty}$-action) $\Acal$ preserves a fully supported Borel probability measure that is weakly irreducible. If the action $\Acal$ of $G$ on $H/\Lambda$ is sufficiently $C^{1}$-closed to $\Acal_{0}$ then there is a homeomorphism (resp. $C^{\infty}$-diffeomorphism) $\lambda:H/\Lambda\rightarrow H/\Lambda$ such that for any $g\in G$, $$\lambda^{-1}\circ \Acal(g)\circ \lambda=\Acal_{0}(g).$$ 
 
  \item  When $D=\Gamma$: If $C^{1}$-action (resp. $C^{\infty}$-action) $\Acal$ of $\Gamma$ on $H/\Lambda$ preserved a fully supported Borel probability measure so that $\Acal\restrict{\Gamma'}$ is induced weakly irreducible and the action $\Acal$ is sufficiently $C^{1}$-closed to $\Acal_{0}$ then there is a homeomorphism (resp. $C^{\infty}$-diffeomorphism) $\lambda:H/\Lambda\rightarrow H/\Lambda$ such that for any $\gamma\in \Gamma$, $$\lambda^{-1}\circ \Acal(\gamma)\circ \lambda = \Acal_{0}(\gamma).$$ 
 \end{enumerate}

 \end{theorem}
 
\begin{rmk}\label{rmk:locrigid}

When all simple factor of $G$ has rank at least $2$, property (T) guarantee there is an absolutely continuous invariant probability measure for perturbed action. Also, in that case, weakly irreducibility is same as ergodicity. Therefore, the action enjoy the ergodic decomposition. This is the reason why \cite{MQ} get local rigidity without assumption about measure.

On the other hand, when $G$ has rank $1$ factor, we assume both conditions. Especially, the action may factor through rank $1$ factor so we need weakly irreducibility assumption.

\end{rmk}

  Note that we may assume without loss of generality, $G$ and $\Gamma$ satisfy \ref{std} with $\textrm{rank}_{\Rbb}(G)\ge 2$ if we think about algebraic simply connected covering of $G$ and preimage of $\Gamma$ on it. (e.g. \cite{MQ}, \cite{FM1} and \cite{FM2}.) Throughout this section, $G$ satisfies \ref{std} with $\textrm{rank}_{\Rbb}(G)\ge 2$.
\subsection{Some facts}

Recall some facts from \cite{MQ},\cite{FM1} and \cite{FM2}. 
\subsubsection{Description of the Affine action} When $D=G$ case, affine actions can be easily described.
 \begin{theorem}[\cite{FM1} Theorem 6.4] When $D=G$, there is a rational homomorphism $\pi_{0}:G\rightarrow H$ such that $\Acal_{0}(g)([x])=[\pi_{0}(g)x]$ for any $g\in G$ and $[x]\in H/\Lambda$.
 \end{theorem}

For $D=\Gamma$ action case, it is much complicate to describe an affine action. Note that we may assume that the group of affine automorphism of $H$ can be written as $\Aff(H)=\Aut^{A}(H)\ltimes H$ as we saw in section \ref{preact}.

\begin{theorem}[\cite{FM1} Theorem 6.5]\label{thm:6.5}
When $D=\Gamma$,  there are finite index subgroup $\Gamma'(\Acal_{0})=\Gamma'<\Gamma$ and a homomorphism $\pi:\Gamma'\rightarrow \Aut(H)^{A}\ltimes H$ such that $$\Acal_{0}(\gamma')([x])=[\pi_{0}(\gamma')x]$$ for any $\gamma'\in \Gamma'$ and $[x]\in \left(\pi(\Gamma')\ltimes H\right)/ \left(\pi(\Gamma')\ltimes \Lambda\right)$. Note that we identify $H/\Lambda$ with $\left(\pi(\Gamma')\ltimes H\right)/ \left(\pi(\Gamma')\ltimes \Lambda\right)$ and abusing notations. Denote $\Abb$ and $A=\Abb(\Rbb)$ be the Zariski closure of $\pi(\Gamma')$ and its real points respectively. Note that $\Abb$ is connected semisimple by Margulis super-rigidity theorem. Then we have $\pi_{0}=\pi_{A}\pi_{H}$ for some homomorphism $\pi_{A}:\Gamma'\rightarrow A$ and $\pi_{H}:\Gamma'\rightarrow H$. Furthermore, we may assume $\pi_{A}$ and $\pi_{H}$ commutes without loss of generality.
\end{theorem}
Note that the $\Gamma'(\Acal_{0})=\Gamma'$ is the finite index subgroup of $\Gamma$ that is appeared in the theorem \ref{thm:localrigid} and \ref{thm:locrigidgen}. We may assume that $\Gamma'$ be a normal subgroup of $\Gamma$ without loss of generality.  The $\Gamma'$ is still a weakly irreducible lattice in $G$. Recall that for any finite index subgroup $\Gamma''<\Gamma$, $\Gamma$ action is weakly hyperbolic if and only if $\Gamma''$ is weakly hyperbolic. Therefore, $\Acal_{0}\restrict{\Gamma'}$ is still weakly hyperbolic. Finally, if we can prove local rigidity for the finite index subgroup $\Gamma'$, then it gives local rigidity for full group $\Gamma$. For example, one can follow the same arguments in \cite{MQ} step 2 in the proof of the local rigidity of $\rho_{A}$ directly. 

For unifying notations, when $D=G$, denote $A$ be a trivial group.

\subsubsection{The cocycle related the perturbed action $\Acal$.}
For the sufficiently small $C^{1}$-perturbed action $\Acal$ from the affine action $\Acal_{0}$, we have cocycle $\beta_{A}$ over $\Acal$ that contains informations about $\Acal$, following \cite{MQ} and \cite{FM1}. The main reason that we can define the cocycle $\beta_{A}$ is that $\Acal$ can be lifted into cover $H$ of $H/\Lambda$. One can see this using triviality of the $2$nd group cohomology or directly follow \cite{MQ}.

\begin{prop}\label{prop:FMQcoc} Assume the notations and settings in the theorem \ref{thm:locrigidgen}. If $\Acal$ is sufficiently $C^{1}$-close to $\Acal_{0}$.
\begin{enumerate}
\item When $D=G$:  There is a continuous cocycle $\beta_{A}:G\times H/\Lambda \rightarrow A\ltimes H$ over the action $\Acal$ such that 
$A$ is trivial group and the action $\Acal$ satisfies $$\Acal(g)([x])=[\beta_{A}(g,[x])x]$$ for any $g\in G$ and $[x]\in H/\Lambda$.
\item When $D=\Gamma$: Let the finite index subgroup $\Gamma'$ be in the theorem \ref{thm:6.5}. There is a continuous cocycle $\beta_{A}:\Gamma'\times H/\Lambda \rightarrow A\ltimes H$ over the action $\Acal$ such that 
 The action $\Acal\restrict{\Gamma'}$ satisfies $$\Acal\restrict{\Gamma'}(\gamma)([x])=\beta_{A}(\gamma,[x])[1,x]$$ for any $\gamma\in \Gamma'$ and $[1,x]\in \left(\pi(\Gamma')\ltimes H\right)/ \left(\pi(\Gamma')\ltimes \Lambda\right)$. Here we identify $H/\Lambda$ with $\left(\pi(\Gamma')\ltimes H\right)/ \left(\pi(\Gamma')\ltimes \Lambda\right)$ again. (Especially, $1$ is the identity element in $A$.) 

\item In both cases, $A$-component of $\beta_{A}$ is $\pi_{A}$. Moreover, $\beta_{A}$ is $C^{0}$-close to $\pi_{0}$. 

\end{enumerate}
We will usually denote $\Lbb=\Abb\ltimes \Hbb$ and $L=A\ltimes H$.  When $D=\Gamma$, we abuse notations so that $D$ can be either $\Gamma$ or $\Gamma'$ if the context is clear. (For example, $\pi_{0}(D)$ will be $\pi_{0}(\Gamma')$.)
\end{prop}

\subsubsection{Notations from Margulis' super-rigidity} Recall some notations in \cite{FM1}. Those notations coincide with in the above settings. We will use notations throughout in this section.

Let $D=G$ or $\Gamma$ be of the form in \ref{std2} with $\textrm{rank}_{\Rbb}(G)\ge 2$. Let $\Lbb$ be an algebraic group and $\Hbb,\Abb<\Lbb$ be an algebraic subgroup such that $\Lbb=\Abb\ltimes \Hbb$. We think $\Lbb$ is subgroup of $\GL(d,\Rbb)$ for some $d>0$ as usual. We further assume that $\Abb$ is connected semisimple algebraic group. As before, denote $L$,$A$ and $H$ be a $\Lbb(\Rbb)$,$\Abb(\Rbb)$ and $\Hbb(\Rbb)$ respectively. Denote $p_{A}:L\rightarrow A$ be the natural projection.

We fix a homomorphism $\pi_{0}:D\rightarrow L$ such that there are homomorphism $\pi_{A}:D\rightarrow A$ and $\pi_{H}:D\rightarrow H$ such that $\pi_{A}$ and $\pi_{H}$ commutes. 
\begin{enumerate}
\item When $D=G$, $\pi$ will be rational homomorphism as before. Further assume that $\pi(D)$ is Zariski dense in $\Abb$. Denote $\pi_{0}^{K}$ for trivial homomorphism and $\pi_{0}^{E}=\pi$. More precisely, denote $\pi_{A}^{E}=\pi_{A}$, $\pi_{H}^{E}=\pi_{H}$ and both $\pi_{A}^{K}$ and $\pi_{H}^{K}$ for trivial homomorphism.
\item When $D=\Gamma$, recall that Margulis' super-rigidity theorem tells us that $\pi_{A}$ and $\pi_{H}$ can be written as $\pi_{A}^{E}\cdot\pi_{A}^{K}$ and $\pi_{H}^{E}\cdot\pi_{H}^{K}$ where $\pi_{A}^{E},\pi_{H}^{E}$ is restriction to $\Gamma$ of rational homomorphism from $\Gbb$ to $\Abb$ and $\Hbb$ respectively and $\pi_{A}^{K}$ and $\pi_{H}^{K}$ is homomorphism from $\Gamma$ to compact subgroup in $A$ and $H$ respectively. 
\item For $\pi_{0}$, we will define $\pi_{0}^{E}$ and $\pi_{0}^{K}$ as $\pi_{0}^{E}=\pi_{A}^{E}\pi_{H}^{E}$ and $\pi_{0}^{K}=\pi_{A}^{K}\pi_{H}^{K}$.
\end{enumerate}
Note that $\pi_{A}^{E}, \pi_{A}^{K}, \pi_{H}^{E}$ and $\pi_{H}^{K}$ commutes each other.

\subsection{Local rigidity of constant cocycle} 
We present local rigidity of constant cocycle as in \cite{FM1} under additional assumptions that holds for applications.  We will assume that $Z=Z_{L}(\pi_{0}^{E}(D))\cap H$ is finite. Note that, this implies $\pi_{H}^{K}(D)$ is finite. This condition will be satisfied when the affine action is weakly hyperbolic. Now we can address the main theorem in this subsection.
\begin{theorem}[Local rigidity of the constant cocycle]\label{thm:locrigidcoc} Assuming that $D=G$ or $\Gamma$ satisfies \ref{std} with $\textrm{rank}_{\Rbb}(G)\ge 2$. Let $(S,\mu)$ be the standard $D$-space and assume the $D$ action satisfies \ref{std2}. Retaining notations in the previous subsection, let $\pi_{0}=\beta_{0}:D\times S\rightarrow L$ be the constant cocycle, that is $\beta_{0}(g,x)=\pi_{0}(g)$ for any $g\in D$ and $x\in S$. Let $\beta :D\times S\rightarrow L$ be measurable cocycle over the action $\Acal$ such that $p_{A}\circ \beta=\pi_{A}$. Further assume that $Z=Z_{L}(\pi_{0}^{E}(D))\cap H$ is finite.

If $\beta$ and $\beta_{0}$ is sufficiently $L^{\infty}$-closed, then there is a measurable map $f: S\rightarrow H$ and a measurable cocycle $\widetilde{z}:D\times \widehat{S}\rightarrow Z$ such that 
\begin{enumerate}
\item we have $\beta(g,x)=f(g.x)^{-1}\pi_{A}(g)\pi_{H}^{E}(g)\widetilde{z}(g,\widehat{x})f(x)$ almost every $\widehat{x}\in\widehat{S}$.
\item the measurable map $f:S\rightarrow H$ is $L^{\infty}$-small in the sense that $f$ is $L^{\infty}$-close to constant map take values at the identity.
\item For any $g\in D$ and almost every $\widehat{x}\in \widehat{S}$, we have $\widetilde{z}(g,\widehat{x})=\pi_{H}^{K}(g)$. This implies, especially, $$\beta(g,x)=f(g.x)^{-1}\pi_{0}(g)f(x)$$ for any $g\in D$ and almost every $x\in S$.
\end{enumerate}
Furthermore, if $S$ is topological space, $\mu$ is fully supported on $S$, the cocycle $\beta$ and the $D$ action on $S$ is continuous, then $f$ can be chosen continuous and $C^{0}$-close to $id_{H}$.
\end{theorem}
The statements of the theorem \ref{thm:locrigidcoc} should be compared with Theorem 5.1 in \cite{FM1} and Theorem 3.1 in \cite{MQ}. The above theorem is middle of two theorems in some sense. On the other hand, we assume that the action satisfies \ref{std2}. Especially, the action is ergodic. Therefore, we don't need to think about ergodic components. The dynamical super-rigidity will be adapted in order to prove the theorem \ref{thm:locrigidcoc}.

However, the main obstruction occurs from the fact that we can not get rid of the unipotent radical of the algebraic hull as a result of dynamical super-rigidity a priori. Also, the (3) in the Theorem \ref{thm:locrigidcoc} is not claimed in \cite{FM1}.  However, when $Z$ is finite, we can overcome these issues.

\begin{rmk}\label{rmk:compMQ}
In \cite{MQ} Theorem 3.1, they proved a version of local rigidity of the constant cocycle with similar finiteness assumption. Indeed, one can prove the version of Theorem 3.1 in \cite{MQ} using dynamical super-rigidity more directly. 
\end{rmk}

\begin{proof}[Proof of Theorem \ref{thm:locrigidcoc}]

Let $\Hbb=\Fbb\ltimes \Ubb$ be the Levi decomposition into $\Fbb$. Following arguments in \cite{FM1}, we can write $\Fbb=\Fbb_{1}\cdot\Fbb_{2}$ an almost direct product so that $\Lbb=((\Abb\ltimes \Fbb_{1})\Fbb_{2})\ltimes \Ubb$.

Following \cite{FM1}, we can change algebraic group structure so that replace $\Abb\ltimes \Fbb_{1}$ with $\Abb\times \Fbb_{1}$ via $(a,f)\mapsto (a,a^{-1}(f))$. Note that this gives homeomorphism between $A\ltimes F_{1}$ and $A\times F_{1}$.  Since $\pi_{A}$ and $\pi_{H}$ commutes, $\pi_{H}$ takes values in $F_{2}$ so it does not effect to $\pi_{0}$. Furthermore, one can check directly that the $\beta$ still satisfies the cocycle condition in the new group structure. Moreover, $\beta$ is still $L^{\infty}$-close (resp. $C^{0}$-close) to $\pi_{0}$. Therefore, we may assume $\Abb$ commutes with $\Fbb$ without loss of generality, $\Lbb=(\Abb\times \Fbb)\ltimes \Ubb$. 

As usual denote $A=\Abb(\Rbb), F=\Fbb(\Rbb)$ and $U=\Ubb(\Rbb)$. Note that $\Ubb$ is unipotent radical in $\Lbb$ and $\Abb\times \Fbb$ is Levi component in $\Lbb$.

Using dynamical cocycle super-rigidity, find a measurable map $\phi':\widehat{S}\rightarrow L$, a homomorphism $\pi':\Gbb\rightarrow \Lbb$, a cocycle $\Esc':D\times \widehat{S}\rightarrow L$ and a measurable map $u':D\times \widehat{S}\rightarrow L$ such that $$\beta(g,x)=\phi'(g.(x,l))^{-1}\pi'(g)\Esc'(g,x,l)u'(x,l)\phi'(x,l)$$ for any $g\in D$ and almost every $(x,l)\in \widehat{S}$. Since $\Abb\times \Fbb$ is Levi component of $\Lbb$, up to conjugation, we may assume that $\pi'$ and $\Esc'$ takes values in $A\times F$. 

For the measurable map $\Psi$ into group $AF\ltimes U$ and a subgroup $B<AF\ltimes U$, we will denote $\Psi^{B}$ for $B$-component of $\Psi$ whenever $B=A,F,U$ and $H$.

Abbreviate notations, 
\begin{align*} 
\beta&=\phi'(g.\widehat{x})^{-1}\pi'\Esc'u'\phi'(\widehat{x})\\
&=\phi'^{H}(g.\widehat{x})^{-1}\phi'^{A}(g.\widehat{x})^{-1}\pi'^{A}\pi'^{F}\Esc'^{A}\Esc'^{F}u'^{A}u'^{H}\phi'^{A}(\widehat{x})^{-1}\phi'^{H}(\widehat{x})
\end{align*}
 Recall that  the cocycle $\beta:D\times S\rightarrow (AF)\ltimes U$ satisfies $p_{A}\circ\beta=\pi_{A}$ where $p_{A}:L\rightarrow A$ be the natural projection. Comparing $A$-components, we can get following.
\begin{equation}\label{c0}
\phi'^{A}(g.\widehat{x})^{-1}\pi'^{A}\Esc'^{A}u'^{A}\phi'^{A}(\widehat{x})=\pi_{A}.
\end{equation}
Furthermore, the cocycle $\pi'^{A}\Esc'^{A}u'^{A}$ is minimal. Indeed, the cocycle $\pi'\Esc'u'$ is minimal and $p_{A}\circ (\pi'\Esc'u')=\pi'^{A}\Esc'^{A}u'^{A}$. This implies that the cocycle $\pi'^{A}\Esc'^{A}u'^{A}$ is minimal.
Using the same argument in \textbf{Part 1} of dynamical super-rigidity (or cocycle super-rigidity for semisimple algebraic hull as in the \cite{FM1}),  this implies that the cocycle $(\pi'^{A}\Esc'^{A}u'^{A})^{\phi'^{A}}$ has same Lyapunov spectrum as $\pi'^{A}$. From the uniqueness of super-rigid homomorphism (Theorem \ref{uniqhom}), we have $\phi'^{A}(\widehat{x})=a^{-1}\phi'^{A}_{z}(\widehat{x})$ and $\pi'^{A}=\left(\pi^{E}_{A}\right)^{a}=a^{-1}\pi_{A}^{E}a$ for some $a\in A$ and the measurable map $\phi'^{A}_{z}:\widehat{S}\rightarrow Z(\pi_{A}^{E}(G))$. 

Now the equation (\ref{c0}) is 
\begin{align*}\label{c3}
\pi_{A}&=\phi_{z}'^{A}(g.\widehat{x})^{-1} a\pi'^{A}\Esc'^{A}u'^{A} a^{-1}\phi_{z}'^{A}(\widehat{x})\\
&=\phi_{z}'^{A}(g.\widehat{x})^{-1}\pi^{E}_{A}\Esc''^{A}u''^{A}\phi_{z}^{A}(\widehat{x})
\end{align*}
where $\Esc''^{A}=\left(\Esc'^{A}\right)^{a^{-1}}$ and $u''^{A}=\left(u'^{A}\right)^{a^{-1}}$.
 
 On the other hand, using (\ref{c0}), $$\beta=\phi'^{H}(g.\widehat{x})^{-1}\pi_{A}\pi'^{F}\Esc'^{F}u''^{H}\phi'^{H}(\widehat{x})$$ where 
$u''^{H}(g,\widehat{x})=\phi'^{A}(\widehat{x})^{-1}u'^{H}\phi'^{A}(\widehat{x})$. 

Then we know that 
\begin{enumerate}
\item the $u'^{H}$ commutes with the $\pi'^{A}(G)=a^{-1}\pi_{A}^{E}a$ and $\pi'^{F}(G)$ by dynamical super-rigidity. Therefore, one can directly check that the $u''^{H}$ commutes with $\pi_{A}^{E}(G)$ and $\pi'^{F}(G)$. 
\item the $\Esc'^{F}$ commutes with $A$ and $\pi'^{F}(G)$ by construction and dynamical super-rigidity.
\end{enumerate}

Find $g_{1},\dots,g_{l}\in D$ such that the group $\Fcal$ generated by their polar parts is Zariski dense in $G$ by the Lemma \ref{Fzd}. We can define the space $\Hom(\Fcal,L)$ with compact-open topology. Note that $\pi_{A}^{E}\pi'^{F}$ and $\pi_{0}^{E}$ can be thought elements in $\Hom(\Fcal,L)$. 

Fix the finite subset $\Pi\subset \Hom(G,L)$ of rational homomorphism from $G$ to $L$ such that any rational homomorphism from $G$ to $L$ is conjugate to the element in $\Pi$. This can be done since the rational homomorphism from $G$ to $L$ is finite up to conjugation. We may assume that $\pi_{0}^{E},\pi',\pi'^{A},\pi'^{F},\pi_{A}^{E}$ and $\pi_{H}^{E}$ are all in $\Pi$.

Using dynamical super-rigidity, we can deduce that
\begin{equation}\label{c1}
\Sp(\beta)=\Sp(\pi_{A}^{E}\pi'^{F})\textrm{ and } \phi'^{H}(\widehat{x})W(\beta_{g}(x))=W(\pi_{A}^{E}(g)\pi'^{F}(g))
\end{equation}
where $W(\cdot)$ is Lyapunov subspaces for generic $x\in S$.

The (\ref{c1}) implies that
\begin{equation}\label{c2}
W(\beta_{g}(x))=\phi'^{H}(\widehat{x})^{-1}W(\pi_{A}^{E}(g)\pi'^{F}(g))=W((\pi_{A}^{E}\pi'^{F})^{\phi'^{H}(\widehat{x})}(g)).
\end{equation}

Note that we can pretend
\begin{equation*} (\pi_{A}^{E}\pi'^{F})^{\phi'^{H}(\widehat{x})}=\phi'^{H}(\widehat{x})^{-1}(\pi_{A}^{E}\pi'^{F})\phi'^{H}(\widehat{x})\in \Hom(\Fcal, L).\end{equation*}

If $\beta$ is sufficiently $L^{\infty}$-close to $\pi_{0}$, $W(\beta_{g}(x))$ is sufficiently close to $W(\pi_{0}^{E}(g))$. Since $\Fcal$ is generated by the polar part of $g_{1},\dots,g_{l}\in D$, this implies that $(\pi_{A}^{E}\pi'^{F})^{\phi'^{H}(\widehat{x})}$ is sufficiently close to $\pi_{0}^{E}$ in $\Hom(\Fcal,L)$.

We have a $H$ action on $\Hom(\Fcal,L)$ via conjugation. Recall that the $H$-orbits in $\Hom(\Fcal,L)$ is closed by theorem \ref{orbit}. Furthermore, the finiteness of $\Pi$  implies that $\pi_{0}^{E}$ and $(\pi_{A}^{E}\pi'^{F})^{\phi'^{H}(\widehat{x})}$ is in the same $H$-orbit. Especially, as $\phi'^{H}$ takes values in the $H$, $\pi_{0}^{E}$ and $\pi_{A}^{E}\pi'^{F}$ is conjugate by the element in $H$. (as elements in $\Hom(\Fcal,L)$.) Since $\Fcal$ is Zarski dense in $G$, we can conclude $\pi_{0}^{E}$ and $\pi_{A}^{E}\pi'^{F}$ are conjugated by the element $h_{0}$ in $H$. In other words, we have $\pi_{0}^{E}=(\pi_{A}^{E}\pi'^{F})^{h_{0}}$.

Let $\Phi$ be a measurable map $\Phi:\widehat{S}\rightarrow \Hom(\Fcal,L)$ as $\Phi(\widehat{x})=(\pi_{A}^{E}\pi'^{F})^{\phi'^{H}(\widehat{x})^{-1}}$.
Since $\Phi$ takes values in one $H$-orbit and $\pi_{0}$ is in that $H$-orbit, we can consider the measurable map $\widetilde{\Phi}:\widehat{S}\rightarrow Z\backslash H$ where $Z$ is the stabilizer of $\pi_{0}$ in $H$ and $\widetilde{\Phi}(\widehat{x})=[h_{0}^{-1}\phi'^{H}(\widehat{x})].$ Note that $Z=Z_{L}(\pi_{0}^{E}(\Fcal))\cap H=Z_{L}(\pi_{0}^{E}(D))\cap H$. Note that, when we write the finite extension space as $\widehat{S}=S\times_{i}I$, (\ref{c2}) says that $\widetilde{\Phi} $ does not depend on $I$. Therefore, we can find measurable maps $f:S\rightarrow H$ and $z:\widehat{S}\rightarrow Z$ such that $\phi'^{H}(\widehat{x})=h_{0}z(\widehat{x})f(x).$  Here $f$ is $L^{\infty}$-close to $id_{H}$ since $(\pi_{A}^{E}\pi'^{F})^{\phi'^{H}(\widehat{x})}$ is sufficiently close to $\pi_{0}^{E}$ in $\Hom(\Fcal,L)$.

Combining these, again abbreviating notations, we can deduce 
\begin{align*}
\beta&=\phi'^{H}(g.\widehat{x})^{-1}\pi_{A}\pi'^{F}\Esc'^{F}u''^{H}\phi'^{H}(\widehat{x})\\
&= f(g.x)^{-1}z(g.\widehat{x})^{-1}h_{0}^{-1}\pi_{A}\pi'^{F}\Esc'^{F}u''^{H}h_{0}z(\widehat{x})f(x)\\
&=f(g.x)^{-1} z(g.\widehat{x})^{-1}h_{0}^{-1}\pi_{A}^{E}\pi'^{F}h_{0}h_{0}^{-1}\pi_{A}^{K}h_{0} h_{0}^{-1}\Esc'^{F}h_{0} h_{0}^{-1}u''^{H}h_{0}z(\widehat{x})f(\widehat{x})\\
&=f(g.x)^{-1}z(g.\widehat{x})^{-1}\pi_{0}^{E}\left(\pi_{A}^{K}\right)^{h_{0}}\Esc''^{F}u'''^{H}z(\widehat{x})f(\widehat{x})\\
&=f(g.x)^{-1}\pi_{0}^{E}\pi_{A}^{K}\left(\pi_{A}^{K}\right)^{-1}z(g.\widehat{x})^{-1}\left(\pi_{A}^{K}\right)^{h_{0}}\Esc''^{F}u'''^{H}z(\widehat{x})f(x)
\end{align*}
where $\Esc''^{F}=(\Esc'^{F})^{h_{0}}$ and $u'''^{H}=(u''^{H})^{h_{0}}$. Note that $\Esc''^{F}$ and $u'''^{H}$ commutes with $\pi_{0}^{E}$ since $u''^{H}$ and $\Esc'^{F}$ commutes with $\pi_{A}^{E}\pi'^{F}$.

Let $\widetilde{z}:D\times \widehat{S}\rightarrow Z$ be the measurable map as
$$\widetilde{z}(g,\widehat{x})=\left(\pi_{A}^{K}(g)\right)^{-1}z(g.\widehat{x})\left(\pi_{A}^{K}\right)^{h_{0}}(g)\Esc''^{F}u'''^{H}z^{-1}(\widehat{x})$$
for any $g\in D$ and almost every $\widehat{x}\in \widehat{S}$. Note that it is easy to see that $\widetilde{z}$ takes values in $Z$. Indeed, $\pi_{0}^{E}$ commutes with $\left(\pi_{A}^{K}\right)^{h_{0}}$ since $\pi_{A}^{K}$ commutes with $\pi_{A}^{E}\pi'^{F}$. Therefore $\widetilde{z}\in Z_{L}(\pi_{0}^{E})$. Furthermore, $\widetilde{z}$ takes values in $H$ as $H$ is a normal subgroup of $L$.

Since $f$ is $L^{\infty}$-close to $\textrm{id}_{H}$, we can conclude that $\widetilde{z}$ is $L^{\infty}$-close to $\pi_{H}^{K}$. This implies that for each $g\in D$, $$\sup_{\widehat{x}\in \widehat{S}} ||\widetilde{z}(g,\widehat{x})\pi_{H}^{K}(g)^{-1}|| $$ is sufficiently small. Now we claim that $\widetilde{z}=\pi_{H}^{K}$.  Indeed, we have $\widetilde{z}(g,\widehat{x}),\pi_{H}^{K}(g)\in Z$ so that possible values of $\widetilde{z}(g,\widehat{x})\pi_{H}^{K}(g)^{-1}$ is finite. This implies that if $\beta$ is $L^{\infty}$ sufficiently close to $\beta_{0}=\pi_{0}$ then $$\widetilde{z}(g,\widehat{x})=\pi_{H}^{K}(g)$$ for any $g\in D$ and almost every $x\in \widehat{S}$. This proves that if $\beta$ is sufficiently close to $\beta_{0}$, then there is a $L^{\infty}$-small measurable map $f:S\rightarrow H$ such that for any $g\in D$ we have $$\beta(g,x)=f(g.x)^{-1}\pi_{0}(g)f(x)$$ almost every $x\in S$.

Finally, when the $S$ is topological space and $\mu$ is fully supported, we can deduce $f$ is continuous map. Indeed, we already know that the Lyapunov spectrum of the $\beta_{g}$ is same as $\pi_{0}^{E}$ at $\widehat{x}\in \widehat{S}$. Furthermore, the Lyapunov subspaces of $\pi_{0}^{E}(g)$ and $\beta_{g}$ at $x\in S$ differ by $\phi'^{H}(\widehat{x})$. If $\beta$ is sufficiently $C^{0}$-close to $\beta_{0}=\pi_{0}$, then the difference modulo stabilizer changes continuously. Especially, this implies that the $f$ is continuous.
\end{proof}

\subsection{$C^{0}_{r,ind}$-local rigidity}
Let $D=G$ or $\Gamma$. 
Let $\Acal_{0}$ be an affine $D$ action on $H/\Lambda$.

The following lemma allows us to apply local rigidity of the constant cocycle.
\begin{lem}[\cite{MQ,FM2}]
Under the above notations and assuptions, especially assuming $\Acal_{0}$ is weakly hyperbolic, the central foliation is trivial. That is $Z=Z_{L}(\pi_{0}^{E}(D))\cap H$ is finite group.

\end{lem}

\begin{proof} Indeed, special type of affine actions for this lemma is mentioned in \cite{MQ}. The arguments in the section $2$ in \cite{FM2} show that along $Z$-orbit on $H/\Lambda$, $\Acal_{0}$ acts isometrically. If the $Z$ has positive dimension, then this gives contradiction with the weakly hyperbolic assumption. This implies that the $Z$ is discrete. Since $H$ is algebraic, $Z$ is finite. 
\end{proof}

We will prove following topological local rigidity theorem in this subsection. This will give the proof of the $C^{0}$-part in the theorem \ref{thm:locrigidgen}.
\begin{theorem}\label{thm:C0} Let $\Acal$ be a $D$ action on $H/\Lambda$ such that 
\begin{enumerate}
\item sufficiently $C^{1}$-close to $\Acal_{0}$,
\item there is a fully supported $\Acal$-invariant Borel probability measure and $\Acal$ is weakly irreducible (resp. $\Acal\restrict{\Gamma'}$ is induced weakly irreducible) when $D=G$ (resp. when $D=\Gamma$). 
\end{enumerate}
Then there is a homeomorphism $\lambda:H/\Lambda\rightarrow H/\Lambda$ such that $\Acal$ and $\Acal_{0}$ is conjugate via $\lambda$. In other words, for any $g\in D$, $\lambda\circ\Acal(g)\circ\lambda^{-1}(x)=\Acal_{0}(g)(x).$
\end{theorem}

\begin{proof} 
\textbf{$D=G$ case :} Note that when $D=G$, the affine action $\Acal_{0}$ comes from the left translation via homomorphism $\pi_{0}$ from $G$ to $H$. Recall that we have the continuous cocycle over $\Acal$-action $$\beta_{A}:G\times H/\Lambda \rightarrow H$$ such that 
\begin{enumerate}
\item $\Acal(g)[x]=\beta_{A}(g,x)[x]$.
\item $\beta_{A}$ is $C^{0}$ close to $\pi_{0}=\pi_{H}$.
\end{enumerate}

Note that in this case, we assume that perturbed $G$ action $\Acal$ is weakly irreducible with respect to fully supported Borel probability measure. Therefore, we can use the theorem \ref{thm:locrigidcoc} with trivial $A$. Therefore, we can find a continuous map $\phi: H/\Lambda \rightarrow H$ such that $$\beta_{A}(g,[x])=\phi(\Acal(g)([x]))^{-1}\pi_{0}(g)\phi([x]).$$ 
If we define $\lambda:H/\Lambda\rightarrow H/\Lambda$ as $$\lambda([x])=[\phi([x])x]$$ then simple calculation shows that $$\lambda(\Acal(g)[x])=\pi_{0}(g)\lambda([x]).$$

Note that the $\lambda$ is $C^{0}$-closed to $\textrm{id}_{H/\Lambda}$. Especially, the $\lambda$ is onto. Furthermore, using the same arguments about expansiveness of the weakly hyperbolic action $\Acal$ as in \cite{MQ}, we can prove that the $\lambda$ is indeed homeomoprhism.

\begin{rmk} When $D=G$, one also can prove the appropriate version of local rigidity of the constant cocycle in \cite{MQ} as mentioned in the Remark \ref{rmk:compMQ}. Then one can directly apply proofs of Theorem 1.2 (1) in \cite{MQ}. 
\end{rmk}

\textbf{$D=\Gamma$ case :} 
Recall that if $\Acal$ is sufficiently $C^{1}$-closed to $\Acal_{0}$ then we have the continuous cocycle $\beta_{A}$ over $\Acal\restrict{\Gamma'}$-action 
$\beta_{A}:\Gamma'\times H/\Lambda \rightarrow A\ltimes H$ such that 
\begin{enumerate}
\item $\beta_{A}(\gamma,[x])=\pi_{A}(\gamma)h(\gamma,[x])$ where homomorphism $\pi_{A}:\Gamma'\rightarrow A$ and continuous map $h:\Gamma'\times H/\Lambda\rightarrow H$.
\item The action $\Acal\restrict{\Gamma'}$ is given by $\Acal\restrict{\Gamma'}(\gamma)(x)=\beta_{A}(\gamma,[x])[(1,x)]$ after identifing $H/\Lambda=(\pi_{A}(\Gamma') \ltimes H)/(\pi_{A}(\Gamma')\ltimes \Lambda)$. 
\item $\beta_{A}$ is $C^{0}$-closed to $\pi_{0}$ for (finite) generators $\{\gamma_{i}\}\subset \Gamma'$.
\end{enumerate}

Note that
\begin{enumerate}
\item $\Acal_{0}\restrict{\Gamma'}$ is still weakly hyperbolic,
\item $\Acal\restrict{\Gamma'}$ is induced weakly irreducible and
\item $\Gamma'$ is still a weakly irreducible lattice in $G$. 
\end{enumerate}
Therefore, we can use the theorem \ref{thm:locrigidcoc}. Then we have continuous map $\phi:H/\Lambda\rightarrow H$ such that $$\beta_{A}(\gamma,[x])=\phi(\Acal(\gamma)([x]))^{-1}\pi_{0}(\gamma)\phi([x])$$  
for any $\gamma\in\Gamma'$. Define continuous map $\lambda:H/\Lambda\rightarrow H/\Lambda$ as $$\lambda([x])=[\phi([x])x].$$ 
Then direct calculations say that for any $\gamma\in \Gamma'$ we have 
\begin{align*}
\lambda\left(\Acal(\gamma)(x)\right)&=\phi\left(\Acal(\gamma)([x])\right)\Acal(\gamma)([x])\\
&=\pi_{0}(\gamma)[\phi([x])x]\\
&=\pi_{0}(\gamma)\lambda(x)
\end{align*}
Furthermore, the $\lambda$ is $C^{0}$-close to $id_{H/\Lambda}$ so it is onto. As in the $D=G$ case, the same expansiveness arguments in \cite{MQ} prove that $\lambda$ is indeed homeomorphism. This proves $\Acal\restrict{\Gamma'}$ is $C^{0}$ conjugate to $\Acal_{0}\restrict{\Gamma'}$. 

For conjugacy between full $\Gamma$-action $\Acal$ and $\Acal_{0}$, the arguments in \cite{MQ} can be directly applied. Recall that we already assume that $\Gamma'$ is a normal subgroup in $\Gamma$.

Using expansiveness property (\cite{MQ} Lemma 4.6.) and the fact that $\Gamma'$ is normal subgroup of $\Gamma$, we can conclude  $\lambda\circ \Acal(\gamma) \circ \lambda^{-1}=\Acal_{0}(\gamma)$ for all $\gamma\in \Gamma$. (See the Step 2 in the proof of $C^{0}$-local rigidity of $\rho_{A}$ in \cite{MQ}.)
\end{proof}

\subsection{$C^{\infty}_{r,ind}$-local rigidity}

We used $C^{0}_{r,ind}$-local rigidity, so that find $C^{0}$-conjugacy between smooth affine action $\Acal_{0}$ and perturbed action $\Acal$. When we have $C^{0}$-conjugacy $\lambda$, $C^{\infty}_{r,ind}$-local rigidity follows from showing that the conjugacy $\lambda$ is indeed $C^{\infty}$. 

In this section $\Acal$ will be $C^{\infty}$-volume preserving action on $H/\Lambda$ that is $C^{1}$-close to $\Acal_{0}$.

\begin{theorem}\label{thm:Cinf}
Let $\lambda$ is a $C^{0}$-conjugacy map $\lambda:H/\Lambda\rightarrow H/\Lambda$ between the affine action $\Acal_{0}$ and $C^{\infty}$-action $\Acal$ that is sufficiently $C^{1}$-closed to $\Acal_{0}$. Then $\lambda$ is indeed $C^{\infty}$.
\end{theorem}

The arguments for proving smoothness appears in \cite{MQ} and \cite{FM2} starting from \cite{KS97}. They used the normal form theory in order to prove smoothness.  Even in our setting, we can use same proofs. Recall that it is enough to prove the analogue of the Theorem 5.8 in \cite{MQ} in our setting. Indeed, following \cite{MQ}, Theorem 5.8 is the first part of the proof. The second part, however, can be used directly. We appeal to use materials in section 6 of \cite{FM2}. Especially, Theorem 6.5 in \cite{FM2} is dealing with affine action so that it is enough to prove Theorem \ref{thm:Cinf}. Note that \cite{MQ} and \cite{FM2} used the results by Prasad and Rapinchuk in \cite{PR01}. The results in \cite{PR01} still hold in our setting so that we can use same results.


\section{Global rigidity}\label{sec:globrigid}
 \subsection{Setting}
 
We will follow notations in \cite{BRHW} throughout section for reader's convenience although this may not coincide with notations with previous sections. 
 For this section, $G$ will be the connected real semisimple higher rank algebraic Lie group without compact factor. Let $\Gamma$ be a weakly irreducible lattice in $G$. Let $M=N/\Lambda$ be a compact nilmanifold, that is $N$ is simply connected nilpotent Lie group and the $\Lambda$ is a lattice (necessarily cocompact) in $N$. We will denote $\nlie$ for Lie algebra of $N$.

 Let $\alpha$ be a $C^{1}$- (resp. $C^{\infty}$-)$\Gamma$ action on $M$ by diffeomorphism. Assume that there is $\gamma_{0}\in \Gamma$ such that $\alpha(\gamma_{0})$ is Anosov diffeomorphism on $M$. In other words, there is an element $\gamma_{0}$ that is an Anosov elemtent. We will always assume that $\alpha$ lifts action on universal cover $N$. Then $\rho:\Gamma\rightarrow \Aut(N)$, or abusing notation $\rho:\Gamma\rightarrow \Aut(N/\Lambda)$ be the associated linear data of $\alpha$.

In this section, we will prove that $\alpha$ is $C^{0}$ conjugate to (resp. $C^{\infty}$ conjugate to) $\rho : \Gamma\rightarrow \Aut(N/\Lambda)$.

 \subsection{$C^{0}$-global rigidity} 
In this subsection, $\alpha$ will be a $C^{1}$-action, that is $\alpha:\Gamma\rightarrow \Diff^{1}(M)$. 

 \begin{theorem}\label{thm:C0glob}
 Let $\alpha$ is $C^{1}$-Anosov $\Gamma$-action on the nilmanifold $M=N/\Lambda$. We will denote $\gamma_{0}\in\Gamma$ for the Anosov element. Also, $\nlie$ will be the Lie algebra of $N$. Assume that
 \begin{enumerate}
 \item the action $\alpha$ lifts to the universal cover $N$ and
 \item there is a fully supported Borel probability measure $\mu$ so that $\alpha$ is weakly induced irreducible with respect to $\mu$.
 \end{enumerate}
 Then there is a homeomorphism $\lambda:M\rightarrow M$ that is $C^{0}$-close to $id$ such that $$\lambda\circ \alpha(\gamma)\circ \lambda^{-1} = \rho(\gamma)$$ where $\rho$ is the linear data of $\alpha$.
 
 \end{theorem}
 Note that we may assume without loss of generality, $G$ and $\Gamma$ satisfies \ref{std}  if we think about algebraic simply connected covering of $G$ and preimage of $\Gamma$ on it.

From now on, we will assume the conditions in the theorem \ref{thm:C0glob}. The lemma \ref{liftcoc} says that, in this case, we can find the cocycle $\beta$ over $\alpha$ such that 

\begin{enumerate}
\item $\beta:\Gamma\times N/\Lambda\rightarrow A\ltimes N$ is the continuous cocycle defined over $\alpha$ where $A$ is Zariski closure of $\rho(\Gamma)$ in $\Aut(N/\Lambda)$. Note that $A$ is semisimple due to Margulis' super-rigidity theorem.
\item For any $\gamma\in\Gamma$, for any $[x]\in N/\Lambda$ we have  $\alpha(\gamma)[x]=\beta(\gamma,[x])[1,x]$ after identify $N/\Lambda$ with $(\rho(\Gamma)\ltimes N)/(\rho(\Gamma)\ltimes \Lambda)$
\item $A$-component of $\beta$ is $\rho$.
\end{enumerate}
Note that the linear data $\rho$ acts on $(\rho(\Gamma)\ltimes N)/(\rho(\Gamma)\ltimes \Lambda)$ after identify with $N/\Lambda$ via $[1,n]\mapsto \rho(\gamma)[1,n]$.

 As before, we will denote $L=A\ltimes N$ and fix embedding $L$ into $\GL(d,\Rbb)$. Also, we will use notations $\rho^{E}$ and $\rho^{K}$ for $\rho$ from Margulis' super-rigidity. 
 
\begin{prop}\label{glob:unitriv}  
 Assume the above notations and settings. Then the algebraic hull of $\beta$ belongs to $A$. Moreover, there is a measurable map $\lambda':\widehat{M}\rightarrow N$ such that $\beta^{\lambda'^{-1}}=\rho$ where $\widehat{M}$ is a finite extension.
\end{prop}
\begin{proof}  Note that we may find algebraic group $\Ubb$ defined over $\Rbb$ such that $\Ubb(\Rbb)=N$. 
Let $\beta=\rho\cdot  u_{0}$ where $u_{0}:\Gamma\times M\rightarrow N$ is measurable map.

On the other hand, using dynamical cocycle super-rigidity for $\beta$, after finite extension, we can say that 
\begin{equation}\label{eqv}
\beta(\gamma,x)=\phi'(\widehat{\alpha}(\gamma)(x,m))^{-1}\pi'(\gamma)\Esc'(\gamma,x,m)u'(\gamma,x,m)\phi'(x,m)
\end{equation}
for some measurable map $\phi':\widehat{M}\rightarrow L$, a rational homomorphism $\pi':\Gbb\rightarrow \Lbb$, a measurable cocycle $\Esc': \Gamma\times \widehat{M}\rightarrow L$ and a measurable map $u':\Gamma\times \widehat{M}\rightarrow L$ where $\widehat{M}$ is a finite extension of $M$. Note that $\pi'$ and $\Esc'$ takes values in Levi component of the algebraic hull of $\beta$. This means that $\pi'$ and $\Esc'$ takes values in $A$. Denote $\phi'^{A},\Esc'^{A},u'^{A}$ and $\pi^{A}$ as $A$- component of $\phi',\Esc',u'$ and $\pi$ respectively. We will also denote for $N$-components of $\phi',\Esc',u'$ and $\pi$ similarly. Note that $\pi'^{N}$ and $\Esc'^{N}$ is trivial.

This means that abbreviately $$\beta=\rho\cdot u_{0}= \phi'^{-1}\pi'\Esc'u'\phi'.$$ We can decompose $u'=u'^{A}u'^{N}$. We claim that $u'^{N}$ is trivial. 
Note that $\beta$ is bounded and for any $\gamma\in \Gamma$, $\Sp(\beta_{g})=\Sp(\rho(g))$ since $u$ takes values in the unipotent radical $N$. Furthermore, using dynamical super-rigidity, $\Sp(\beta_{g})=\Sp(\pi'(g))$. Uniqueness of super-rigidity homomorphism implies that $\rho^{E}$ is conjugate to $\pi'$. Therefore, after changing $\Esc'$ and $u'$ to conjugation, we may assume that $\pi'=\rho^{E}$ and $u'(\Gamma\times \widehat{M})$ hence $u'^{N}(\Gamma\times \widehat{M})$ commutes with $\rho^{E}$. 

Assume that $u'^{N}$ is not trivial. Then we can find $X\in \nlie$ such that $\exp(X)\in N$ is contained in the essential image of $u'^{N}$. Then $\rho^{E}(\gamma_{0})(\exp(X))=\exp(X)$ since $\rho^{E}(\gamma_{0})$ commutes with $\exp(X)$ in $L$. If we derivative the equations then we get $D\rho^{E}(\gamma_{0})(X)=X$. This implies that $D\rho^{E}(\gamma_{0})$ is not hyperbolic, so that $D\rho(\gamma_{0})$ can not be hyperbolic. As we assume that $\alpha(\gamma_{0})$ is an Anosov diffeomorphism, it gives contradiction. Therefore, $u'^{N}$ is trivial.

Let denote $\lambda'=\phi'^{N}$. Simple calculation shows followings.
\begin{align*}
\beta(\gamma,x)^{\lambda'^{-1}}&=\underbrace{\left(\phi'^{A}(\widehat{\alpha}(\gamma)(x,m))\right)^{-1} \pi'^{-1}(\gamma)\Esc'(\gamma,x,m)u'^{A}(\gamma,x,m)\phi'^{A}(x,m)}_{\in A}\\
&=\lambda'(\widehat{\alpha}(\gamma)(x,m))^{-1}\rho(\gamma)u_{0}(\gamma,x)\lambda'(x,m)\\
&=\underbrace{\rho(\gamma)}_{\in A}\underbrace{\rho(\gamma)^{-1}\lambda'(\widehat{\alpha}(\gamma)(x,m))\rho(\gamma)u_{0}(\gamma,x)\lambda'(x,m)^{-1}}_{\in N}
\end{align*}

Comparing $A$-component of the above equations, we can conclude that $\beta^{\lambda'^{-1}}=\rho$. 
\end{proof}

The above propostion \ref{glob:unitriv} should be compared with the Lemma 6.2 in \cite{MQ}. Now the theorem \ref{thm:C0glob} comes from the arguments in the proof of the Theorem 1.3 in \cite{MQ} directly. Note that the map $\lambda:\widehat{M}\rightarrow M$, $\lambda(\widehat{x})=\lambda'(\widehat{x})x$ satisfies $$\lambda(\widehat{\alpha}(\gamma)(\widehat{x}))=\rho(\gamma)(\lambda(\widehat{x})).$$ Furthermore, when we denote $\widehat{M}=M\times_{i} I$, $\lambda$ does not depend on $I$ and continuous on $M$. This implies $\lambda$ gives continuous semi-conjugacy between $\alpha$ and $\rho$. (See Lemma 6.4 and 6.5 in \cite{MQ}.) Finally, an Anosov element provides that $\lambda$ is indeed a homeomorphism. 

In our case, we have advantage that $\alpha$ is ergodic with respect to $\mu$. Therefore, we don't need ergodic component arguments in \cite{MQ}. 

\subsection{$C^{\infty}$-global rigidity}\label{sec:Cinfglobrigid}
In this subsection we will assume that $\alpha$ is $C^{\infty}$ action, that is $\alpha:\Gamma\rightarrow \Diff^{\infty}(M)$. It is enough to prove following theorem. As before, using algebriacally simply connected covering, we may assume that $G$ and $\Gamma$ satisfy \ref{std}. We will denote $\gamma_{0}\in \Gamma$ be an element in $\Gamma$ such that $\alpha(\gamma_{0})$ is an Anosov. We have the linear data $\rho$ of $\alpha$ that is the homomorphism $\rho:\Gamma\rightarrow \Aut(N)$. Theorem \ref{thm:C0glob} gives us a homeomorphism $h:M\rightarrow M$ such that $h\circ\alpha(\gamma)=\rho(\gamma)\circ h$. We will prove following theorem in this section. Note that theorem \ref{thm:C0glob} and \ref{Cinf} gives the theorem \ref{thm:globalrigid}.

\begin{theorem}\label{Cinf}
Let $\alpha$ be a $C^{\infty}$ action on $N/\Lambda$. Assume that 
\begin{enumerate}
\item The $\alpha$ lifts on the universal cover $N$. Let $\rho$ be the linear data of $\alpha$.
\item The $\alpha$ is $C^{0}$-conjugate to its linear data $\rho$. 
\item There is $\gamma\in\Gamma$ such that $\alpha(\gamma)$ is an Anosov diffeomorphism.
\end{enumerate}
Then $\alpha$ is $C^{\infty}$-conjugate to $\rho$.

\end{theorem}

We will prove the above theorem \ref{Cinf} as follows. First, we will find a good finite index subgroup $\Gamma_{0}$ in the $\Gamma$ in order to use dynamical super-rigidity. Note that we can use the $\lambda$ as the $C^{0}$ conjugacy between $\alpha\restrict{\Gamma_{0}}$ and $\rho\restrict{\Gamma_{0}}$. The main reason for finding the finite index subgroup $\Gamma_{0}$ is due to induced weakly irreducibility.  After that, following the \cite{BRHW}, we find a free abelian subgroup $\Sigma$ of $\Gamma_{0}$ that has an Anosov element and acts without rank one factor. 
Using \cite{RHW}, we can say that $\alpha\restrict{\Gamma_{0}}$ is $C^{\infty}$-conjugate to $\rho\restrict{\Gamma_{0}}$. Then the theorem \ref{Cinf} holds as in the Remark 1.8 in the \cite{BRHW}. (e.g. the arguments in \cite{MQ}) 

$\lambda$ is indeed $C^{\infty}$ due to the Anosov element.

\noindent\textbf{Step 1: Finding $\Gamma_{0}$.}

The followings are variants from the \cite{BRHW} section 2,4 and 6.  
\begin{prop}\label{rmkkey} We have followings.
\begin{enumerate}
\item Using Margulis' arithmeticity theorem, we may find semisimple algebraically simply connected algebraic group $\Hbb$ defined over $\Qbb$ and a surjective algebraic morphism $\phi:\Hbb\rightarrow \Gbb$ such that $\phi:\Hbb(\Rbb)^{0}\rightarrow G$ has compact kernel and $\phi(\Hbb(\Zbb)\cap \Hbb(\Rbb)^{0})$ is commensurable with $\Gamma$. Furthermore, every $\Qbb$-simple factor has $\Rbb$-rank $2$ or higher since we assume that $\Gamma$ is weakly irreducible lattice. 
\item  Define $\widehat{\Gamma}=\phi^{-1}(\Gamma)\cap \Hbb(\Zbb)$, $\widehat{\alpha}=\alpha\circ\phi$ and $\widehat{\rho}=\rho\circ \phi$. Then $\widehat{\Gamma}$ action $\widehat{\alpha}$ and $\widehat{\rho}$ factor through $\Gamma$ action $\alpha$ and $\rho$ respectively. Here $\widehat{\rho}$ is the linear data of $\widehat{\alpha}$.
\item There is a $\Qbb$-structure of $\nlie$ such that $D\rho$ sends $\Gamma$ into $\GL(d,\Qbb)$. This implies $D\widehat{\rho}$ sends $\widehat{\Gamma}$ to $\GL(d,\Qbb)$. Using Margulis' super-rigidity for arithmetic lattice, we may find a finite index subgroup $\widehat{\Gamma}'\subset \widehat{\Gamma}$ such that $D\widehat{\rho}\restrict{\widehat{\Gamma}'}$ extends to a $\Qbb$-representation $\tau:\Hbb\rightarrow \mathbb{GL}(d)$.

\item Find a finite index subgroup $\widehat{\Gamma}''$ in $\widehat{\Gamma}'$ such that $\widehat{\Gamma}''$ is a product of irreducible lattices. In other words, we can write $H$ as $H=H_{1}\times\dots\times H_{k}$ by semisimple groups so that $\widehat{\Gamma}''=\prod_{j=1}^{k} \widehat{\Gamma}''_{j}$ where $\widehat{\Gamma}''_{j}$ is an irreducible lattice in $H_{j}$.
\item Since we assume that $\Gbb$ do not have $\Rbb$-anisotropic simple factor and $\Gamma$ be a weakly irreducible lattice, $\widehat{\Gamma}''$ is Zariski dense in $\Hbb$ by Borel density theorem. 
\item  Let $\Gamma_{0}=\phi(\widehat{\Gamma}'')$. Then $\Gamma_{0}$ is a finite index subgroup of the $\Gamma$. Furthermore, up to powers of $\gamma_{0}$, we may assume that $\gamma_{0}\in\Gamma_{0}$.
\item 
 Using Ratner's theorem, we know that any orbit closure $\overline{\rho(\Gamma_{0})(x)}$ is the homogeneous sub-nilmanifold. Furthermore, due to the existence of $\gamma_{0}$, any Haar measure on the homogeneous sub-nilmanifold $\mu_{x}$ is $\rho\restrict{\Gamma_{0}}$ ergodic and the orbit closure $\overline{\rho(\Gamma_{0})(x)}$ is the support of $\mu_{x}$.
\end{enumerate}
 \end{prop}
  
The above proposition is stated or easily deduced from the statement in the \cite{BRHW}. Now we need following in order to use dynamical super-rigidity.

\begin{theorem}\label{keyCinf}
Under the above notations and settings, $\rho(\Gamma_{0})$-ergodic measure $\mu_{x}$ is induced weakly irreducible for every $x\in N/\Lambda$. 
\end{theorem}
\begin{proof}
Fix rank $1$-factor $F$ and its complement $F^{c}$ of $G$. Since $\mu_{x}$ is the Haar measure on the homogeneous sub-nilmanifold, without loss of generality, we may assume that $\mu_{x}$ is Haar measure $\mu$ on $N/\Lambda$ after changing $N/\Lambda$ to sub-nilmanifold.

We need to prove $F^{c}$ action on the induced $G$-space $\left(G/\Gamma_{0}\times N/\Lambda, m_{G/\Gamma_{0}}\otimes \mu\right)$ is ergodic where $m_{G/\Gamma_{0}}$ is Haar measure on $G/\Gamma_{0}$. Note that $m_{G/\Gamma_{0}}\otimes \mu$ is $G$-ergodic. 

Note that we can find $\widetilde{F^{c}}$ such that $\phi^{-1}(F^{c})=\widetilde{F^{c}}K$ for some compact group $K$ such that the $\widetilde{F^{c}}$ is the product of non-compact $\Rbb$-simple factors in $H$. Recall that the $\phi$ is a surjective continuous homomoprhism $\phi:H\rightarrow G$ evaluated $\phi$ to $\Rbb$ points. Furthermore, it is enough to show that the $\left(\widetilde{F^{c}}K\right)$-action is ergodic on 
$\left(H/\widehat{\Gamma}''\times N/\Lambda, m_{H/\widehat{\Gamma}''}\otimes \mu\right)$ as the $F^{c}$-action on the induced $G$-space $\left(G/\Gamma_{0}\times N/\Lambda, m_{G/\Gamma_{0}}\otimes \mu\right)$ is a factor of it.

Note that the induced $H$-space $\left(H/\widehat{\Gamma}''\times N/\Lambda, m_{H/\widehat{\Gamma}''}\otimes \mu\right)$ can be identified with the finite homogeneous space $\left(H\ltimes_{\tau} N\right)\Big/\left(\widehat{\Gamma}''\ltimes_{\tau} \Lambda\right)$ with the Haar measure $\widetilde{\mu}$ as $H$-spaces. Note that $\widehat{\Gamma}''\ltimes_{\tau} \Lambda$ is a lattice in $H\ltimes N$.

Therefore, it is enough to prove that the  $\left(\widetilde{F^{c}}K\right)$-action is ergodic on the $\left(H\ltimes_{\tau'} N\right)\Big/\left(\widehat{\Gamma}''\ltimes_{\tau} \Lambda\right)$ with respect to $\widetilde{\mu}$.

As $\widetilde{F^{c}}$ is subgroup of $\left(\widetilde{F^{c}}K\right)$, it is enough to prove that the $\widetilde{F^{c}}$-action is ergodic on the $\left(H\ltimes_{\tau} N\right)\Big/\left(\widehat{\Gamma}''\ltimes_{\tau'} \Lambda\right)$. 

Since $\widetilde{F^{c}}$ is a product of $\Rbb$-simple non-compact groups, we can find an $1$-parameter subgroup $\{a_{t}\}$ such that projects into the unbounded subgroup of $\Rbb$-split torus for each non-compact $\Rbb$-simple factors of $\widetilde{F^{c}}$. 

Now we appeal to use the result in \cite{Dani} (or \cite{BM} Theorem 6.1). 
\begin{theorem}[\cite{Dani}]\label{Dani} Let $J$ be connected Lie group. Assume that the Levi decomposition of $J$ is given by $J=H\ltimes R$ where $R$ is radical. Let $C$ be a lattice in $J$. Then the flow $g_{\Rbb}\subset J$ in the finite volume homogeneous space $J/C$ is ergodic if and only if the semisimple flow $(J/\overline{RC}, g_{\Rbb})$ and the solvable flow $(J/\overline{J_{\infty}C}, g_{\Rbb})$ is ergodic where $J_{\infty}$ is the group generated by $\textrm{Int}_{J}H=\{jHj^{-1}:j\in J\}$. Note that $J_{\infty}$ is the smallest normal subgroup containing $H$.
\end{theorem}
In our case, $J=H\ltimes N$ is a Levi decomposition since the $H$ is semisimple and the nilradical $N$ is radical. Therefore the maximal semisimple quotient of the $\left(H\ltimes_{\tau} N\right)\Big/\left(\widehat{\Gamma}''\ltimes_{\tau} \Lambda\right)$ is the $H/\widehat{\Gamma}''$. We will prove that the flow $\{a_{t}\}$ is ergodic on $H/\widehat{\Gamma}''$ and solvable quotient is a trivial.

\textbf{Maximal semisimple factor:} For maximal semisimple quotient, we will prove that the flow $\{\psi(a_{t})\}\subset \widetilde{F^{c}}$ is ergodic on $H/\widehat{\Gamma}''$ where $\psi:H\ltimes_{\tau}N\rightarrow H$ be the natural projection.

First we claim that the $\widetilde{F^{c}}$-action on $H/\widehat{\Gamma}''$ is ergodic. This is equivalent to $\widehat{\Gamma}''$ is ergodic on $H/\widetilde{F^{c}}$. Note that $H/\widetilde{F^{c}}\simeq F^{*} K$ where $F^{*}$ is rank $1$ $\Rbb$-simple factor in $H$ such that $\phi(F^{*})=F$. As $\Gamma_{0}$ is weakly irreducible lattice, $(F^{*}K)\cdot \widehat{\Gamma}''$ is dense in $H$. Therefore, $\widehat{\Gamma}''$ projects densely into $F^{*} K$. This implies that $\widehat{\Gamma}''$ is ergodic on $H/\widetilde{F^{c}}$, so that $\widetilde{F^{c}}$ acts ergodically on $H/\widehat{\Gamma}''$.
 
 Now we claim that the flow $\{\psi(a_{t})\}\subset \widetilde{F^{c}}$ is ergodic on $H/\widehat{\Gamma}''$. Indeed, this comes from Howe-Moore theorem (or Mautner phenomenon) and the fact that the flow $\{\psi(a_{t})\}$ projects into unbounded sbugroup in $\Rbb$-split torus for all $\Rbb$-simple factor of $\widetilde{F^{c}}$. 
 
 If we write $\widehat{\Gamma}''$ into the product of irreducible lattices as $\widehat{\Gamma}''=\prod\widehat{\Gamma}''_{j}$, we can easily reduce the problem into the irreducible lattice case. In this case, ergodicity implies mixing so that any unbounded subgroup acts ergodically. This implies that $\{\psi(a_{t})\}$ acts ergodically on $H/\widehat{\Gamma}''$. This proves that the flow is ergodic on maximal semisimple quotient.

 \textbf{Maximal solvable factor:} We will prove that there is no toral quotient. Indeed, we will prove that the group $J_{\infty}=J$. Note that there is $h\in H$ such that $D\widehat{\rho}(h)$ is hyperbolic in $\GL(\nlie)$. Furthermore, $\tau(h)(\exp(X))=\exp(D\widehat{\rho}(h)(X))$ for any $X\in \nlie$ where $\exp:\nlie\rightarrow N$ is the exponential map. For any $n\in N$, if we can find $m\in N$ such that $hnh^{-1}n=m$ then we can deduce $N\subset J_{\infty}$ since $J_{\infty}$ contains $H$. This implies that $J_{\infty}=J$. Equivalently, it is enough to prove that for any $Z\in\nlie$, there is a solution $X\in\nlie$ of the equation \begin{equation}\label{eqnil} \exp(D\widehat{\rho}(h)(X))=\exp(Z)\exp(X). \end{equation}

\emph{When $N$ is abelian:} In this case, $\exp(Z)\exp(X)=\exp(Z+X)$. Therefore it is enough to find $X$ such that $D\widehat{\rho}(h)(X)=X+Z$. Note that $D\widehat{\rho}(h)$ is hyperbolic so that $D\widehat{\rho}(h)-id_{\nlie}$ is invertible. Therefore, the $X=\left(D\widehat{\rho}(h)-id_{\nlie}\right)^{-1}(Z)$ is the desired solution.

 \emph{For general $N$:} In this case, we have Baker-Campbell-Hausdorff formula gives that $\exp(Z)\exp(X)=\exp(F(Z,X)).$ Here $$F(Z,X)=Z+X+\frac{1}{2}[Z,X]+\dots$$ is a finite sum since $N$ is nilpotent. Let $F_{k}(\cdot,\cdot)$ be the sum of terms in $F(\cdot,\cdot)$ that have $k-1$-brackets. For example, $F_{1}(Z,X)=Z+X$ and $F_{2}(X,Z)=\frac{1}{2}[Z,X]$. 
 
 Fix any $Z\in\nlie$. Let $\nlie_{0}=\nlie$, $\nlie_{1}=[\nlie,\nlie]$, and $\nlie_{k}$ is the subspace generated by the elements comes from $k$ brackets operations with elements in $\nlie$. Note that there is a minimal $r$ such that for any $k> r$, $\nlie_{k}=0$ since $N$ is nilpotent.  
 
 In addition, $\nlie_{k}$ is $D\widehat{\rho}(H)$ invariant subspaces for each $k$. Since $H$ is semisimple, $D\widehat{\rho}$ is fully reducible. Therefore we can find subspaces $U_{k}$ such that $\nlie_{k-1}=\nlie_{k}\oplus U_{k}$ as $H$-module. This gives a decomposition 
 $$\nlie=U_{1}\oplus\dots\oplus U_{r}$$ as $H$-modules. 
 
Let $p_{k}:\nlie\rightarrow \nlie/\left(U_{k+1}\oplus\dots\oplus U_{r}\right)=\overline{U_{1}\oplus\dots\oplus U_{k}}$ be the projection with respect to the above decomposition. 

We will find solution $X$ using induction on $k$. For $k=1$,  then we can find $\overline{X_{1}}\in \overline{U_{1}}$ such that
$\overline{X_{1}}$ is the solution of $D\widehat{\rho}(h)(\overline{X_{1}})=p_{1}(Z)+\overline{X_{1}}$ as in the abelian $N$ case. Find $X_{1}\in U_{1}$ such that $p_{1}(X_{1})=\overline{X_{1}}$. Then we have $p_{1}(D\widehat{\rho}(h)(X_{1}))= \overline{X_{1}}+p_{1}(Z)$.

Assume that we have $X_{1},\dots,X_{k}$ for $1\le k\le r-1$ such that 
$$p_{k}(D\widehat{\rho(h)}(X))=F_{1}(p_{k}(X),p_{k}(Z))+\dots+F_{k}(p_{k}(X),p_{k}(Z))$$ where $X=X_{1}+\dots+X_{k}$.
We need to find $X_{k+1}\in U_{k+1}$ such that 
\begin{equation}\label{nil} 
p_{k+1}(D\widehat{\rho(h)}(X'))=
F_{1}(p_{k}(X'),p_{k}(Z))+\dots+F_{k+1}(p_{k+1}(X'),p_{k+1}(Z))
\end{equation} where $\overline{X_{k+1}}=p_{k+1}(X_{k+1})$ and $X'=X_{1}+\dots+X_{k+1}$. As $U_{k+1}$ consists of elements of the form $k$ brackets operation with elements in $\nlie$ and $p_{k+1}$ delete $(k+1)$ bracket operations, the equation (\ref{nil}) is same as 
\begin{equation}\label{nil2}p_{k+1}(D\widehat{\rho}(h)(X_{k+1}))=\overline{X_{k+1}}+(\textrm{terms only involving $Z$ and $X$}).\end{equation}
Using hyperbolicity of $D\widehat{\rho(h)}$, we can deduce that $D\widehat{\rho}(h)-id_{\overline{U_{1}\oplus\dots\oplus U_{k+1}}}$ is invertible. Therefore inductively, we can find a solution $\overline{X_{k+1}}$ of the equation (\ref{nil2}). We can find $X_{k+1}\in U_{k+1}$ such that $p_{k+1}(X_{1}+\cdots+X_{k+1})=\overline{X_{1}}+\dots+\overline{X_{k+1}}$ and satisfies the desired equation (\ref{nil}).

This proves that we can find the solution of the equation (\ref{eqnil}) for any $Z\in\nlie$. Therefore, we proved that $J_{\infty}=J$ and there is no maximal solvable quotient.

As a result, we proved that the ergodicity of the $F^{c}$ action on the induced $G$-space $\left(G/\Gamma_{0}\times N/\Lambda, m_{G/\Gamma_{0}}\otimes \mu\right)$ so we proved induced weakly irreducibility of $\rho(\Gamma_{0})$ action on $(N/\Lambda, \mu_{x})$. \end{proof}
\begin{rmk} If $\Gamma$ is a non-uniform weakly irreducible lattice, then the proof is much easier. Up to finite index normal subgroup, we can lift $\rho$ to $G$ directly. Then one can use Theorem \ref{Dani} for $G$ itself in order to prove there is no non-trivial solvable quotient.
\end{rmk}

Now we can prove the followings that is same in the \cite{BRHW} Proposition 8.5.
\begin{prop}\label{keyglob}
Under the above notations and settings, there is a Zariski dense sub-semigroup $S$ in $\Gamma_{0}$ such that for every element $\gamma\in S$, $\alpha(\gamma)$ is an Anosov diffeomorphism. 
\end{prop}
Indeed, $S$ will be defined using cone of stable/unstable distribution of $\gamma_{0}$. The following is the key ingredient, that is the \cite{BRHW} Proposition 8.7.
\begin{prop}\label{keyp}
There is a Zariski open set $W\subset G$ such that for every $\eta\in \Gamma_{0}\cap W$, there is $N>0$ such that $\gamma_{0}^{N}\eta\gamma_{0}^{N}\in S$. 
\end{prop}
The above proposition \ref{keyp} is the place that the \cite{BRHW} used Zimmer's cocycle super-rigidity theorem. We can follow exactly same proof using dynamical super-rigidity instead of Zimmer's cocycle super-rigidity theorem thanks to the Theorem \ref{keyCinf}. 

After proving the proposition \ref{keyp}, we can deduce the proposition \ref{keyglob} easily as same as the proof of the Proposition 8.5 in the \cite{BRHW}.

\noindent\textbf{Step 2: Finding $\Sigma$.}

Recall that $\widehat{\Gamma}'\subset \Hbb(\Zbb)\cap H$ is an arithmetic lattice in $H$ such that $D\widehat{\rho}\restrict{\widehat{\Gamma}'}$ lifts to $\tau:\Hbb\rightarrow \GL(d)$. Note that $\Gamma_{0}\subset\phi(\widehat{\Gamma}')$. 

Let $\widehat{S}=\phi^{-1}(S)\cap \widehat{\Gamma}'$ and $\widehat{W}=\phi^{-1}(W)$. Then $\phi^{-1}(W)$ is Zariski-open in $H$. Fix $\widehat{\gamma_{0}}\in \phi^{-1}(\gamma_{0})$. Then theorem \ref{keyglob} shows that the conclusion in the theorem \ref{keyglob} is still true for $\widehat{S}$.

As in the \cite{BRHW} Proposition 8.15, using Prasad and Rapinchuk in \cite{PR03,PR05}, we can find Zariski dense subset $\widehat{S}'\subset \widehat{S}$ such that for any $\gamma\in \widehat{S}'$, 
 \begin{enumerate} 
 \item the identity component $Z_{\Hbb}(\gamma)^{0}$ of the centralizer $\gamma$ in $\Hbb(\Rbb)$ is a maximal $\Qbb$-torus, containing $\Rbb$-split torus of $\Hbb$, and $\gamma\in Z_{\Hbb}(\gamma)^{0}$.
 \item $\gamma$ is regular and $\Rbb$-regular.
 \item For $\Tbb=Z_{\Hbb}(\gamma)^{0}$, the Galois action contains all elements from Weyl group $W(\Hbb,\Tbb)$, and the cyclic group $\langle \gamma\rangle$ is Zariski dense in $\Tbb$.
 \end{enumerate}
 Finally, using Zariski density of $\widehat{S}'$ and the fact that hyper-regular condition is Zariski-open, we can find $\widehat{\gamma}\in \widehat{S}'$ such that $\widehat{\gamma}$ is hyper-regular. The regularity and hyper-regularity of $\widehat{\gamma}$ imply that there is a finite index free abelian subgroup $\Sigma$ in $Z_{H}(\widehat{\gamma})^{0}\cap \widehat{\Gamma}'$ such that $$\Zbb^{r}\simeq \Sigma \subset \widehat{\Gamma}'\cap Z_{\Hbb}(\widehat{\gamma})^{0}$$ where $r=\textrm{rank}_{\Rbb}(\Hbb)$.
Since $\widehat{\gamma}\in S'\subset S$, we also can get $\widehat{\gamma}\in \Sigma$ such that $\widehat{\gamma}$ is an Anosov element. These arguments present that there is an Anosov element that has large centralizer. 

Now it is enough to prove that there is no rank one factor for $\rho\restrict{\Sigma}$. Note that the arguments in \cite{BRHW}, using Weyl group action, used the fact that all $\Rbb$-simple factor of $\Hbb$ is either anistropic or of $\textrm{rank}_{\Rbb}$ is $2$ or higher. We will give an arguments that can be applied in our setting. For simplifying notations, we rewrite our goal as follows. 

\begin{prop}\label{rankone}
Let $\Hbb$ be the algebraic group defined over $\Qbb$. Assume that $\Hbb$ is $\Qbb$-semisimple and $\Rbb$-rank of every $\Qbb$-simple factor is $2$ or higher. Let $\Gamma\subset \Hbb(\Zbb)\cap H$ be an arithmetic lattice. 
Suppose, $\rho:\Gamma\rightarrow \Aut(N/\Lambda)$ is a $\Gamma$ action given by automorphisms on a compact nilmanifold $N/\Lambda$. Assume that $D\rho:\Gamma\rightarrow \Aut(\nlie)$ extends to $\Qbb$-representation $D\rho:\Hbb\rightarrow \GL(\nlie)$. 
Let $\gamma \in \Gamma$ be a regular, $\Rbb$-regular and hyper-regular element. Then there is a finite index free abelian $\Sigma\simeq\Zbb^{r}\subset Z_{\Hbb}(\gamma)^{0}\cap\Gamma$ such that $\rho\restrict{\Sigma}$ does not have rank one factor.
\end{prop}

We will still follow the strategy and notations in \cite{BRHW}. Recall the notations and arguments in there.

Assume that $\rho_{1}:\Sigma \rightarrow \Aut(M_{1})$ be a rank one factor of $\rho\restrict{\Sigma}$. Let $M_{1}=N_{1}/\Lambda_{1}$. Replacing to finite index subgroup, we may assume that $\rho_{1}(\Sigma)$ is cyclic group. 
As before, $\Tbb=Z_{\Hbb}(\gamma)^{0}$ is $\Qbb$-maximal torus and containing maximal $\Rbb$-split torus. Let $\Tbb_{s}$ be a maximal $\Rbb$-split torus in $\Tbb$. Then $\Tbb=\Tbb_{s}\cdot \Tbb_{an}$ where $\Tbb_{an}$ is $\Rbb$-anisotropic torus. Note that Zariski density of cyclic group generated by $\gamma$ gives us to say that $\rho\restrict{\Tbb}$ projects to $\Qbb$-representation $D\rho_{1}$ of the $\Qbb$-torus $\Tbb$ on $\nlie_{1}\otimes \Cbb$ where $\nlie_{1}$ is the Lie algebra of $N_{1}$ as explained in \cite{BRHW}.

We will denote character space as $\Xbb^{*}$. Note that the restriction map $\Xbb(\Tbb)\rightarrow \Xbb(\Tbb_{s})$ is a group homomorphism. 
After replacing $\Sigma$ into finite index subgroup if necessarily, we may have unique decomposition of $\sigma\in\Sigma$ into $\sigma=\sigma_{s}\sigma_{a}$ where $\sigma_{s}\in\Tbb_{s}(\Rbb)^{0}$ and $\sigma_{a}\in\Tbb_{an}(\Rbb)$. We also can get any non-trivial $\sigma\in\Sigma$, $\sigma_{s}$ is non-trivial as it is proved in \cite{BRHW}.

Fix a basis $\sigma_{1},\dots,\sigma_{r}$ in $\Sigma$ and write $(\sigma_{i})$ as $(e^{\theta_{i1}},\dots, e^{\theta_{ir}})$ in $\Tbb(\Rbb)^{0}\simeq\Rbb^{r}_{>0}$. We can define group morphism $$\Lcal:\Xbb^{*}(\Tbb_{s})\rightarrow \Rbb^{r}$$  $$\Lcal(\chi)=\left(\ln|\chi_{s}(\sigma_{1})|,\dots, \ln|\chi_{s}(\sigma_{r})|\right)$$ where $\sigma_{s}$ is $\Tbb_{s}$ part of $\sigma$.

Denote $\Lambda\subset \Xbb^{*}(\Tbb)$ be the $\Qbb$-weights of $\Qbb$-representation $D\rho_{1}$ of $\Tbb$. Note that we have Galois action on $\Xbb^{*}(\Tbb)$ as for $q\in\textrm{Gal}(\overline{\Qbb}/\Qbb)$, $$(q.\chi)(t)=q^{-1}(\chi(q(t)))$$ for any character $\chi\in\Xbb^{*}(\Tbb)$. Then $\Lambda$ is invariant under Galois action since $D\rho_{1}$ is $\Qbb$-representation. Note that for $\gamma$, $D\rho_{1}(\gamma)$ is hyperbolic since $\gamma \in S'\subset S$.

If we define $\widetilde{\chi}$ for any $\chi\in\Xbb^{*}(\Tbb)$ as $\widetilde{\chi}(t)=\chi(t)\overline{\chi(\overline{t})}$, then it is defined over $\Rbb$ and its restriction to $\Tbb_{a}$ is trivial. Furthermore, for any $\lambda\in\Lambda$, we have $\widetilde{\lambda}(\gamma)=|\lambda(\gamma)|^{2}\neq 1$ due to hyperbolicity.  Therefore, $\widetilde{\lambda}((\gamma^{2})_{s})=\widetilde{\lambda}(\gamma^{2})\neq 1$ so that $\widetilde{\lambda}\restrict{\Tbb_{s}}$ is not trivial. Let $\widetilde{\Lambda}_{s}=\{\widetilde{\lambda}\restrict{\Tbb_{s}}:\lambda\in\Lambda\}.$ Then it is non-empty and $W(\Hbb,\Tbb_{s})$ invariant.

As calculated in \cite{BRHW}, we can deduce that $\Lcal(\widetilde{\Lambda}_{s})$ belongs to a fixed one dimensional subspace for every $\lambda$ since we assume that $\rho_{1}$ is a rank one factor.

Therefore, if we can prove that dimension of the subspace spanned by $\Lcal(\widetilde{\Lambda}_{s})$ is bigger than $1$ then we proved proposition \ref{rankone}. Note that, in \cite{BRHW}, they use higher rank assumption on each $\Rbb$-simple factor for this purpose. 

\begin{proof}[Proof of Proposition \ref{rankone}] As discussed before, it is enough to prove that the dimension of the subspace spanned by 
$\Lcal(\widetilde{\Lambda}_{s})$ is bigger than $1$. Let's denote the decomposition of $\Hbb$ into $\Qbb$-simple factor $\Hbb=\Hbb^{(1)}\cdots\Hbb^{(r)}$ and we can decompose each $\Qbb$-simple factor $\Hbb^{(i)}$ into absolutely almost simple groups: $\Hbb^{(i)}=\Hbb^{(i)}_{1}\cdots\Hbb^{(i)}_{r_{i}}$.

For each $\Qbb$-simple factor $\Hbb^{(i)}$ in $\Hbb$, define $\Tbb^{(i)}=\Hbb^{(i)}\cap \Tbb$. Denote $\Phi(\Hbb^{(i)},\Tbb^{(i)})$ for roots on $\Hbb^{(i)}$  with respect to $\Tbb^{(i)}$. Let $V_{j}^{(i)}$ be the subspace spanned by $\Phi(\Hbb^{(i)}_{j},\Tbb^{(i)}_{j})$ of the $V=\Xbb^{*}(\Tbb)\otimes \Qbb$. Denote $V^{(i)}$ be the direct sum of $V_{j}^{(i)}$ for $j=1,\dots, r_{i}$; $V^{(i)}=\bigoplus_{j=1}^{r_{i}}V_{j}^{(i)}$. Note that $V=\bigoplus_{i=1}^{r}\bigoplus_{j=1}^{r_{i}}V_{j}^{(i)}$.

Recall that the torus $\Tbb$, comes from the results in \cite{PR01} and \cite{PR03}, is quasi-irreducible.  Equivalently, for any $\textrm{Gal}(\overline{\Qbb}/\Qbb)$-invariant subspace of the $V$ is direct sum of some $V^{(i)}$'s. 

The $\Qbb$-weight space $\Lambda$ is $\textrm{Gal}(\overline{\Qbb}/\Qbb)$-invariant. Let $Y$ be the subspace of $V$ that is a $\Qbb$-span of $\Lambda$. As discussed before, $Y$ is a direct sum of some $V^{(i)}$'s. Rearranging numbering, let $Y=V^{(1)}\oplus\dots\oplus V^{(k)}$. As $\Lambda$ is non-trivial space, $k\ge 1$.

Let $\Psi:\Xbb^{*}(\Tbb)\rightarrow \Xbb^{*}(\Tbb_{s})$ be the map defined by restriction to $\Tbb_{s}$. As $\Xbb^{*}(\Tbb_{s})$ contains roots $\Phi(\Hbb,\Tbb_{s})$ and $\Psi(\Xbb^{*}(\Tbb))$ is finite index subgroup in $\Xbb^{*}(\Tbb_{s})$, $\Lcal(\Psi(\Lambda))$ can not belongs to one dimensional subspace of $\Rbb^{r}$ since we assume that $\Hbb^{(1)}$ has $\Rbb$-rank at least $2$. Note that $\dim_{\Qbb}\Xbb^{*}(\Tbb^{(1)})\otimes \Qbb=r\ge 2$ and $\Lcal$ embeds $\Xbb^{*}(\Tbb)$ in $\Rbb^{r}$ as a lattice.

On the other hand, if $\lambda\in\Lambda$ then we compare $\Lcal(\Psi(\lambda))$ and $\Lcal(\widetilde{\Lambda}_{s})$.
\begin{align*}
\Lcal(\Psi(\lambda)) & =\Lcal(\lambda\restrict{{\Tbb}_{s}})\\
&=( \ln|\lambda((\sigma_{1})_{s})|,\dots, \ln|\lambda((\sigma_{r})_{s})| ),\quad \textrm{and}
\end{align*}
 \begin{align*}
\Lcal(\widetilde{\lambda}\restrict{\Tbb_{s}})&=(\ln|\widetilde{\lambda}\restrict{\Tbb_{s}}((\sigma_{1})_{s})|,\dots,\ln|\widetilde{\lambda}\restrict{\Tbb_{s}}((\sigma_{r})_{s})|)\\
&=( \ln|(\lambda\restrict{\Tbb_{s}})^{2}((\sigma_{1})_{s})|,\dots, \ln|(\lambda\restrict{\Tbb_{s}})^{2}((\sigma_{r})_{s})| )\\
&=2(\ln|\lambda((\sigma_{1})_{s})|,\dots, \ln|\lambda((\sigma_{r})_{s})| ).
\end{align*}
This implies that $\Rbb$-span of $\Lcal(\Psi(\Lambda))$ is the same subspace of $\Rbb$-span of $\Lcal(\widetilde{\Lambda}_{s})$. Especially, the ($\Rbb$-) dimension of the subspace spanned by $\Lcal(\widetilde{\Lambda}_{s})$ is bigger than $1$. This gives desired contradiction. \end{proof}

In summary, we proved following. 
\begin{prop}\label{prop8.3}  The proposition 8.3 in \cite{BRHW} is still true in our case. That is, there is a free abelian subgroup $\Sigma\subset \Gamma_{0}$ such that $\rho\restrict{\Sigma}$ does not have rank one factor actions and there is a $\gamma_{1}\in\Sigma$ such that $\alpha(\gamma_{1})$ is an Anosov diffeomorphism.
\end{prop}
After that we appeal to use the theorem in \cite{RHW}. Note that the linearization of $\alpha\restrict{\Sigma}$ in the \cite{RHW} is same as $\rho\restrict{\Sigma}$.  
\begin{theorem}[\cite{RHW}] In the above notations and settings, the action $\alpha\restrict{\Sigma}$ is conjugate to $\rho\restrict{\Sigma}$ by a $C^{\infty}$ diffeomorphism that is homotopic to identity. 
\end{theorem}
Since the uniquness of conjugacy for the Anosov element $\gamma_{1}\in\Sigma$, we can prove $h$ is $C^{\infty}$ diffeomorphism. (Or one may be able to use the arguments in the proof of global rigidity in the \cite{MQ}.) This proves the Theorem \ref{Cinf}.

\section{Other Applications}\label{sec:other}

In this section, we will collect a few direct consequences from dynamical super-rigidity \ref{thm:superrigidG} and \ref{thm:superrigidGamma}. There will be other examples of direct applications of dynamical super-rigidity. We will see a few of them.

For this section, $G$ always will be the connected real semisimple higher rank algebraic Lie group without compact factor.  Furthermore, $\Gamma$ will always be a weakly irreducible lattice in $G$. Let's denote $n(G)$ be the minimum dimension of a non-trivial rational representation of algebraic universal cover $\Gbb$ of $G$. As before, we may assume that $G$ and $\Gamma$ satisfy \ref{std} with $\textrm{rank}_{\Rbb}(G)\ge 2$.

First of all, we will think about (smooth) vector bundle over compact manifold and group action on the bundle as vector bundle automorphism. Roughly, Lyapunov exponents has an algebraic origin. Note that following corollary use only the Part 1 in Theorem \ref{thm:superrigidG}.

\begin{coro}\label{bundle} Let $D=G$ or $\Gamma$. 
Let $M$ be a compact manifold and $E$ be a (smooth) vector bundle over $M$.  Let $d$ be a dimension of the fiber. 
Suppose $D$ acts on $E$ as bundle automorphisms. Assume that there is Borel probability measure $\mu$ on $M$ such that $D$ action on $(M,\mu)$ satisfies \ref{std2}. Let $\Gbb$ be algebraic universal covering of $G$.  
Then there is rational homomorphism $\pi:\Gbb\rightarrow \GL(d,\Cbb)$ such that for any $g\in G$, $\Sp(\pi(\tilde{g}))=\Sp_{\mu}(\rho(g))$ where $\tilde{g}\in \Gbb(\Rbb)$ is any lift of $g\in G$ on $\Gbb(\Rbb)$.
\end{coro}

Recall that, when $D=\Gamma$, the assumption that induced action is irreducible if we assume that $\Gamma$ action is mixing.

Following corollary is immediate from the corollary \ref{bundle} for the derivative cocycle and Pesin's theorem.
 \begin{coro}\label{entropy}  Let $D=G$ or $\Gamma$ and $M$ be a compact manifold. Assume that $D$ acts on $M$ as $C^{2}$ diffeomorphisms. Denote action map as $\rho:D\rightarrow \Diff^{2}(M).$  Assume that there is a smooth probability measure $\mu$ on $M$ such that $D$ action on $(M,\mu)$ satisfies \ref{std2}. Denote the dimension of $M$ as $d$ and algebraic universal cover of $G$ as $\Gbb$. 
 
 Then there is a rational homomorphism $\pi:\Gbb\rightarrow \GL(d,\Cbb)$ such that  for any unbounded $g\in G_{0}$, $$h_{\mu}(g)=\sum_{|\lambda_{i}|>1}  \dim(\Cbb^{d}_{\lambda_{i}})\ln|\lambda_{i}|.$$ Here, sum runs over the eigenvalues $\lambda_{i}$ of $\pi(\tilde{g})$ for any fixed lift $\tilde{g}\in \Gbb(\Rbb)$ of $g\in D$ such that $|\lambda_{i}|>1$. Furthermore, denote eigenspace of eigenvalue $\lambda_{i}$ in $\Cbb^{d}$ as $\Cbb_{\lambda_{i}}^{d}$.  Especially, if $\dim(M)<n(G)$ then $h_{\mu}(g)=0$ for any  $g\in D$.
\end{coro}

Finally, we also can generalize theorems about local rigidity of entropy in \cite{QZ}.

\begin{coro}
Let $D=G$ or $\Gamma$. Let $\rho_{0}$ be a $C^{1}$-volume preserving action of $D$ on smooth compact manifold $M$ with point Mather spectrum. Let $\rho$ be a $C^{1}$-volume preserving action. If $\rho$ is sufficiently $C^{1}$-close to $\rho_{0}$ then the set of Lyapunov exponents of $\rho(g)$ is the same of $\rho_{0}(g)$ for all $g\in D$. Especially, if we further assume that $\rho$ and $\rho_{0}$ be a $C^{1+\epsilon}$ action then the entropy with respect to volume measure are all same, i.e. $h_{\Vol}(\rho(g))=h_{\Vol}(\rho_{0}(g))$ for all $g\in D$.  

\end{coro}
The proof is exactly same as in \cite{QZ}. As in there, we can use \cite{Pesin} in order to deduce local rigidity of Lyapunov exponents and entropy. Indeed, the theorem \ref{locrigidhom} that is about local rigidity of dynamical cocycle super-rigidity homomorphism is proved similar way in the proof of the above corollary.

\end{document}